\documentclass[11pt]{amsart}

\usepackage{pstricks, pst-node,pst-plot, pst-tree,pstricks-add,pst-eucl}
\usepackage{multido}
\usepackage{verse}
\usepackage{pst-text}
\usepackage{pst-bar}
\usepackage{amssymb}
\usepackage{pst-plot}
\usepackage{comment}
\usepackage{graphicx}
\usepackage{csquotes}
\usepackage{palatino}
\usepackage{accents}
\usepackage{mathrsfs}
\usepackage{mathtools}

\usepackage[margin=1.0in]{geometry}
\usepackage{hyperref}
\hypersetup{
    colorlinks=true,
    linkcolor=blue,
    filecolor=magenta,
    urlcolor=cyan,
}
 \newtheorem{thm}{Theorem}[section]
 
 \newtheorem{lem}[thm]{Lemma}
 \newtheorem{prop}[thm]{Proposition}
 \theoremstyle{definition}
 \newtheorem{defn}[thm]{Definition}
 \theoremstyle{remark}
 \newtheorem{rem}[thm]{Remark}
\newtheorem{ques}[thm]{Question}
 \newtheorem{exm}{Example}
 \numberwithin{equation}{section}

\newrgbcolor{darkgreen}{0.0 0.39215686 0}
\newrgbcolor{LemonChiffon}{1.0 0.98 0.8}
\newrgbcolor{magenta}{1.0 0 1.0}
\newrgbcolor{mistyrose}{1.0 0.89 0.88}
\newrgbcolor{deeppink}{1.0 0.08 0.58}
\newrgbcolor{lightsalmon}{1.0 0.63 0.48}
\newrgbcolor{lightred}{0.7 0.0 0.0}
\newrgbcolor{lightblue}{0.68 0.85 0.90}
\newrgbcolor{indianred}{0.80  0.36  0.36}
\newrgbcolor{lightgreen}{0.56 0.93 0.56}
\newrgbcolor{darkred}{0.89928057554   0   0}
\newrgbcolor{orange}{1   .5   0}

\makeatletter
\let\save@mathaccent\mathaccent
\newcommand*\if@single[3]{%
  \setbox0\hbox{${\mathaccent"0362{#1}}^H$}%
  \setbox2\hbox{${\mathaccent"0362{\kern0pt#1}}^H$}%
  \ifdim\ht0=\ht2 #3\else #2\fi
  }
\newcommand*\rel@kern[1]{\kern#1\dimexpr\macc@kerna}
\newcommand*\widebar[1]{\@ifnextchar^{{\wide@bar{#1}{0}}}{\wide@bar{#1}{1}}}
\newcommand*\wide@bar[2]{\if@single{#1}{\wide@bar@{#1}{#2}{1}}{\wide@bar@{#1}{#2}{2}}}
\newcommand*\wide@bar@[3]{%
  \begingroup
  \def\mathaccent##1##2{%
    \let\mathaccent\save@mathaccent
    \if#32 \let\macc@nucleus\first@char \fi
    \setbox\z@\hbox{$\macc@style{\macc@nucleus}_{}$}%
    \setbox\tw@\hbox{$\macc@style{\macc@nucleus}{}_{}$}%
    \dimen@\wd\tw@
    \advance\dimen@-\wd\z@
    \divide\dimen@ 3
    \@tempdima\wd\tw@
    \advance\@tempdima-\scriptspace
    \divide\@tempdima 10
    \advance\dimen@-\@tempdima
    \ifdim\dimen@>\z@ \dimen@0pt\fi
    \rel@kern{0.6}\kern-\dimen@
    \if#31
      \overline{\rel@kern{-0.6}\kern\dimen@\macc@nucleus\rel@kern{0.4}\kern\dimen@}%
      \advance\dimen@0.4\dimexpr\macc@kerna
      \let\final@kern#2%
      \ifdim\dimen@<\z@ \let\final@kern1\fi
      \if\final@kern1 \kern-\dimen@\fi
    \else
      \overline{\rel@kern{-0.6}\kern\dimen@#1}%
    \fi
  }%
  \macc@depth\@ne
  \let\math@bgroup\@empty \let\math@egroup\macc@set@skewchar
  \mathsurround\z@ \frozen@everymath{\mathgroup\macc@group\relax}%
  \macc@set@skewchar\relax
  \let\mathaccentV\macc@nested@a
  \if#31
    \macc@nested@a\relax111{#1}%
  \else
    \def\gobble@till@marker##1\endmarker{}%
    \futurelet\first@char\gobble@till@marker#1\endmarker
    \ifcat\noexpand\first@char A\else
      \def\first@char{}%
    \fi
    \macc@nested@a\relax111{\first@char}%
  \fi
  \endgroup
}
\makeatother

 \def\co{\colon\thinspace}
 
\begin{document}

\title[Duality Preserving Bijections]{Non-crossing trees, quadrangular dissections, ternary trees, and duality preserving bijections}

\author[Nikos Apostolakis]{Nikos Apostolakis}

\address{%
  Department of Mathematics \& Computer Science\\
  Bronx Community College\\
  The City University of New York}

\email{nikolaos.apostolakis@bcc.cuny.edu}

\thanks{ The Computer Algebra Systems Sage~\cite{sagemath}, and
  GAP~\cite{GAP4} were used extensively to confirm calculations and
  check conjectures at several stages of this project.  I would like
  to thank Cormac O'Sullivan for valuable comments.  Finally I would
  also like to thank the referee of an earlier version of this paper
  for invaluable comments and for suggesting interesting connections
  with the literature.}

\date{\today}

\begin{abstract}
  Using the theory of Properly Embedded Graphs developed in an earlier
  work we define an involutory duality on the set labeled non-crossing
  trees that lifts the obvious duality in the set of unlabeled
  non-crossing trees.  The set of non-crossing trees is a free ternary
  magma with one generator and this duality is an instance of a
  duality that is defined in any such magma.  Any two free ternary
  magmas with one generator are isomorphic via a unique isomorphism
  that we call the structural bijection.  Besides the set of
  non-crossing trees we also consider as free ternary magmas with one
  generator the set of ternary trees, the set of quadrangular
  dissections, and the set of flagged Perfectly Chain Decomposed
  Ditrees, and we give topological and/or combinatorial
  interpretations of the structural bijections between them.  In
  particular the bijection from the set of quadrangular dissections to
  the set of non-crossing trees seems to be new. Further we give
  explicit formulas for the number of self-dual labeled and unlabeled
  non-crossing trees and the set of quadrangular dissections up to
  rotations and up to rotations and reflections.
 \end{abstract}

\maketitle

\section{Introduction}
\label{sec:intro}

This paper follows~\cite{Apos2018arXivApril} as the second in a
planned series that explore the theory and applications of
\emph{Properly Embedded Graphs} (\emph{pegs}) and their duality.  The
main motivation is to understand duality of non-crossing trees and in
particular to enumerate the set of self-dual objects.  This leads to a
more broad investigation of duality in Fuss-Catalan objects for
$p=3$.\footnote{See for example~\cite{HiltonPedersen1991} for the
  basic definitions of Fuss-Catalan numbers, called \emph{generalized
    Catalan numbers} there.}

Non-crossing trees are well studied in the literature, see for
example~\cite{DulPen1993}, and~\cite{Noy1998301}.  A
\emph{non-crossing tree} is a tree with vertices on a circle,
typically at the vertices of a regular polygon, and edges mutually
non-intersecting chords.  Usually the vertices are labeled
$1,\ldots,n $, where $n$ is the number of vertices, and if this is not
the case we will talk of \emph{unlabeled} non-crossing trees.  There
is a topologically obvious way to define the dual of an unlabeled
non-crossing tree $t$: removing the tree breaks the circle into arcs
and the interior of the circle into simply connected regions, and
there is exactly one arc in the boundary of each region.  The dual
$t^{*}$ is defined by putting a vertex in each arc and connecting two
of these vertices by an edge if and only if the corresponding regions
share an edge.  Clearly $t^{*}$ is non-crossing and
$\left( t^{*} \right)^{*} = t$; for an example of this construction
see Figure~\ref{fig:exncunl}.  We call this duality \emph{nc-duality}.

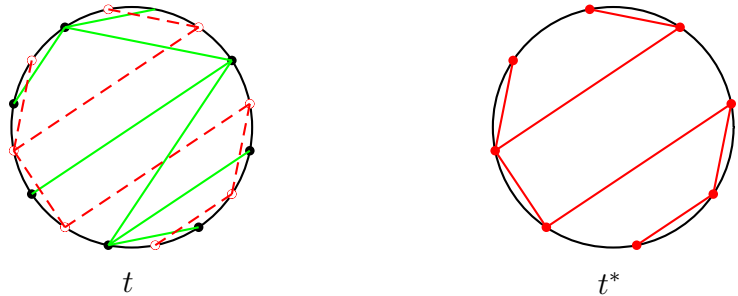
\begin{figure}[ht]
  \centering
  \psset{unit=1.6}
  \begin{pspicture}(-3,-1.3)(3,1.3)
    \rput(-2,0){
      \begin{pspicture}(-1.28079, -1.28079)(1.28079, 1.28079)
    \pscircle(0,0){1}
    \pnode(-0.19509,-0.98079){1}
    \psdot(-0.19509,-0.98079)
    \pnode(0.19509,-0.98079){1d}
    \psdot[linecolor=red,dotstyle=o](0.19509,-0.98079)
    \pnode(0.55557,-0.83147){2}
    \psdot(0.55557,-0.83147)
    \pnode(0.83147,-0.55557){2d}
    \psdot[linecolor=red,dotstyle=o](0.83147,-0.55557)
    \pnode(0.98079,-0.19509){3}
    \psdot(0.98079,-0.19509)
    \pnode(0.98079,0.19509){3d}
    \psdot[linecolor=red,dotstyle=o](0.98079,0.19509)
    \pnode(0.83147,0.55557){4}
    \psdot(0.83147,0.55557)
    \pnode(0.55557,0.83147){4d}
    \psdot[linecolor=red,dotstyle=o](0.55557,0.83147)
    \pnode(0.19509,0.98079){5}
    \pnode(-0.19509,0.98079){5d}
    \psdot[linecolor=red,dotstyle=o](-0.19509,0.98079)
    \pnode(-0.55557,0.83147){6}
    \psdot(-0.55557,0.83147)
    \pnode(-0.83147,0.55557){6d}
    \psdot[linecolor=red,dotstyle=o](-0.83147,0.55557)
    \pnode(-0.98079,0.19509){7}
    \psdot(-0.98079,0.19509)
    \pnode(-0.98079,-0.19509){7d}
    \psdot[linecolor=red,dotstyle=o](-0.98079,-0.19509)
    \pnode(-0.83147,-0.55557){8}
    \psdot(-0.83147,-0.55557)
    \pnode(-0.55557,-0.83147){8d}
    \psdot[linecolor=red,dotstyle=o](-0.55557,-0.83147)
    \ncline[linecolor=green]{1}{2}
    \ncline[linecolor=green]{1}{3}
    \ncline[linecolor=green]{1}{4}
    \ncline[linecolor=green]{4}{6}
    \ncline[linecolor=green]{4}{8}
    \ncline[linecolor=green]{5}{6}
    \ncline[linecolor=green]{6}{7}
    \ncline[linecolor=red, linestyle = dashed]{1d}{2d}
    \ncline[linecolor=red, linestyle = dashed]{2d}{3d}
    \ncline[linecolor=red, linestyle = dashed]{3d}{8d}
    \ncline[linecolor=red, linestyle = dashed]{4d}{5d}
    \ncline[linecolor=red, linestyle = dashed]{4d}{7d}
    \ncline[linecolor=red, linestyle = dashed]{6d}{7d}
    \ncline[linecolor=red, linestyle = dashed]{7d}{8d}      
  \end{pspicture}}
\rput(2,0){
  \begin{pspicture}(-1.28079, -1.28079)(1.28079, 1.28079)
    \pscircle(0,0){1}
    \pnode(0.19509,-0.98079){1d}
    \psdot[linecolor=red](0.19509,-0.98079)
    \pnode(0.83147,-0.55557){2d}
    \psdot[linecolor=red](0.83147,-0.55557)
    \pnode(0.98079,0.19509){3d}
    \psdot[linecolor=red](0.98079,0.19509)
    \pnode(0.55557,0.83147){4d}
    \psdot[linecolor=red](0.55557,0.83147)
    \pnode(-0.19509,0.98079){5d}
    \psdot[linecolor=red](-0.19509,0.98079)
    \pnode(-0.83147,0.55557){6d}
    \psdot[linecolor=red](-0.83147,0.55557)
    \pnode(-0.98079,-0.19509){7d}
    \psdot[linecolor=red](-0.98079,-0.19509)
    \pnode(-0.55557,-0.83147){8d}
    \psdot[linecolor=red](-0.55557,-0.83147)
    \ncline[linecolor=red]{1d}{2d}
    \ncline[linecolor=red]{2d}{3d}
    \ncline[linecolor=red]{3d}{8d}
    \ncline[linecolor=red]{4d}{5d}
    \ncline[linecolor=red]{4d}{7d}
    \ncline[linecolor=red]{6d}{7d}
    \ncline[linecolor=red]{7d}{8d}      
  \end{pspicture}}
\uput[-90](-2,-1.1){\large $t$}
\uput[-90](2,-1.1){\large $t^{*}$}
  \end{pspicture}
  \caption{An unlabeled non-crossing tree and its dual.}
  \label{fig:exncunl}
\end{figure}

Lifting nc-duality to an \emph{involutory} duality at the level of
labeled non-crossing trees is not straightforward, and it involves
clarifying some subtle issues that have to do with the orientation of
the circle.  For example a ``duality'' for labeled non-crossing trees
was defined in~\cite{Hernando1999} by labeling the dual vertex that
follows $i$ in the standard (counterclockwise) orientation of the
circle by $i$, as in Figure~\ref{fig:compl}.  Clearly this operation
is not involutory, rather it has order $2n$ where $n$ is the number of
vertices of the tree. For this reason we call the resulting tree the
\emph{complement}, rather than the \emph{dual}, of $t$ and denote it
by $\kappa(t)$, since as we will see in Section~\ref{sec:kreweras} it
is induced by the \emph{Kreweras complement} in the lattice of
non-crossing partitions.

\begin{figure}[htp]
  \centering
  \psset{unit=1.4}
  \begin{pspicture}(-3.4,-2)(3.6,1.5)
    \rput(-2,0){
      \begin{pspicture}(-1.28079, -1.28079)(1.28079, 1.28079)
    \pscircle(0,0){1}
    \pnode(-0.19509,-0.98079){1}
    \psdot(-0.19509,-0.98079)
    \uput[-90](-0.19509,-0.98079){\small $1$}
    \pnode(0.19509,-0.98079){1d}
    \psdot[linecolor=red,dotstyle=o](0.19509,-0.98079)
    \pnode(0.55557,-0.83147){2}
    \psdot(0.55557,-0.83147)
     \uput[-20](0.55557,-0.83147){\small $2$}
    \pnode(0.83147,-0.55557){2d}
    \psdot[linecolor=red,dotstyle=o](0.83147,-0.55557)
    \pnode(0.98079,-0.19509){3}
    \psdot(0.98079,-0.19509)
    \uput[0](0.98079,-0.19509){\small $3$}
    \pnode(0.98079,0.19509){3d}
    \psdot[linecolor=red,dotstyle=o](0.98079,0.19509)
    \pnode(0.83147,0.55557){4}
    \psdot(0.83147,0.55557)
    \uput[20](0.83147,0.55557){\small $4$}
    \pnode(0.55557,0.83147){4d}
    \psdot[linecolor=red,dotstyle=o](0.55557,0.83147)
    \pnode(0.19509,0.98079){5}
    \uput[90](0.19509,0.98079){\small $5$}
    \pnode(-0.19509,0.98079){5d}
    \psdot[linecolor=red,dotstyle=o](-0.19509,0.98079)
    \pnode(-0.55557,0.83147){6}
    \psdot(-0.55557,0.83147)
    \uput[110](-0.55557,0.83147){\small $6$}
    \pnode(-0.83147,0.55557){6d}
    \psdot[linecolor=red,dotstyle=o](-0.83147,0.55557)
    \pnode(-0.98079,0.19509){7}
    \psdot(-0.98079,0.19509)
    \uput[180](-0.98079,0.19509){\small $7$}
    \pnode(-0.98079,-0.19509){7d}
    \psdot[linecolor=red,dotstyle=o](-0.98079,-0.19509)
    \pnode(-0.83147,-0.55557){8}
    \psdot(-0.83147,-0.55557)
    \uput[200](-0.83147,-0.55557){\small $8$}
    \pnode(-0.55557,-0.83147){8d}
    \psdot[linecolor=red,dotstyle=o](-0.55557,-0.83147)
    \ncline[linecolor=green]{1}{2}
    \ncline[linecolor=green]{1}{3}
    \ncline[linecolor=green]{1}{4}
    \ncline[linecolor=green]{4}{6}
    \ncline[linecolor=green]{4}{8}
    \ncline[linecolor=green]{5}{6}
    \ncline[linecolor=green]{6}{7}
    \ncline[linecolor=red, linestyle = dashed]{1d}{2d}
    \ncline[linecolor=red, linestyle = dashed]{2d}{3d}
    \ncline[linecolor=red, linestyle = dashed]{3d}{8d}
    \ncline[linecolor=red, linestyle = dashed]{4d}{5d}
    \ncline[linecolor=red, linestyle = dashed]{4d}{7d}
    \ncline[linecolor=red, linestyle = dashed]{6d}{7d}
    \ncline[linecolor=red, linestyle = dashed]{7d}{8d}      
  \end{pspicture}}
\rput(2,0){
  \begin{pspicture}(-1.28079, -1.28079)(1.28079, 1.28079)
    \pscircle(0,0){1}
    \pnode(0.19509,-0.98079){1d}
    \psdot[linecolor=red](0.19509,-0.98079)
    \uput[-90](0.19509,-0.98079){\red \small $1^{*}$}
    \pnode(0.83147,-0.55557){2d}
    \psdot[linecolor=red](0.83147,-0.55557)
    \uput[-20](0.83147,-0.55557){\red \small $2^{*}$}
    \pnode(0.98079,0.19509){3d}
    \psdot[linecolor=red](0.98079,0.19509)
    \uput[0](0.98079,0.19509){\red \small $3^{*}$}
    \pnode(0.55557,0.83147){4d}
    \psdot[linecolor=red](0.55557,0.83147)
     \uput[20](0.55557,0.83147){\red \small $4^{*}$}
    \pnode(-0.19509,0.98079){5d}
    \psdot[linecolor=red](-0.19509,0.98079)
     \uput[90](-0.19509,0.98079){\red \small $5^{*}$}
    \pnode(-0.83147,0.55557){6d}
    \psdot[linecolor=red](-0.83147,0.55557)
    \uput[110](-0.83147,0.55557){\red \small $6^{*}$}
    \pnode(-0.98079,-0.19509){7d}
    \psdot[linecolor=red](-0.98079,-0.19509)
    \uput[180](-0.98079,-0.19509){\red \small $7^{*}$}
    \pnode(-0.55557,-0.83147){8d}
    \psdot[linecolor=red](-0.55557,-0.83147)
    \uput[200](-0.55557,-0.83147){\red \small $8^{*}$}
    \ncline[linecolor=red]{1d}{2d}
    \ncline[linecolor=red]{2d}{3d}
    \ncline[linecolor=red]{3d}{8d}
    \ncline[linecolor=red]{4d}{5d}
    \ncline[linecolor=red]{4d}{7d}
    \ncline[linecolor=red]{6d}{7d}
    \ncline[linecolor=red]{7d}{8d}      
  \end{pspicture}}
\uput[-90](-2,-1.6){\large $t$}
\uput[-90](2,-1.6){\large $\kappa(t)$}
  \end{pspicture} 
  \caption{The complement defined in~\cite{Hernando1999}.}
  \label{fig:compl}
\end{figure}
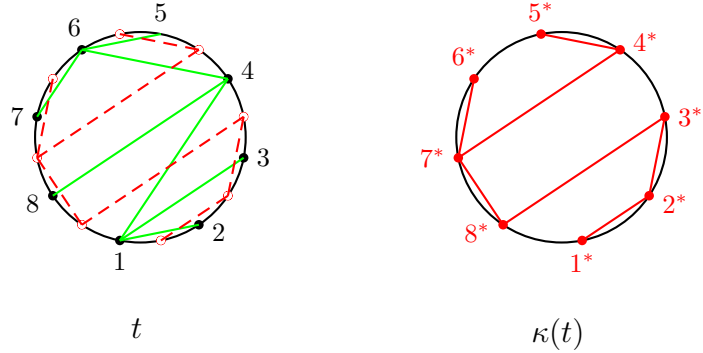

From the point of view of~\cite{Apos2018arXivApril}\footnote{All the
  relevant notions and terminology are reviewed in
  Section~\ref{sec:pegs}.}, non-crossing trees are trees
\emph{properly embedded} in the disk, and nc-duality in the unlabeled
case is simply the \emph{mind-body duality}. One could then use 
mind-body duality for labeled pegs to lift nc-duality to labeled
non-crossing trees. This approach however, has the drawback that the
mind-body dual $t^{*}$ of a peg is embedded in the oppositely oriented
surface, so that we get a duality
$\mathcal{N} \to \mathcal{N}^{\intercal}$, where $\mathcal{N}$ stands
for the set of trees pegged in the standard disk with the
counterclockwise orientation, and $\mathcal{N}^{\intercal}$ for the
set of trees pegged in the disk with the clockwise orientation.

As observed in Section~5 of~\cite{Apos2018arXivApril} this drawback
can be rectified by using mind-body duality at the level of \emph{rooted
  edge-labeled} trees.  The upshot is that one can define an involutory
duality $*\co \mathcal{N} \to \mathcal{N}$ lifting the duality of unlabeled
trees by
\begin{equation}
  \label{eq:ncdudef}
  t^{*} = s(\kappa(t))
\end{equation}
where $s\co \mathcal{N} \to \mathcal{N}$ stands for the map induced by
reflection of the circle across the diameter that passes through the
vertex labeled $1$.  We emphasize that we consider reflections and
rotations to act on the edges of non-crossing trees leaving the
vertices fixed, so $t^{*}$ is still pegged on the standard disk
endowed with the counterclockwise orientation. More concretely, the
nc-dual $t^{*}$ of a non-crossing tree $t$ is obtained by labeling the
dual vertices in a \emph{clockwise} order starting with the dual
vertex that immediately follows the vertex of $t$ labeled $1$, and
then transferring the dual tree to the counterclockwise oriented
circle.  For example see Figure~\ref{fig:ncdudef} for the dual of the
non-crossing tree of Figure~\ref{fig:compl}.

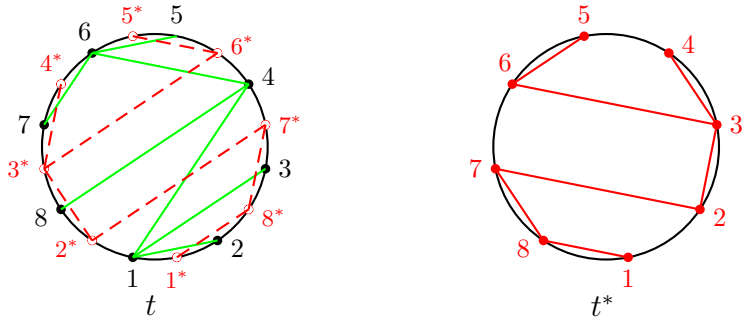
\begin{figure}[htp]
  \centering
  \psset{unit=1.5}
  \begin{pspicture}(-3.4,-1.8)(3.4,2)
    \rput(-2,0){
      \begin{pspicture}(-1.28079, -1.28079)(1.28079, 1.28079)
    \pscircle(0,0){1}
    \pnode(-0.19509,-0.98079){1}
    \psdot(-0.19509,-0.98079)
    \uput[-90](-0.19509,-0.98079){\small $1$}
    \pnode(0.19509,-0.98079){1d}
    \psdot[linecolor=red,dotstyle=o](0.19509,-0.98079)
    \pnode(0.55557,-0.83147){2}
     \uput[-90](0.19509,-0.98079){\red \footnotesize $1^{*}$}
    \psdot(0.55557,-0.83147)
     \uput[-20](0.55557,-0.83147){\small $2$}
    \pnode(0.83147,-0.55557){2d}
    \psdot[linecolor=red,dotstyle=o](0.83147,-0.55557)
    \uput[-20](0.83147,-0.55557){\red \footnotesize $8^{*}$}
    \pnode(0.98079,-0.19509){3}
    \psdot(0.98079,-0.19509)
    \uput[0](0.98079,-0.19509){\small $3$}
    \pnode(0.98079,0.19509){3d}
    \psdot[linecolor=red,dotstyle=o](0.98079,0.19509)
    \uput[0](0.98079,0.19509){\red \footnotesize $7^{*}$}
    \pnode(0.83147,0.55557){4}
    \psdot(0.83147,0.55557)
    \uput[20](0.83147,0.55557){\small $4$}
    \pnode(0.55557,0.83147){4d}
    \psdot[linecolor=red,dotstyle=o](0.55557,0.83147)
    \uput[20](0.55557,0.83147){\red \footnotesize $6^{*}$}
    \pnode(0.19509,0.98079){5}
    \uput[90](0.19509,0.98079){\small $5$}
    \pnode(-0.19509,0.98079){5d}
    \psdot[linecolor=red,dotstyle=o](-0.19509,0.98079)
    \uput[90](-0.19509,0.98079){\red \footnotesize $5^{*}$}
    \pnode(-0.55557,0.83147){6}
    \psdot(-0.55557,0.83147)
    \uput[110](-0.55557,0.83147){\small $6$}
    \pnode(-0.83147,0.55557){6d}
    \psdot[linecolor=red,dotstyle=o](-0.83147,0.55557)
    \uput[110](-0.83147,0.55557){\red \footnotesize $4^{*}$}
    \pnode(-0.98079,0.19509){7}
    \psdot(-0.98079,0.19509)
    \uput[180](-0.98079,0.19509){\small $7$}
    \pnode(-0.98079,-0.19509){7d}
    \psdot[linecolor=red,dotstyle=o](-0.98079,-0.19509)
    \uput[180](-0.98079,-0.19509){\red \footnotesize $3^{*}$}
    \pnode(-0.83147,-0.55557){8}
    \psdot(-0.83147,-0.55557)
    \uput[200](-0.83147,-0.55557){\small $8$}
    \pnode(-0.55557,-0.83147){8d}
    \psdot[linecolor=red,dotstyle=o](-0.55557,-0.83147)
    \uput[200](-0.55557,-0.83147){\red \footnotesize $2^{*}$}
    \ncline[linecolor=green]{1}{2}
    \ncline[linecolor=green]{1}{3}
    \ncline[linecolor=green]{1}{4}
    \ncline[linecolor=green]{4}{6}
    \ncline[linecolor=green]{4}{8}
    \ncline[linecolor=green]{5}{6}
    \ncline[linecolor=green]{6}{7}
    \ncline[linecolor=red, linestyle = dashed]{1d}{2d}
    \ncline[linecolor=red, linestyle = dashed]{2d}{3d}
    \ncline[linecolor=red, linestyle = dashed]{3d}{8d}
    \ncline[linecolor=red, linestyle = dashed]{4d}{5d}
    \ncline[linecolor=red, linestyle = dashed]{4d}{7d}
    \ncline[linecolor=red, linestyle = dashed]{6d}{7d}
    \ncline[linecolor=red, linestyle = dashed]{7d}{8d}      
  \end{pspicture}}
\rput(2,0){
  \begin{pspicture}(-1.28079, -1.28079)(1.28079, 1.28079)
    \pscircle(0,0){1}
    \pnode(0.19509,-0.98079){1d}
    \psdot[linecolor=red](0.19509,-0.98079)
    \uput[-90](0.19509,-0.98079){\red \small $1$}
    \pnode(0.83147,-0.55557){2d}
    \psdot[linecolor=red](0.83147,-0.55557)
    \uput[-20](0.83147,-0.55557){\red \small $2$}
    \pnode(0.98079,0.19509){3d}
    \psdot[linecolor=red](0.98079,0.19509)
    \uput[0](0.98079,0.19509){\red \small $3$}
    \pnode(0.55557,0.83147){4d}
    \psdot[linecolor=red](0.55557,0.83147)
     \uput[20](0.55557,0.83147){\red \small $4$}
    \pnode(-0.19509,0.98079){5d}
    \psdot[linecolor=red](-0.19509,0.98079)
     \uput[90](-0.19509,0.98079){\red \small $5$}
    \pnode(-0.83147,0.55557){6d}
    \psdot[linecolor=red](-0.83147,0.55557)
    \uput[110](-0.83147,0.55557){\red \small $6$}
    \pnode(-0.98079,-0.19509){7d}
    \psdot[linecolor=red](-0.98079,-0.19509)
    \uput[180](-0.98079,-0.19509){\red \small $7$}
    \pnode(-0.55557,-0.83147){8d}
    \psdot[linecolor=red](-0.55557,-0.83147)
    \uput[200](-0.55557,-0.83147){\red \small $8$}
    \ncline[linecolor=red]{1d}{8d}
    \ncline[linecolor=red]{2d}{3d}
    \ncline[linecolor=red]{3d}{6d}
    \ncline[linecolor=red]{3d}{4d}
    \ncline[linecolor=red]{5d}{6d}
    \ncline[linecolor=red]{2d}{7d}
    \ncline[linecolor=red]{7d}{8d}      
  \end{pspicture}}
\uput[-90](-2,-1.2){\large $t$}
\uput[-90](2,-1.2){\large $t^{*}$}
  \end{pspicture} 
  \caption{The dual of an non-crossing tree.}
  \label{fig:ncdudef}
\end{figure}

We note that there is a price to be paid for getting the dual of a
non-crossing tree to be pegged in the same oriented disk, namely the
natural correspondence between the edges of two dual trees is lost.
in the case of mind-body duality each edge of $t$ crosses once only
one edge of $t^{*}$ and so every edge $e$ of $t$ has a dual edge
$e^{*}$ in $t^{*}$ with a natural topological relation. No such
topologically obvious correspondence exists between the edges of two
nc-dual labeled non-crossing trees.

This duality for non-crossing trees and the necessary background material
about pegs are developed in Section~\ref{sec:pegs}.

The set of non-crossing trees
$\mathcal{N} = \bigsqcup_{m\ge 0} \mathcal{N}_m$, where
$\mathcal{N}_m$ is the set of non-crossing trees with $m$ edges is an
example of a ($p = 3$) Fuss-Catalan family. Fuss-Catalan families have
been extensively studied in the literature from various points of
view, see for example~\cite{Cigler1987}, \cite{HiltonPedersen1991},
and~\cite{PrzytyckiSikora2000}.  It turns out that nc-duality is an
instance of a duality that exists in all such families.

Inspired by~\cite{Cigler1987} we consider ($p=3$) Fuss-Catalan
families as instances of the \emph{free ternary magma with one
  generator}, that is a set with a ternary operation satisfying the
usual universal property of ``freeness'', we give the details in
Section~\ref{sec:bastern}, including a standard construction of the
ternary magma freely generated by a set $X$ as a set of words.  There
is a natural notion of \emph{rank} for elements of a free ternary
magma $M$ namely the number of occurrences of the ternary operator and
we denote by $M_m$ the set of elements of $M$ that have rank $m$.

Among the many instances of Fuss-Catalan families (or ternary magmas
freely generated by one element $\lambda$) we consider in
Section~\ref{sec:ternmagma}
\begin{itemize}
\item The ``standard'' ternary magma freely generated by one element:
  $$\mathcal{A} = \bigsqcup_{m\ge 0}\mathcal{A}_m$$ 
  where $\mathcal{A}_m$ is the set of elements of rank $m$, and is
  usually thought of as ways of parenthesizing $m$ applications of a
  ternary operation.  The basic theory of free ternary magmas is
  developed in Section~\ref{sec:bastern}.
\item The set of (full) ternary trees, where a ternary tree is an
  ordered tree with the out-degree of every vertex  $0$ or $3$:
  $$\mathcal{T} = \bigsqcup_{m\ge 0}\mathcal{T}_m.$$ 
  The rank is the number of internal vertices and the generator
  $\lambda$ is the ternary tree with one vertex and no edges.
  For details see Section~\ref{sec:terntrees}.
 \item The set of labeled non-crossing trees
   $$\mathcal{N} = \bigsqcup_{m\ge 0}\mathcal{N}_m.$$ 
   The rank is given by the number of edges and the generator $\lambda$
   is the non-crossing tree with one vertex and no edges.  The ternary
   structure of $\mathcal{N}$ is exposed in Section~\ref{sec:pmagma}.
 \item The set of flagged Perfectly Chain Decomposed Ditrees (PCDDs)
   $$\mathcal{P} = \bigsqcup_{m\ge 0}\mathcal{P}_m.$$
   This set arises from the application of the concept of medial
   digraph (developed in Section 2.2 of~\cite{Apos2018arXivApril}) in
   our case.  A \emph{ditree} is a digraph with underlying graph a
   tree, and a \emph{medial ditree} is a ditree with the in and out
   degrees of every vertex at most $2$.  A \emph{Perfectly Chain
     Decomposed Ditree} (PCDD) is a medial ditree endowed with a
   \emph{Perfect Chain Decomposition} (PCD), that is, a decomposition
   of its edges into chains with the property that each vertex belongs
   to exactly two chains.  A \emph{flagged Perfectly Chain Decomposed
     Ditree} is a PCDD with a distinguished chain called its flag.
   The rank of an element of $\mathcal{P}$ is the number of its
   vertices, and the generator is the degenerated empty PCDD.  For
   details see Section~\ref{sec:pcdd}.
\item The set of quadrangular dissections of polygons 
  $$\mathcal{Q} = \bigsqcup_{m\ge 0}\mathcal{Q}_m.$$
  By a quadrangular dissection of a (convex) polygon we mean a
  dissection of the polygon into quadrangular cells via a set of
  non-crossing diagonals. It's easy to see that only polygons with
  even number of vertices admit quadrangular dissections.

  To understand the ternary structure it is more convenient to think
  of elements of $\mathcal{Q}$ as $4$-clusters, that is $2$-complexes
  obtained by gluing quadrangular cells along edges in such a way that
  no $1$-cycles are created, see Section~\ref{sec:quandrtern} for
  details.  The rank is given by the number of $2$-cells and the generator
  $\lambda$ is the (trivial) $4$-cluster consisting of a single edge and no
  $2$-cells.
\end{itemize}

The number of rank $m$ elements of a ternary magma freely generated by one element
is given by the ($p=3$) Fuss-Catalan numbers
\begin{equation}
  \label{eq:3cat}
  \nu_m := \frac{1}{2m+1} \binom{3m}{m}.
\end{equation}
There are many proofs of this result, and in
Theorem~\ref{thm:terntuples} we generalize the proof
in~\cite{Cigler1987} to give an elementary proof for the formula
giving the number of $k$-tuples of rank $m$.  To our knowledge this is
the only elementary (without use of generating functions) proof of
that formula.

Any two ternary magmas freely generated by one element are isomorphic via a unique
isomorphism and we call any such isomorphism a \emph{structural bijection}.  
Uniqueness implies that any diagram of structural bijections commutes and in particular
Diagram~~\eqref{eq:commdiag} commutes.

\begin{equation}
  \label{eq:commdiag}
    \begin{psmatrix}[mnode=R,colsep=2.3cm,rowsep=2.3cm]
                 & \mathcal{\mathcal{Q}} & \mathcal{\mathcal{N}} \\
                 & \mathcal{\mathcal{T}} & \mathcal{\mathcal{P}} \\
  \mathcal{A} &  &
  \end{psmatrix}
  \psset{nodesep=0.3cm}
  \ncLine[arrowsize=.2]{->}{1,2}{1,3}
  \Aput{\phi}
  \ncLine[arrowsize=.2]{->}{1,2}{2,2}
  \Bput{\psi}
  \ncLine[arrowsize=.2]{->}{1,3}{2,3}
  \Aput{\mathcal{M}}
  \ncLine[arrowsize=.2]{->}{2,3}{2,2}
 \Aput{\tau}
 \ncLine[arrowsize=.2]{->}{1,3}{2,2}
 \Bput{\sigma}
 \ncLine[linestyle=dashed,arrowsize=.2]{->}{3,1}{1,2}
 \ncLine[linestyle=dashed,arrowsize=.2]{->}{3,1}{2,2}
 \ncLine[linestyle=dashed,arrowsize=.2]{->}{3,1}{2,3}
\end{equation}

Interchanging the first and third argument in any occurrence of the
ternary operator while leaving the second argument fixed defines a
duality in $\mathcal{A}$, that satisfies, and is determined by a
fundamental equation namely Equation~\ref{eq:terndual}.  This duality
is transferred via the structural bijection to a duality in any
ternary magma freely generated by one element.  In the magmas we
consider these turn out to be quite natural and/or known:

\begin{itemize}
\item In $T$ it transfers to interchanging the left and right subtree of every
  internal vertex.  This duality was considered in~\cite{DeutschFereticNoy2002}.
\item In $\mathcal{N}$ it transfers to nc-duality.
\item In $\mathcal{P}$ it transfers to ``mind-body'' duality.  A PCD
  is determined by a binary choice at every vertex: choosing which
  incoming edge to connect to which outgoing one.  Mind-body duality
  consists of making the opposite choice at every vertex.  See
  Section~2.2 of~\cite{Apos2018arXivApril} for details.
\item In $\mathcal{Q}$ it transfers to reflection across the perpendicular bisector
  of an edge.
\end{itemize}

For brevity we will refer to a ternary magma freely generated by one
element and endowed with the above duality as a \emph{free $*$-magma}.
One of the original motivations for the present work was to understand
self-duality for (labeled and unlabeled) non-crossing trees.  We
achieve that for the labeled case in Theorem~\ref{thm:sdternformula}
where we provide an explicit formula for the number of self-dual
elements with given rank of a free $*$-magma. This formula was proven
in~\cite{DeutschFereticNoy2002} in the case of ternary trees using a
generating function argument.  We prove it by giving, in
Theorem~\ref{thm:sdtern}, bijections from the set of self-dual
elements of rank $m$, to the set of elements of rank $m/2$ for $m$
even, and to the set of pairs of elements or total rank $(m-1)/2$ for
$m$ odd.  Since we have given, in Theorem~\ref{thm:terntuples}, an
elementary proof for the counting formulae of these sets, our proof of
Theorem~\ref{thm:sdternformula} is completely elementary.

The structural bijections in Diagram~\eqref{eq:commdiag} have
interesting combinatorial and/or topological interpretations which we
explain in Section~\ref{sec:strbij}.  

The interpretation of $\psi\co \mathcal{Q} \to \mathcal{T}$ was given
in~\cite{HiltonPedersen1991}.  We give an exposition of that
interpretation in Section~\ref{sec:psi}.

We give an interpretation of $\phi\co \mathcal{Q}_m \to \mathcal{N}_m$
in Section~\ref{sec:phi}. Consider a quadrangular dissection $q$. The
dissected polygon has an even number of vertices and it's easy to see
that the dissecting diagonals connect vertices with labels of opposite
parities and therefore one of the diagonals of a cell of $q$ connects
vertices with odd labels while the other one vertices with even
labels.  The non-crossing tree $\phi(q)$ is then obtained by the
``odd'' diagonals of all the cells.  This interpretation is, to our
knowledge, new.  There is however a close connection between $\phi$
and the Scaeffer bijection between rooted quadrangulations of the
sphere and well labeled trees given in~\cite{Schaeffer} (see
also~\cite{ChassaingSchaeffer2004})\footnote{Thanks to the anonymous
  referee of a previous version for pointing this out.}.  In order to
explain this connection, in Section~\ref{sec:schaeffer}, we define (in
Section~\ref{sec:bipart}) the set $\mathcal{BO}_m$ of
\emph{bipartisan trees} with $m$ edges, and exhibit a bijection
$\mathcal{N}_m \to \mathcal{BO}_m$.  A bipartisan tree is an ordered
tree where the children of each non-root vertex are divided into two
sets the \emph{left children} and the \emph{right children}, in such a
way that all the right children are less than the left ones.

Actually $\phi$ preserves more structure, it is equivariant with
respect to two actions of the dihedral group $\mathrm{D}_{2(m+1)}$
with $4(m+1)$ elements.  The action on $\mathcal{Q}_m$ is induced by
the defining action on a regular polygon, and the action on
$\mathcal{N}_m$ is generated by $\kappa$ and $s$.  Furthermore
the action of the subgroup generated by $\kappa^2$ and $s$ is the
standard action of $\mathrm{D}_{m+1}$ on $\mathcal{N}_m$.   This
observation allows us to achieve our goal of enumerating
self-dual unlabeled non-crossing trees in Section~\ref{sec:enum}.

A rank preserving bijection $\mathcal{N} \to \mathcal{T}$ has been
given in~\cite{DulPen1993} modulo an arbitrary choice when $m=2$, and
it turns out that with the appropriate choice that bijection is
exactly the structural bijection, see Section~\ref{sec:sigma}.  Since
by the commutativity of Diagram~\eqref{eq:commdiag},
$\sigma = \psi\circ \phi^{-1}$ we have an interpretation of $\sigma$
as well.

In Section~\ref{sec:enum} we examine the dihedral action on
$\mathcal{Q}_m$ and in Theorem~\ref{thm:fp} we count its fixed points.
This allows us to use Burnside's Lemma to deduce explicit formulae for
the numbers of quadrangular dissections of a $2(m+1)$-gon up to
rotations, and up to rotations and reflections.  These formulae don't
appear to be previously known.  As a corollary we also reprove the
formula for the number of unlabeled non-crossing trees
in~\cite{Noy1998301}.

Additionally, Theorem~\ref{thm:fp} in combination with the ``Counting
Lemma'' of Robinson (Lemma~\ref{lem:robinson},
see~\cite{Robinson1981}) allows us to get explicit formulae for the
number of unlabeled self-dual (oriented or not) non-crossing trees,
one of the original motivations of this work. 

We conclude with some future directions and open questions in
Section~\ref{sec:future}.

\paragraph{Conventions, notation, terminology}
Throughout the paper we use standard notation and terminology, with a
few exceptions that we explain now.

We use the notation $[n] := \left\{ 1,\ldots,n \right\}$. For a finite
set $X$ we denote its \emph{cardinality} by $\left| X \right|$.  For a
set $X$, a subset of $X$ with $k$ elements (respectively, an ordered
$k$-tuple of distinct elements of $X$) is called a
\emph{$k$-combination} (respectively, \emph{a $k$-permutation}) of
elements of $X$.

For a graph its \emph{order} is the number of its vertices, and its
\emph{size} is the number of its edges.  We typically denote by $n$
the order of a graph and by $m$ its size, and since we are typically
dealing with trees, very often we have $n = m+1$.  We call the set of all
edges incident to a given vertex $v$ the \emph{star} of $v$.

We also use the abbreviations, v-graph (respectively e-graph) for a
graph with vertices (respectively edges) labeled by the elements of
$[n]$ (respectively $[m]$).  An  e-v-graph is a graph with 
vertices labeled by $[n]$ and edges labeled by $[m]$.

A \emph{ditree} is a directed tree, that is, a digraph with underlying
undirected graph a tree. A \emph{dag} is a \emph{Directed Acyclic Graph},
that is a digraph with no \emph{oriented} cycles.  A \emph{topological sort}
of a dag is a linear order of its vertices that extends the corresponding
partial order.

Our rooted trees grow upwards and, consistently with the stanadard
orientation of the plane, the children of a vertex of an ordered tree
increase from right to left.

For a set $X$ we denote the symmetric group of $X$ by $\mathrm{S}_{X}$
and when $X=[n]$ we just use the symbol $\mathrm{S}_n$.  We multiply
permutations from left to right so that $(1\,2)(1\,3) = (1\,2\,3)$.

Finally, we use left and right exponential notation for conjugation in
a group, i.e.  $g^{h} := h^{-1} g h$ and
$\prescript{h}{}g := h g h^{-1}$.

\section{Ternary magmas}
\label{sec:ternmagma}

\subsection{Basic Theory}
\label{sec:bastern}

By a \emph{ternary magma} we mean a set $M$ endowed with a ternary
operation $\Upsilon\co M^3 \to M$, which we call \emph{fusion}.  As
expected, a homomorphism of ternary magmas is a map that preserves the
ternary operation and a homomorphism that has an inverse is called an
isomorphism.

If $M$ is a ternary magma and $X\subset M$ we say that $M$ is
\emph{freely generated} by $X$ if for every ternary magma $N$
and any function $f\co X \to N$, there exist a \emph{unique}
ternary homomorphism $\phi\co M \to N$ extending $f$, i.e.
so that the following diagram commutes:

$$
\begin{psmatrix}[mnode=R,colsep=2cm,rowsep=2cm]
  X &   & M \\
    & N & 
\end{psmatrix}
\psset{nodesep=0.3cm}
\ncLine[arrowsize=.2,hooklength=4mm,hookwidth=-2mm]{H->}{1,1}{1,3}
\ncLine[arrowsize=.2]{->}{1,1}{2,2}
\Bput{f}
\ncLine[arrowsize=.2,linestyle=dashed]{->}{1,3}{2,2}
\Aput{\phi}
$$
where the top arrow stands for the inclusion of $X$ into $M$.

Let $ X = \left\{ \lambda_1, \ldots, \lambda_n \right\}$ be a set with
$n$ elements.  One particular realization of the ternary magma freely
generated by $X$ is as the set of words $M(X)$ on the
alphabet $\left\{ \lambda_1,\ldots,\lambda_n, \Upsilon,(,) \right\}$
defined recursively by the rules:
\begin{itemize}
\item $\lambda_i \in M(X)$, for $i=1,\ldots,n$,
\item if $w_{\mathrm{l}},w_\mathrm{m},w_\mathrm{r} \in M( X )$ then $\Upsilon(w_\mathrm{l},w_\mathrm{m},w_\mathrm{r}) \in M ( X )$,
\end{itemize}
with the tautological ternary operator
$\left( w_\mathrm{l},w_\mathrm{m},w_\mathrm{r} \right) \mapsto
\Upsilon( w_\mathrm{l},w_\mathrm{m},w_\mathrm{r}).$  

Clearly for any element $x \in M(X)$ that is not a generator there are
\emph{uniquely determined} elements $x_\mathrm{l}$, $x_{\mathrm{m}}$,
and $x_{\mathrm{r}}$ such that
$x = \Upsilon(x_{\mathrm{l}}, x_{\mathrm{m}}, x_{\mathrm{r}})$.

\begin{defn}
  \label{defn:rank}
  The \emph{rank}, $\mathrm{rk}(w)$, of an element $w \in M( X )$ is
  the number of occurrences of the letter ``$\Upsilon$'' (or
  equivalently the number of matching pairs of parentheses ``$(,)$'')
  in $w$.  For a free ternary magma $M$ we will denote the set of
  elements of rank $m$ by $M_m$.

  More generally, for a free ternary magma, we define the rank of an
  element $\bar{a} = (a_1,\ldots,a_k) \in M^{k}$ as the sum of the
  ranks of its coordinates, i.e.
  $$ \mathrm{rk}(\bar{a}) = \mathrm{rk}(a_1) +\cdots+ \mathrm{rk}(a_k)$$
  and we denote the set of elements of $M^k$ of rank $m$ by
  $M^k_{\phantom{k}m}$.  The rank of a $k$-combination of elements of
  $M$ is defined similarly.
\end{defn}

An easy inductive argument shows that an element of $M^k_{\phantom{k}m}$
has $2m+k$ occurrences of $\lambda_i$s.

By standard abstract nonsense we have that any bijection between $X$
and $Y$ extends to an isomorphism between $M(X)$ and $M(Y)$, and so,
up to isomorphism, it makes sense to talk about the free ternary
magma with $n$ generators. When the generators are not important
we will just use $M(n)$ to denote the free ternary magma with $n$
generators.

Our main interest is in the special case that the generating set
contains only one element $\lambda$.  In that case for any ternary
magma $N$ the choice of one element $n_0\in N$ determines a unique
homomorphism $f\co M\to N$ with $f(\lambda) = n_0$.  In particular any
two ternary magmas freely generated by a single element are isomorphic
via a \emph{unique} isomorphism.  So it makes sense to talk about
\emph{the} ternary magma freely generated by one element.  We will
denote the ternary magma freely generated by one element by
$\mathcal{A}$, so that we have the following recursive definition:
$$\mathcal{A} = \bigcup_{m\ge 0}\mathcal{A}_m$$
where
\begin{itemize}
  \item $\mathcal{A}_{0} = \left\{ \lambda \right\}$,
  \item $\mathcal{A}_{m+1} = \left\{ \Upsilon(a_\mathrm{l},a_\mathrm{m},a_\mathrm{r}): a_\mathrm{l}\in \mathcal{A}_{i}, a_\mathrm{m}\in \mathcal{A}_j, a_\mathrm{r}\in \mathcal{A}_k,\quad i+j+k = m \right\}$. 
\end{itemize}
We will refer to the unique isomorphism between two ternary magmas
freely generated by one element as the \emph{structural bijection}.

It is well known that $\mathcal{A}_m$ is counted by the ($p=3$)
Fuss-Catalan numbers

\begin{equation}
  \label{eq:nu}
 \left| \mathcal{A}_m \right| = \nu_m := \frac{1}{2m+1} \binom{3m}{m}.
\end{equation}

We present a proof of Equation~\eqref{eq:nu} next. In fact, using a
slight generalization of the method of~\cite{Cigler1987}, we prove the
following more general Theorem.

\begin{thm}
  \label{thm:terntuples}
  The number of rank $m$ elements of $\mathcal{A}^k$ is given by
  $$ \left| \mathcal{A}^k_{\phantom{k}m} \right| = \frac{k}{2m+k} \dbinom{3m+k-1}{m}.$$
\end{thm}

\begin{defn}
  \label{defn:sqfree} An element of $M(n)^k$, or a combination of
  elements of $M(n)$, is called \emph{repetition-free} if no generator
  repeats, that is the arguments of all occurrences of $\Upsilon$ are
  pairwise distinct.
\end{defn}

Notice that there are repetition-free elements of rank $m$ in $M(n)^k$ if
and only if $n \ge 2m + k$.

\begin{lem}
  \label{lem:sfree}
  The number of repetition-free $k$-combinations of elements of ${M(2m+k)}$
  of rank $m$ is
  $$\frac{(3m+k-1)!}{m!(k-1)!}.$$
\end{lem}

\begin{proof}
  Let $C$ be the set of such combinations.  We will construct a
  bijection $f\co C\to W$, where $W$ is the set of $m$-combinations of
  words of length $3$ from the alphabet $[3m+k-1]$ with the property
  that all the symbols that occur are distinct.  In other words an
  element of $W$ is a set
  $\left\{ a_{11}\,a_{12}\,a_{13}, \ldots, a_{m1}\,a_{m2}\,a_{m3}
  \right\}$ obtained by splitting a $3m$-permutation of $[3m+k-1]$
  into $m$ words of length $3$.  Such a set of words is obtained by
  first choosing $k-1$ symbols to be omitted from $[3m+k-1]$, and then
  a permutation of the remaining $3m$ symbols.  Since the order of the
  words is not important, every element of $W$ is obtained by $m!$
  such choices.  So:
  \begin{align*}
    \left| W \right| &= \dbinom{3m+k-1}{k-1} \frac{(3m)!}{m!} \\
                     &= \frac{(3m+k-1)!}{m!(k-1)!}.
  \end{align*}

  Let $\bar{p} = \left\{ p_1,\ldots,p_k \right\}$ be an element of
  $C$.  Call an occurrence of $\Upsilon$ in $\bar{p}$ \emph{innermost}
  if all three arguments are $\lambda_i$s.  
  In what follows we will just use $i$ to stand for $\lambda_i$.

  To find $f(\bar{p})$, the word that corresponds to $\bar{p}$, we
  start by ordering all innermost occurrences of $\Upsilon$ with
  respect to increasing largest argument and call the smaller such
  innermost occurrence $2m+k+1$.  One of our $3$-letter words will be
  formed by the three arguments of that occurrence.  Replacing that
  occurrence with $2m+k+1$ gives us a $k$-combination of elements
  of a ternary magma freely generated by $2(m-1)+k$ elements. 
  Proceeding inductively we replace the smallest inner occurrence of
  $\Upsilon$ in this combination with $2m+k+2$ and let its arguments
  form our second word, and so on until we obtain a set of $m$ words
  each of length~$3$.

  Conversely, let $w = \left\{ w_1, w_2, \ldots, w_m \right\}$ be an
  element of $W$.  To find $f^{-1}(w)$ we order the words with respect
  to increasing largest element and call them $2m+k+1,\ldots, 3m+k$ in
  that order.

  Notice that all the symbols that occur in the word
  named $2m+k+i$ are less than $2m+k+i$, for $i=1,\ldots,m$. Indeed,
  for each $i$ there are $m-i$ words larger than $2m+k+i$, and so there
  need to be at least $m-i$ elements of $[3m+k-1]$ larger than the
  largest element of that word that have not been used before.
  
  Let $r_1,\ldots, r_{k-1},3m+k$ be the $k$ symbols from $[3m+k]$ that
  do not occur in any of the $w_i$s.  If $r_i>2m+k$, i.e. it is not a
  generator of $M(2k+2)$, replace the corresponding word, say
  $x_\mathrm{l}\,x_\mathrm{m}\,x_\mathrm{r}$ with
  $\Upsilon \left( x_\mathrm{l},x_\mathrm{m},x_\mathrm{r} \right)$.
  Proceed recursively to get a set of $k$ elements of total rank $m$.
\end{proof}

We give two examples to illustrate the proof. As in the body of the
proof we use $i$ to stand for $\lambda_i$.
\begin{exm}
\label{exm:1}
  Consider $m = 6$ and the following triple of elements of $M(15)$:
  $$ 3, \quad \Upsilon \left( 4, \Upsilon \left( 6, 8, 5 \right),9 \right),
  \quad
  \Upsilon \left( \Upsilon \left( 12,2,7 \right), 13, \Upsilon \left( 10, \Upsilon \left( 11, 15, 1\right), 14 \right) \right).
  $$

  Inductively we get the sequence:
  \begin{align*}
    16 &= 6 \,8 \, 5\\
    17 &= 12 \, 2 \, 7\\
    18 &= 11 \, 15 \, 1\\
    19 &= 4\, 16\, 9\\
    20 &= 10\, 18 \, 14\\
    21 &= 17\, 13\, 20
  \end{align*}
  
  Thus this triple corresponds to the following set of words:
  $$ \left\{ 6 \, 8 \, 5, 12 \, 2 \, 7, 11 \, 15 \, 1, 4\, 16\, 9, 10\, 18 \, 14, 17\, 13\, 20 \right\}. $$
  \end{exm}

  \begin{exm}
    \label{exm:2}
    Conversely, for $m=6$ and $k=3$ let's take the set of words from
    Example~\ref{exm:1}:
      $$ \left\{ 6 \, 8 \, 5, 12 \, 2 \, 7, 11 \, 15 \, 1, 4\, 16\, 9, 10\, 18 \, 14, 17\, 13\, 20 \right\}. $$
  To find the corresponding pair of elements we start by observing
  that  the omitted symbols are $3,19,21$. Label the words as:
  \begin{align*}
    16 &= 6 \,8 \, 5\\
    17 &= 12 \, 2 \, 7\\
    18 &= 11 \, 15 \, 1\\
    19 &= 4\, 16\, 9\\
    20 &= 10\, 18 \, 14\\
    21 &= 17\, 13\, 20
  \end{align*}
  and expanding successively we get the triple:
  \begin{align*}
    T &= 3,\quad 19,\quad 21 \\
      &= 3,\quad \Upsilon \left( 4, 16, 9 \right),\quad \Upsilon \left( 17, 13, 20 \right)\\
      &= 3,\quad \Upsilon \left( 4, \Upsilon \left( 6, 8, 5 \right), 9 \right),\quad \Upsilon \left( \Upsilon \left( 12, 2, 7 \right), 13, \Upsilon \left( 10, 18, 14 \right) \right)\\
      &= 3,\quad \Upsilon \left( 4, \Upsilon \left( 6, 8, 5 \right), 9 \right),\quad \Upsilon \left( \Upsilon \left( 12, 2, 7 \right), 13, \Upsilon \left( 10, \Upsilon \left( 11, 15, 1 \right), 14 \right) \right).
  \end{align*}
\end{exm}

Now we can prove Theorem~\ref{thm:terntuples}.
\begin{proof}[Proof of Theorem~\ref{thm:terntuples}]
  By Lemma~\ref{lem:sfree} we have that the number of $k$-tuples of
  repetition-free elements of $M(2m+k)$ of rank $m$ is
  $$k! \frac{(3m+k-1)!}{m!(k-1)!} = k \frac{(3m+k-1)!}{m!}. $$
  Now there is a $(2m+k)! : 1$ map from the set of such tuples to
  $\mathcal{A}^k_{\phantom{k} m}$ given by replacing all generators
  $\lambda_i$ by the single generator $\lambda$.  It follows that
  \begin{align*}
    \left| \mathcal{A}^k_{\phantom{k} m} \right| &= \frac{k}{(2m+k)!}\,\frac{(3m+k-1)!}{m!}\\
                                   &= \frac{k}{2m+k} \dbinom{3m+k-1}{m}.
  \end{align*}
\end{proof}

\begin{rem}
  The proof of Theorem~\ref{thm:terntuples} given above for $k=1$ appears
  in~\cite{Cigler1987} in the more general context of $r$-ary
  magmas. We chose to expose only the case $r=3$, but the proof,
 \emph{mutatis mutandis}, easily works in the general case. One gets that
  the number of $k$-tuples of rank $m$ of elements of the $r$-ary magma
  freely generated by one element is
  $$ \frac{k}{(r-1)m+k} \dbinom{rm+k-1}{m}. $$

  An equivalent formula appears in page 201 of~\cite{Concrete2e}, see
  also~\cite{Knuth2014Christmas}.  As far as we know the above is the
  only elementary (without the use of generating functions) proof of
  this result.
\end{rem}

\subsection{Duality in $\mathcal{A}$}
\label{sec:magdu}

There is a natural duality in $\mathcal{A}$ defined by
recursively interchanging the left and right argument of any instance
of $\Upsilon$ while leaving the middle argument fixed\footnote{This
  definition was given for ternary trees
  in~\cite{DeutschFereticNoy2002}.  See also Remark~\ref{rem:sddeu}.}.  Formally, the duality is recursively
defined by
\begin{equation}
  \label{eq:terndual}
  \begin{aligned}
    \lambda^{*} &= \lambda\\
    \Upsilon \left(  a_\mathrm{l}, a_\mathrm{m}, a_\mathrm{r}  \right)^{*}
    &= \Upsilon \left( a_\mathrm{r}^{*}, a_{\mathrm{m}}^{*}, a_\mathrm{l}^{*} \right)
  \end{aligned}
\end{equation}
and it is clearly rank preserving.

This duality is transferred via the structural bijection to a duality
in any free ternary magma with one generator.  In what follows we will
refer to a free ternary magma with one generator endowed with that
duality as a \emph{free $*$-magma}.  In the following subsections we
will see that many well known dualities are simply manifestations of
the fact that the underlying set is a free $*$-magma.

\begin{defn}
  \label{defn:sdtern}
  An element of $\mathcal{A}$ is called \emph{self-dual} if
  $a^{*} = a$.  We let
  $\mathcal{S} := \left\{ a \in \mathcal{A} : a^{*} = a \right\}$ and
  we denote by $\mathcal{S}_m$ the set of rank $m$ elements of
  $\mathcal{S}$.
\end{defn}

\begin{thm}
  \label{thm:sdtern} 
  For even $m$, $\mathcal{S}_m$ is in bijection with
  $\mathcal{A}_{\frac{m}{2}}$, while for $m$ odd $\mathcal{S}_{m}$ is in
  bijection with $\mathcal{A}^2_{\phantom{2}\frac{m-1}{2}}$.
\end{thm}

\begin{proof}

  By Equation~\eqref{eq:terndual} we have that if $a\in \mathcal{S}$
  then
  \begin{enumerate}
  \item $a_\mathrm{r} = a_\mathrm{l}^{*}$,
  \item $a_\mathsf{m}\in \mathcal{S}$, and therefore
  \item $\mathrm{rk}(a) = 2 \mathrm{rk}(a_\mathrm{l}) + \mathrm{rk}(a_\mathrm{m}) + 1.$    
  \end{enumerate}

  For each $m$ we will recursively define a bijection $\beta_m$ that
  sends a self-dual element $a$ of rank $m$ to an element of
  $\mathcal{A}_{\frac{m}{2}}$ when $m$ is even and an element of
  $\mathcal{A}^2_{\phantom{2} \frac{m-1}{2}}$ when $m$ is odd. For
  $m = 0,1$ all relevant sets have one element so $\beta_m$ is
  defined.  Assume then that such a bijection $\beta_k$ has been
  defined for all values $k < m$ and let $a\in \mathrm{S}_m$.

  If $m$ is even the third item above implies that $a_\mathrm{m}$ is a self
  dual element of odd rank, so $\beta_{\mathrm{rk}(a_\mathrm{m})}$ is a pair
  of elements of $\mathcal{A}$.  We can then define
  $\beta_m (a)= \Upsilon \left( a_\mathrm{l}, \beta_{\mathrm{rk}(a_{\mathrm{m}})}(a_\mathrm{m}) \right)$.

  If $m$ is odd then $a_\mathrm{m}$ has even rank and
  $\beta_{\mathrm{rk}(a_\mathrm{m})}(a_\mathrm{m})$ is an element of
  $\mathcal{A}_{\frac{\mathrm{rk}(a_\mathrm{m})}{2}}$.  We can then
  define
  $\beta_m (a) =
  (a_\mathrm{l},\beta_{\mathrm{rk}(a_\mathrm{m})}(a_\mathrm{m}) )$.

  To simplify notation we use $\beta$ without subscripts.  To see that
  $\beta$ is indeed a bijection notice that if $b\in \mathcal{A}_k$
  then
  $\beta^{-1}(b) = \Upsilon \left( b_\mathrm{l},
    \beta^{-1}(b_\mathrm{m}, b_\mathrm{r}), b_\mathrm{l}^{*} \right)$,
  while if $(a,b) \in \mathcal{A}^2_{\phantom{2}k}$ then
  $\beta^{-1}(a,b) = \Upsilon \left( a, \beta^{-1}(b), a^{*} \right)$.
\end{proof}

So as a corollary, using the cases $k=1$ and $k=2$ of
Theorem~\ref{thm:terntuples} we have the following explicit formula
for $s_m$ the number of self-dual elements of $\mathcal{A}$ of rank
$m$.

\begin{thm}
  \label{thm:sdternformula} The number of self-dual elements of $\mathcal{A}_m$ is
  $$
  s_m =
  \begin{cases}
    \dfrac{1}{2k+1} \dbinom{3k}{k} & \text{ if $m = 2k$} \\[20pt]
    \dfrac{1}{k+1} \dbinom{3k+1}{k} & \text{ if $m=2k+1$.}
  \end{cases}
  $$
\end{thm}

\begin{rem}
  \label{rem:sddeu}
  Equation~\eqref{eq:terndual} was used
  in~\cite{DeutschFereticNoy2002} to deduce the formula of
  Theorem~\ref{thm:sdternformula} using a generating function
  argument. In that paper the authors prove that $s_m$ is the number
  of self-dual ternary trees\footnote{Called ``symmetric ternary
    trees'' there.} with $m$ internal vertices.
\end{rem}


\begin{rem}
  \label{rem:ternop}
If $M$ is any ternary magma then $\mathcal{A}_m$ acts on $M^{2m+1}$ in
an ``operadic way''.  Namely consider an element $a\in \mathcal{A}_m$
and $\bar{x}\in M^{2m+1}$, and think of the occurrences of $\lambda$
in $a$ as placeholders, the action $a\cdot \bar{x}$ is then given by
substituting $x_i$, the $i$th coordinate of $\bar{x}$ for the $i$th
occurrence of $\lambda$ and evaluating the resulting expression in $M$.
The basic property of this ``action'' is the following operadic
property: let $a = \Upsilon \left( a_l, a_m, a_r \right)$ with
$\mathrm{rk}(a_l) = m_1$, $\mathrm{rk}(a_m) = m_2$, and
$\mathrm{rk}(a_r) = m_3$, and let $\bar{x} \in M^{2m+1}$.  Write
$\bar{x}$ as the concatenation of $\bar{x}_l$, $\bar{x}_m$, and
$\bar{x}_r$, where $\bar{x}_l\in M^{2m_1 + 1}$,
$\bar{x}_m\in M^{2m_2+1}$, and $\bar{x}_r\in M^{2m_3+1}$.  Then we
have
$$
\Upsilon \left( a_\mathrm{l},a_\mathrm{m},a_\mathrm{r} \right) \cdot (\bar{x}_\mathrm{l}, \bar{x}_\mathrm{m}, \bar{x}_\mathrm{r})
= \Upsilon( a_\mathrm{l}\cdot \bar{x}_\mathrm{l}, a_m\cdot \bar{x}_\mathrm{m}, a_\mathrm{r}\cdot \bar{x}_{\mathrm{r}}).
$$
This interpretation of $\mathcal{A}$ as operators is well known to
computer scientists especially with the realization of $\mathcal{A}$
as the set of  ternary trees.
\end{rem}

\subsection{Ternary trees}
\label{sec:terntrees}

Perhaps the most well known example of a free $*$-magma is the set of
(full) ternary trees $\mathcal{T}$.  A ternary tree is an ordered tree
where every internal vertex has exactly three children.  The standard
recursive definition of ternary trees\footnote{See for
  example~\cite{rosen1995discrete} sections 5.3 and 11.1, or any
  ``Discrete Mathematics'' textbook.} exhibits $\mathcal{T}$ as a
ternary magma freely generated by $\lambda$, the ternary tree
consisting of a single vertex, the root, and no edges.  If
$t_{\mathrm{l}}$, $t_{\mathrm{m}}$, and $t_{\mathrm{r}}$ are three
ternary trees, then their fusion
$\Upsilon \left( t_{\mathrm{l}}, t_{\mathrm{m}}, t_{\mathrm{r}}
\right)$ is defined by adding a new vertex $v_0$ declaring it to be
the root, and adding edges from $v_0$ to the roots of
$t_{\mathrm{l}}$, $t_{\mathrm{m}}$, and $t_{\mathrm{r}}$, see
Figure~\ref{fig:ternfus} for an example.

\begin{figure}[htp]
  \centering
  \psset{unit=.8}
  \begin{pspicture}(-6,-1.5)(10,6.7)
    \rput(-3,5){
    \psset{unit=.6}
    \begin{pspicture}(-0.30000, -0.30000)(9.30000, 3.30000)
      \pnode(5.00000,0.00000){0}
      \psdot(5.00000,0.00000)
      \pnode(8.00000,1.00000){1}
      \psdot(8.00000,1.00000)
      \pnode(5.00000,1.00000){2}
      \psdot(5.00000,1.00000)
      \pnode(2.00000,1.00000){3}
      \psdot(2.00000,1.00000)
      \pnode(9.00000,2.00000){4}
      \psdot(9.00000,2.00000)
      \pnode(8.00000,2.00000){5}
      \psdot(8.00000,2.00000)
      \pnode(7.00000,2.00000){6}
      \psdot(7.00000,2.00000)
      \pnode(6.00000,2.00000){7}
      \psdot(6.00000,2.00000)
      \pnode(5.00000,2.00000){8}
      \psdot(5.00000,2.00000)
      \pnode(4.00000,2.00000){9}
      \psdot(4.00000,2.00000)
      \pnode(3.00000,2.00000){10}
      \psdot(3.00000,2.00000)
      \pnode(2.00000,2.00000){11}
      \psdot(2.00000,2.00000)
      \pnode(1.00000,2.00000){12}
      \psdot(1.00000,2.00000)
      \pnode(2.00000,3.00000){13}
      \psdot(2.00000,3.00000)
      \pnode(1.00000,3.00000){14}
      \psdot(1.00000,3.00000)
      \pnode(0.00000,3.00000){15}
      \psdot(0.00000,3.00000)
      \ncline{0}{1}
      \ncline{0}{2}
      \ncline{0}{3}
      \ncline{1}{4}
      \ncline{1}{5}
      \ncline{1}{6}
      \ncline{2}{7}
      \ncline{2}{8}
      \ncline{2}{9}
      \ncline{3}{10}
      \ncline{3}{11}
      \ncline{3}{12}
      \ncline{12}{13}
      \ncline{12}{14}
      \ncline{12}{15}
    \end{pspicture}}
  \uput[-90](-2.6,4){$t_{\mathrm{l}}$}
    \rput(3,5){
    \psset{unit=.6}
    \begin{pspicture}(-0.30000, -0.30000)(5.30000, 3.30000)
      \pnode(3.33333,0.00000){0}
      \psdot(3.33333,0.00000)
      \pnode(4.33333,1.00000){1}
      \psdot(4.33333,1.00000)
      \pnode(3.33333,1.00000){2}
      \psdot(3.33333,1.00000)
      \pnode(2.33333,1.00000){3}
      \psdot(2.33333,1.00000)
      \pnode(4.00000,2.00000){4}
      \psdot(4.00000,2.00000)
      \pnode(2.00000,2.00000){5}
      \psdot(2.00000,2.00000)
      \pnode(1.00000,2.00000){6}
      \psdot(1.00000,2.00000)
      \pnode(5.00000,3.00000){7}
      \psdot(5.00000,3.00000)
      \pnode(4.00000,3.00000){8}
      \psdot(4.00000,3.00000)
      \pnode(3.00000,3.00000){9}
      \psdot(3.00000,3.00000)
      \pnode(2.00000,3.00000){10}
      \psdot(2.00000,3.00000)
      \pnode(1.00000,3.00000){11}
      \psdot(1.00000,3.00000)
      \pnode(0.00000,3.00000){12}
      \psdot(0.00000,3.00000)
      \ncline{0}{1}
      \ncline{0}{2}
      \ncline{0}{3}
      \ncline{3}{4}
      \ncline{3}{5}
      \ncline{3}{6}
      \ncline{4}{7}
      \ncline{4}{8}
      \ncline{4}{9}
      \ncline{6}{10}
      \ncline{6}{11}
      \ncline{6}{12}
    \end{pspicture}}
    \uput[-90](3.55,4){$t_{\mathrm{m}}$}
      \rput(8,5){
    \psset{unit=.6}
    \begin{pspicture}(-0.30000, -0.30000)(5.30000, 4.30000)
      \pnode(2.11111,0.00000){0}
      \psdot(2.11111,0.00000)
      \pnode(3.11111,1.00000){1}
      \psdot(3.11111,1.00000)
      \pnode(2.11111,1.00000){2}
      \psdot(2.11111,1.00000)
      \pnode(1.11111,1.00000){3}
      \psdot(1.11111,1.00000)
      \pnode(2.33333,2.00000){4}
      \psdot(2.33333,2.00000)
      \pnode(1.00000,2.00000){5}
      \psdot(1.00000,2.00000)
      \pnode(0.00000,2.00000){6}
      \psdot(0.00000,2.00000)
      \pnode(4.00000,3.00000){7}
      \psdot(4.00000,3.00000)
      \pnode(2.00000,3.00000){8}
      \psdot(2.00000,3.00000)
      \pnode(1.00000,3.00000){9}
      \psdot(1.00000,3.00000)
      \pnode(5.00000,4.00000){10}
      \psdot(5.00000,4.00000)
      \pnode(4.00000,4.00000){11}
      \psdot(4.00000,4.00000)
      \pnode(3.00000,4.00000){12}
      \psdot(3.00000,4.00000)
      \pnode(2.00000,4.00000){13}
      \psdot(2.00000,4.00000)
      \pnode(1.00000,4.00000){14}
      \psdot(1.00000,4.00000)
      \pnode(0.00000,4.00000){15}
      \psdot(0.00000,4.00000)
      \ncline{0}{1}
      \ncline{0}{2}
      \ncline{0}{3}
      \ncline{3}{4}
      \ncline{3}{5}
      \ncline{3}{6}
      \ncline{4}{7}
      \ncline{4}{8}
      \ncline{4}{9}
      \ncline{7}{10}
      \ncline{7}{11}
      \ncline{7}{12}
      \ncline{9}{13}
      \ncline{9}{14}
      \ncline{9}{15}
    \end{pspicture}}
  \uput[-90](8,3.7){$t_{\mathrm{r}}$}
      \rput(1.5,.7){
    \psset{unit=.5}
    \begin{pspicture}(-0.30000, -0.30000)(19.30000, 5.30000)
      \pnode(11.14815,0.00000){0}
      \psdot(11.14815,0.00000)
      \pnode(16.11111,1.00000){1}
      \psdot(16.11111,1.00000)
      \pnode(17.11111,2.00000){2}
      \psdot(17.11111,2.00000)
      \pnode(16.11111,2.00000){3}
      \psdot(16.11111,2.00000)
      \pnode(15.11111,2.00000){4}
      \psdot(15.11111,2.00000)
      \pnode(16.33333,3.00000){5}
      \psdot(16.33333,3.00000)
      \pnode(15.00000,3.00000){6}
      \psdot(15.00000,3.00000)
      \pnode(14.00000,3.00000){7}
      \psdot(14.00000,3.00000)
      \pnode(18.00000,4.00000){8}
      \psdot(18.00000,4.00000)
      \pnode(16.00000,4.00000){9}
      \psdot(16.00000,4.00000)
      \pnode(15.00000,4.00000){10}
      \psdot(15.00000,4.00000)
      \pnode(19.00000,5.00000){11}
      \psdot(19.00000,5.00000)
      \pnode(18.00000,5.00000){12}
      \psdot(18.00000,5.00000)
      \pnode(17.00000,5.00000){13}
      \psdot(17.00000,5.00000)
      \pnode(16.00000,5.00000){14}
      \psdot(16.00000,5.00000)
      \pnode(15.00000,5.00000){15}
      \psdot(15.00000,5.00000)
      \pnode(14.00000,5.00000){16}
      \psdot(14.00000,5.00000)
      \pnode(12.33333,1.00000){17}
      \psdot(12.33333,1.00000)
      \pnode(13.33333,2.00000){18}
      \psdot(13.33333,2.00000)
      \pnode(12.33333,2.00000){19}
      \psdot(12.33333,2.00000)
      \pnode(11.33333,2.00000){20}
      \psdot(11.33333,2.00000)
      \pnode(13.00000,3.00000){21}
      \psdot(13.00000,3.00000)
      \pnode(11.00000,3.00000){22}
      \psdot(11.00000,3.00000)
      \pnode(10.00000,3.00000){23}
      \psdot(10.00000,3.00000)
      \pnode(14.00000,4.00000){24}
      \psdot(14.00000,4.00000)
      \pnode(13.00000,4.00000){25}
      \psdot(13.00000,4.00000)
      \pnode(12.00000,4.00000){26}
      \psdot(12.00000,4.00000)
      \pnode(11.00000,4.00000){27}
      \psdot(11.00000,4.00000)
      \pnode(10.00000,4.00000){28}
      \psdot(10.00000,4.00000)
      \pnode(9.00000,4.00000){29}
      \psdot(9.00000,4.00000)
      \pnode(5.00000,1.00000){30}
      \psdot(5.00000,1.00000)
      \pnode(8.00000,2.00000){31}
      \psdot(8.00000,2.00000)
      \pnode(5.00000,2.00000){32}
      \psdot(5.00000,2.00000)
      \pnode(2.00000,2.00000){33}
      \psdot(2.00000,2.00000)
      \pnode(9.00000,3.00000){34}
      \psdot(9.00000,3.00000)
      \pnode(8.00000,3.00000){35}
      \psdot(8.00000,3.00000)
      \pnode(7.00000,3.00000){36}
      \psdot(7.00000,3.00000)
      \pnode(6.00000,3.00000){37}
      \psdot(6.00000,3.00000)
      \pnode(5.00000,3.00000){38}
      \psdot(5.00000,3.00000)
      \pnode(4.00000,3.00000){39}
      \psdot(4.00000,3.00000)
      \pnode(3.00000,3.00000){40}
      \psdot(3.00000,3.00000)
      \pnode(2.00000,3.00000){41}
      \psdot(2.00000,3.00000)
      \pnode(1.00000,3.00000){42}
      \psdot(1.00000,3.00000)
      \pnode(2.00000,4.00000){43}
      \psdot(2.00000,4.00000)
      \pnode(1.00000,4.00000){44}
      \psdot(1.00000,4.00000)
      \pnode(0.00000,4.00000){45}
      \psdot(0.00000,4.00000)
      \ncline{0}{1}
      \ncline{0}{17}
      \ncline{0}{30}
      \ncline{1}{2}
      \ncline{1}{3}
      \ncline{1}{4}
      \ncline{4}{5}
      \ncline{4}{6}
      \ncline{4}{7}
      \ncline{5}{8}
      \ncline{5}{9}
      \ncline{5}{10}
      \ncline{8}{11}
      \ncline{8}{12}
      \ncline{8}{13}
      \ncline{10}{14}
      \ncline{10}{15}
      \ncline{10}{16}
      \ncline{17}{18}
      \ncline{17}{19}
      \ncline{17}{20}
      \ncline{20}{21}
      \ncline{20}{22}
      \ncline{20}{23}
      \ncline{21}{24}
      \ncline{21}{25}
      \ncline{21}{26}
      \ncline{23}{27}
      \ncline{23}{28}
      \ncline{23}{29}
      \ncline{30}{31}
      \ncline{30}{32}
      \ncline{30}{33}
      \ncline{31}{34}
      \ncline{31}{35}
      \ncline{31}{36}
      \ncline{32}{37}
      \ncline{32}{38}
      \ncline{32}{39}
      \ncline{33}{40}
      \ncline{33}{41}
      \ncline{33}{42}
      \ncline{42}{43}
      \ncline{42}{44}
      \ncline{42}{45}
    \end{pspicture}}
  \uput[-90](2.3,-.7){$\Upsilon(t_{\mathrm{l}},t_{\mathrm{m}},t_{\mathrm{r}})$}
  \end{pspicture}
  \caption{Fusion of ternary trees.}
  \label{fig:ternfus}
\end{figure}
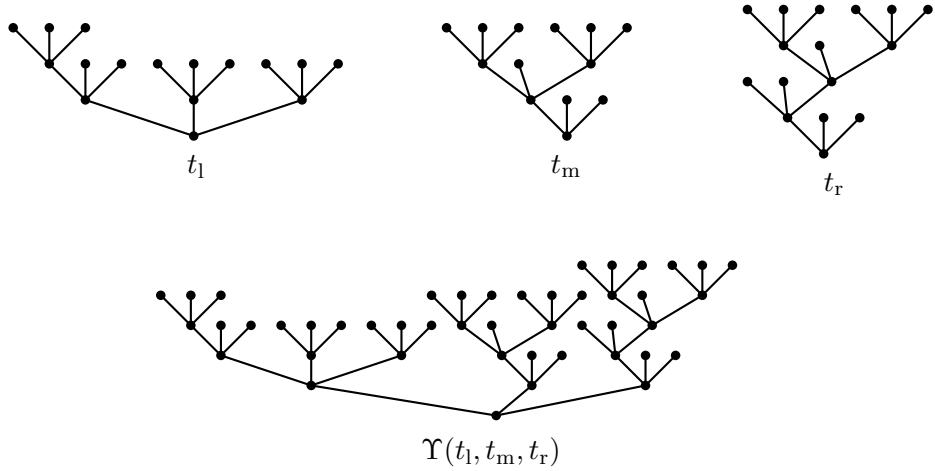
The leaves of a ternary tree, from left to right correspond to
occurrences of $\lambda$ while the internal vertices to occurrences of
$\Upsilon$, so that $\mathcal{T}_m$ consists of all ternary trees with
$m$ internal vertices and therefore, $2m+1$ leaves.  An innermost
occurrence of $\Upsilon$ (see the proof of Lemma~\ref{lem:sfree})
corresponds to \emph{extremal inner vertices}, that is inner vertices
with only leaves as children.  The action of $\mathcal{A}$ on a
ternary magma $M$ described in Remark~\ref{rem:ternop} has the
following graphical interpretation: Let
$\bar{x} = (x_1, \ldots, x_{2m+1}) \in M^{2m+1}$ and
$t\in \mathcal{T}_m$ corresponding to $a \in \mathcal{A}_m$ under
the structural bijection.  To find $a\cdot \bar{x}$ label the leaves
of $t$ with the coordinates of $\bar{x}$ as you encounter them from
left to right. Label every extremal internal vertex with children
labeled $x_{\mathrm{l}}$, $x_{\mathrm{m}}$, $x_{\mathrm{r}}$ by
$\Upsilon \left( x_{\mathrm{l}}, x_{\mathrm{m}}, x_{\mathrm{r}}
\right)$, and proceed to label each vertex that has all its children 
labeled by the fusion of its children.  Then $a\cdot \bar{x}$ is the
label of the root.

The proof of Lemma~\ref{lem:sfree} admits also a graphical
interpretation that we leave to the so inclined reader.  The triple of
elements in Examples~\ref{exm:1} and~~\ref{exm:2} corresponds to the
forest of three ternary trees in Figure~\ref{fig:ciglbij}.

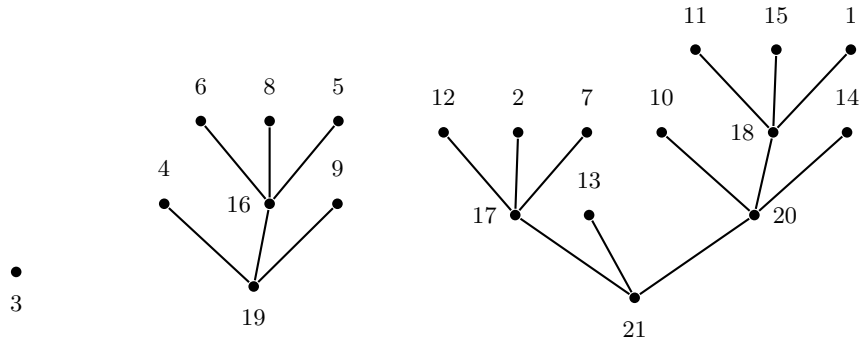
\begin{figure}[ht]
  \centering
  \psset{unit=.8}
    \begin{pspicture}(0,-1.8)(13.5,3.5)
      \rput(0,-.8){
        \psset{unit=.6}
        \footnotesize
        \pstree[treemode=U,levelsep=1.1cm]
        {\Tc*{2pt}~[tnpos=b]{$3$}}{}} 
      \rput(4,0){
        \psset{unit=.7}
        \footnotesize
        \pstree[treemode=U,levelsep=1.1cm]
        {\Tc*{2pt}~[tnpos=b]{$19$}}{
        {\Tc*{2pt}~[tnpos=a]{$4$}}
        \pstree{\Tc*{2pt}~[tnpos=l]{$16$}}{
          \Tc*{2pt}~[tnpos=a]{$6$}
          \Tc*{2pt}~[tnpos=a]{$8$}
          \Tc*{2pt}~[tnpos=a]{$5$}
        }
        {\Tc*{2pt}~[tnpos=a]{$9$}}
        }
      }
      \rput(10.5,.5){
        \psset{unit=.7}
        \footnotesize
        \pstree[treemode=U,levelsep=1.1cm]
        {\Tc*{2pt}~[tnpos=b]{$21$}}{
          \pstree{\Tc*{2pt}~[tnpos=l]{$17$}}{
            \Tc*{2pt}~[tnpos=a]{$12$}
            \Tc*{2pt}~[tnpos=a]{$2$}
            \Tc*{2pt}~[tnpos=a]{$7$}
          }
          \Tc*{2pt}~[tnpos=a]{$13$}
          \pstree{\Tc*{2pt}~[tnpos=r]{$20$}}{
            \Tc*{2pt}~[tnpos=a]{$10$}
            \pstree{\Tc*{2pt}~[tnpos=l]{$18$}}{
              \Tc*{2pt}~[tnpos=a]{$11$}
              \Tc*{2pt}~[tnpos=a]{$15$}
              \Tc*{2pt}~[tnpos=a]{$1$}
            }
            \Tc*{2pt}~[tnpos=a]{$14$}
          }          
        }
        }
    \end{pspicture}
  \caption{The forest of ternary trees corresponding to Examples~\ref{exm:1} and~~\ref{exm:2}.}
  \label{fig:ciglbij}
\end{figure}

For a ternary tree $t$ its dual $t^{*}$ is obtained by interchanging
the left and right subtrees of every internal
vertex. ``Geometrically'' the duality $*$ can be interpreted as
``reflection'' across the middle for all subtrees, see
Figure~\ref{fig:ternduex} for an example.

\begin{figure}[ht]
  \centering
  \begin{pspicture}(6,2)(-6,-2.2)
    \rput(-3,0){
      \psset{unit=.6}
      \begin{pspicture}(-0.30000, -0.30000)(6.30000, 4.30000)
        \pnode(2.33333,0.00000){0}
        \psdot(2.33333,0.00000)
        \pnode(5.00000,1.00000){1}
        \psdot(5.00000,1.00000)
        \pnode(2.00000,1.00000){2}
        \psdot(2.00000,1.00000)
        \pnode(0.00000,1.00000){3}
        \psdot(0.00000,1.00000)
        \pnode(6.00000,2.00000){4}
        \psdot(6.00000,2.00000)
        \pnode(5.00000,2.00000){5}
        \psdot(5.00000,2.00000)
        \pnode(4.00000,2.00000){6}
        \psdot(4.00000,2.00000)
        \pnode(3.00000,2.00000){7}
        \psdot(3.00000,2.00000)
        \pnode(2.00000,2.00000){8}
        \psdot(2.00000,2.00000)
        \pnode(1.00000,2.00000){9}
        \psdot(1.00000,2.00000)
        \pnode(2.00000,3.00000){10}
        \psdot(2.00000,3.00000)
        \pnode(1.00000,3.00000){11}
        \psdot(1.00000,3.00000)
        \pnode(0.00000,3.00000){12}
        \psdot(0.00000,3.00000)
        \pnode(3.00000,4.00000){13}
        \psdot(3.00000,4.00000)
        \pnode(2.00000,4.00000){14}
        \psdot(2.00000,4.00000)
        \pnode(1.00000,4.00000){15}
        \psdot(1.00000,4.00000)
        \ncline{0}{1}
        \ncline{0}{2}
        \ncline{0}{3}
        \ncline{1}{4}
        \ncline{1}{5}
        \ncline{1}{6}
        \ncline{2}{7}
        \ncline{2}{8}
        \ncline{2}{9}
        \ncline{9}{10}
        \ncline{9}{11}
        \ncline{9}{12}
        \ncline{10}{13}
        \ncline{10}{14}
        \ncline{10}{15}
        \uput[-90](2.333,-.3){$t$}
      \end{pspicture}}
    \rput(3,0){
      \psset{unit=.7}
      \begin{pspicture}(-0.30000, -0.30000)(6.30000, 4.30000)
        \pnode(3.33333,0.00000){0}
        \psdot(3.33333,0.00000)
        \pnode(5.00000,1.00000){1}
        \psdot(5.00000,1.00000)
        \pnode(4.00000,1.00000){2}
        \psdot(4.00000,1.00000)
        \pnode(1.00000,1.00000){3}
        \psdot(1.00000,1.00000)
        \pnode(2.00000,2.00000){4}
        \psdot(2.00000,2.00000)
        \pnode(1.00000,2.00000){5}
        \psdot(1.00000,2.00000)
        \pnode(0.00000,2.00000){6}
        \psdot(0.00000,2.00000)
        \pnode(5.00000,2.00000){7}
        \psdot(5.00000,2.00000)
        \pnode(4.00000,2.00000){8}
        \psdot(4.00000,2.00000)
        \pnode(3.00000,2.00000){9}
        \psdot(3.00000,2.00000)
        \pnode(6.00000,3.00000){10}
        \psdot(6.00000,3.00000)
        \pnode(5.00000,3.00000){11}
        \psdot(5.00000,3.00000)
        \pnode(4.00000,3.00000){12}
        \psdot(4.00000,3.00000)
        \pnode(5.00000,4.00000){13}
        \psdot(5.00000,4.00000)
        \pnode(4.00000,4.00000){14}
        \psdot(4.00000,4.00000)
        \pnode(3.00000,4.00000){15}
        \psdot(3.00000,4.00000)
        \ncline{0}{1}
        \ncline{0}{2}
        \ncline{0}{3}
        \ncline{2}{7}
        \ncline{2}{8}
        \ncline{2}{9}
        \ncline{3}{4}
        \ncline{3}{5}
        \ncline{3}{6}
        \ncline{7}{10}
        \ncline{7}{11}
        \ncline{7}{12}
        \ncline{12}{13}
        \ncline{12}{14}
        \ncline{12}{15}
        \uput[-90](3.3333,-.3){$t^{*}$}
      \end{pspicture}}
  \end{pspicture}
  \caption{A ternary tree and its dual.}
  \label{fig:ternduex}
\end{figure}
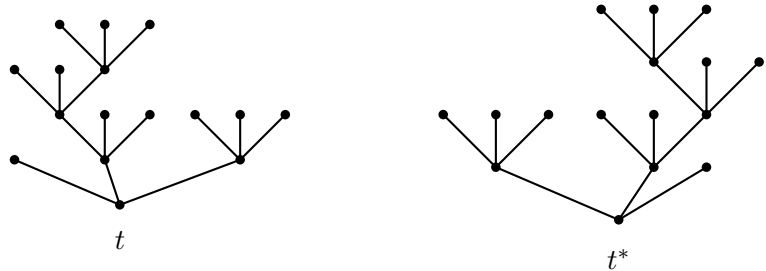

\subsection{Quadrangular dissections of a polygon}
\label{sec:quandrtern}

By a \emph{quadrangular dissection} $q$ of a vertex-labeled
polygon $P$ we mean a subdivision of $P$ into quadrangular cells by
means of non-intersecting diagonals.  An example of a quadrangular
dissection of a decagon is shown on the left side of
Figure~\ref{fig:4clust}, the middle of the same figure shows the same
dissection with the labels of the polygon suppressed, instead we have
chosen a \emph{root edge} which stands for the edge $1\,2$; clearly
the labels of the polygon can be deduced from the root edge and the
standard (counterclockwise) orientation of the plane.  In what follows
we will routinely identify quadrangular dissections of a labeled
polygon with rooted dissections of an unlabeled polygon, and refer to
the cell containing the root edge as the \emph{root cell}, and to the
starting vertex of the root edge as the \emph{root vertex}.

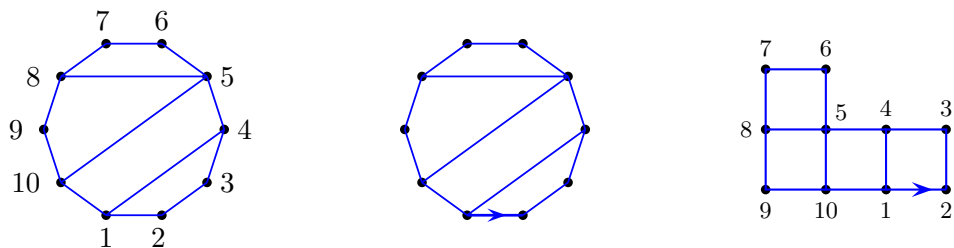
\begin{figure}[ht]
\psset{unit=.8}
  \begin{pspicture}(-6.2,-2)(10,2)
    \rput(7,0){%
      \begin{pspicture}(-3,2)
        \psdots(0,0)(1,0)(1,1)(0,1)(-1,0)(-1,1)(-2,0)(-2,1)(-2,2)(-1,2)
        \psline[linecolor=blue,arrowsize=.25,ArrowInside=->,ArrowInsidePos=0.7](0,0)(1,0)
        \psline[linecolor=blue](1,0)(1,1)(0,1)(0,0)(-1,0)(-1,1)(0,1)
        \psline[linecolor=blue](-1,0)(-2,0)(-2,1)(-1,1)
        \psline[linecolor=blue](-2,1)(-2,2)(-1,2)(-1,1)
        \uput[-90](0,0){\footnotesize $1$}
        \uput[-90](1,0){\footnotesize $2$}
        \uput[90](1,1){\footnotesize $3$}
        \uput[90](0,1){\footnotesize $4$}
        \uput[-90](-1,0){\footnotesize $10$}
        \uput[45](-1,1){\footnotesize $5$}
        \uput[-90](-2,0){\footnotesize $9$}
        \uput[180](-2,1){\footnotesize $8$}
        \uput[90](-2,2){\footnotesize $7$}
        \uput[90](-1,2){\footnotesize $6$}
      \end{pspicture}}
    \rput(-4,0){%
      \psset{unit=1.5}
      \begin{pspicture}(-1,-1)(1,1)
        \pnode(-0.309016994374947, 0.951056516295154){7}
        \psdot(-0.309016994374947, 0.951056516295154)
        \uput[90](-0.309016994374947, 0.951056516295154){$\tiny 7$ }
        \pnode(-0.809016994374947, 0.587785252292473){8}
        \psdot(-0.809016994374947, 0.587785252292473)
        \uput[180](-0.809016994374947, 0.587785252292473){$\tiny 8$ }
        \pnode(-1.00000000000000, 0){9}
        \psdot(-1.00000000000000, 0)
        \uput[180](-1,0){$\tiny 9$ }
        \pnode(-0.809016994374947, -0.587785252292473){10}
        \psdot(-0.809016994374947, -0.587785252292473)
        \uput[180](-0.809016994374947, -0.587785252292473){$\tiny 10$ }
        \pnode(-0.309016994374948, -0.951056516295154){1}
        \psdot(-0.309016994374948, -0.951056516295154)
        \uput[-90](-0.309016994374948, -0.951056516295154){$\tiny 1$}
        \pnode(0.309016994374947, -0.951056516295154){2}
        \psdot(0.309016994374947, -0.951056516295154)
        \uput[-90](0.309016994374947, -0.951056516295154){$\tiny 2$ }
        \pnode(0.809016994374947, -0.587785252292473){3}
        \psdot(0.809016994374947, -0.587785252292473)
        \uput[0](0.809016994374947, -0.587785252292473){$\tiny 3$ }
        \pnode(1.00000000000000, 0){4}
        \psdot(1.00000000000000, 0)
        \uput[0](1,0){$\tiny 4$}
        \pnode(0.809016994374947, 0.587785252292473){5}
        \psdot(0.809016994374947, 0.587785252292473)
        \uput[0](0.809016994374947, 0.587785252292473){$\tiny 5$}
        \pnode(0.309016994374948, 0.951056516295153){6}
        \psdot(0.309016994374948, 0.951056516295153)
        \uput[90](0.309016994374948, 0.951056516295153){$\tiny 6$}
        \ncline[linecolor=blue,linewidth=.02]{1}{2}
        \ncline[linecolor=blue,linewidth=.02]{2}{3}
        \ncline[linecolor=blue,linewidth=.02]{3}{4}
        \ncline[linecolor=blue,linewidth=.02]{4}{5}
        \ncline[linecolor=blue,linewidth=.02]{5}{6}
        \ncline[linecolor=blue,linewidth=.02]{6}{7}
        \ncline[linecolor=blue,linewidth=.02]{7}{8}
        \ncline[linecolor=blue,linewidth=.02]{8}{9}
        \ncline[linecolor=blue,linewidth=.02]{9}{10}
        \ncline[linecolor=blue,linewidth=.02]{10}{1}
        \ncline[linecolor=blue,linewidth=.02]{1}{4}
        \ncline[linecolor=blue,linewidth=.02]{5}{8}
        \ncline[linecolor=blue,linewidth=.02]{5}{10}
      \end{pspicture}}
    \rput(2,0){%
      \psset{unit=1.5,arrowsize=.15}
      \begin{pspicture}(-1,-1)(1,1)
        \pnode(-0.309016994374947, 0.951056516295154){7}
        \psdot(-0.309016994374947, 0.951056516295154)
        \pnode(-0.809016994374947, 0.587785252292473){8}
        \psdot(-0.809016994374947, 0.587785252292473)
        \pnode(-1.00000000000000, 0){9}
        \psdot(-1.00000000000000, 0)
        \pnode(-0.809016994374947, -0.587785252292473){10}
        \psdot(-0.809016994374947, -0.587785252292473)
        \pnode(-0.309016994374948, -0.951056516295154){1}
        \psdot(-0.309016994374948, -0.951056516295154)
        \pnode(0.309016994374947, -0.951056516295154){2}
        \psdot(0.309016994374947, -0.951056516295154)
        \pnode(0.809016994374947, -0.587785252292473){3}
        \psdot(0.809016994374947, -0.587785252292473)
        \pnode(1.00000000000000, 0){4}
        \psdot(1.00000000000000, 0)
        \pnode(0.809016994374947, 0.587785252292473){5}
        \psdot(0.809016994374947, 0.587785252292473)
        \pnode(0.309016994374948, 0.951056516295153){6}
        \psdot(0.309016994374948, 0.951056516295153)
        \ncline[linecolor=blue,linewidth=.03,ArrowInside=->,ArrowInsidePos=0.7]{1}{2}
        \ncline[linecolor=blue,linewidth=.02]{2}{3}
        \ncline[linecolor=blue,linewidth=.02]{3}{4}
        \ncline[linecolor=blue,linewidth=.02]{4}{5}
        \ncline[linecolor=blue,linewidth=.02]{5}{6}
        \ncline[linecolor=blue,linewidth=.02]{6}{7}
        \ncline[linecolor=blue,linewidth=.02]{7}{8}
        \ncline[linecolor=blue,linewidth=.02]{8}{9}
        \ncline[linecolor=blue,linewidth=.02]{9}{10}
        \ncline[linecolor=blue,linewidth=.02]{10}{1}
        \ncline[linecolor=blue,linewidth=.02]{1}{4}
        \ncline[linecolor=blue,linewidth=.02]{5}{8}
        \ncline[linecolor=blue,linewidth=.02]{5}{10}
      \end{pspicture}}    
  \end{pspicture}
  \caption{A quadrangular dissection of a decagon and the associated $4$-cluster.}
  \label{fig:4clust}
\end{figure}

Let $\mathcal{Q} = \bigcup_{m\ge0}\mathcal{Q}_m$, where
$\mathcal{Q}_{m}$ denotes the set of quadrangular dissections with $m$
cells.  In the spirit of~\cite{HararyPalmerRead1975}, we can consider
quadrangular dissections as $4$-clusters, that is as $2$-complexes
defined recursively as follows: the only element of $\mathcal{Q}_1$ is
the standard square with root edge the bottom one oriented from left
to right.  If $q \in \mathcal{Q}_m$ is a $4$-cluster with $m$ cells,
then the $2$-complex obtained by gluing a new square $p$ to $q$ by
identifying, via an orientation reversing homeomorphism, the root edge
of $p$ with a non-root boundary edge of $q$, is a $4$-cluster with
$m+1$ cells and root edge the root of $q$. The right side of
Figure~\ref{fig:4clust} shows the quadrangular dissection in the left
side as a $4$-cluster.

We can easily check, for example using the fact that the Euler
characteristic of the disk is $1$, that a $4$-cluster with $m$ cells
has $2m+2$ vertices and $3m + 1$ edges, $m-1$ of which are
diagonals of the polygon.

In order to exhibit $\mathcal{Q}$ as a free $*$-magma we define
$\lambda$ to be the degenerate quadrangular dissection with $0$ cells
consisting of a single root edge $1\,2$, and set
$\mathcal{Q}_0 = \left\{ \lambda \right\}$.  For
$q_{\mathrm{l}}, q_{\mathrm{m}}, q_{\mathrm{r}}\in \mathcal{Q}$,
$\Upsilon(q_{\mathrm{l}}, q_{\mathrm{m}}, q_{\mathrm{r}})$ is the
quadrangular dissection obtained by identifying the root edge of
$q_{\mathrm{l}}$ ($q_{\mathrm{m}}$ or $q_{\mathrm{r}}$ respectively)
to the left (middle or right respectively) edge of the standard square
by an orientation reversing homeomorphism, in particular
$\Upsilon(\lambda,\lambda,\lambda)$ is the standard square.  Clearly
every quadrangulation is
$\Upsilon(q_{\mathrm{l}}, q_{\mathrm{m}}, q_{\mathrm{r}})$ for some
uniquely defined $q_{\mathrm{l}}$, $q_{\mathrm{m}}$, and
$q_{\mathrm{r}}$. Indeed if $2\,k$ is the leftmost edge of the root
cell of $q$ and $1\,l$ the rightmost, then $q_{\mathrm{l}}$
($q_{\mathrm{m}}$ or $q_{\mathrm{r}}$ respectively) is the
$4$-subcluster of $q$ spanned by the vertices $l,\ldots,1$
($k,\ldots, l$ or $1,\ldots,3$ respectively), see
Figure~\ref{fig:quadups}.  Therefore $\mathcal{Q}$ is a ternary magma
freely generated by $\lambda$.

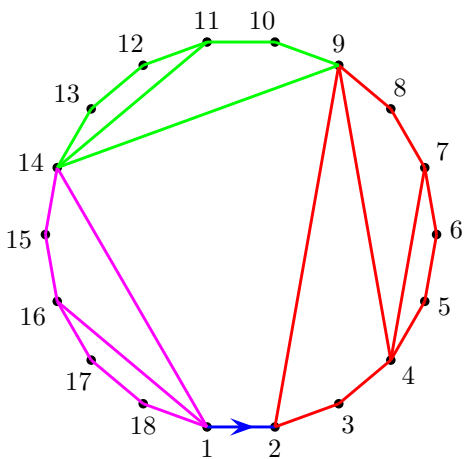
\begin{figure}[ht]
  \centering
  \psset{unit=2.6,arrowsize=.09}
\begin{pspicture}(-1.30000, -1.28481)(1.30000, 1.28481)
  \pnode(-0.17365,-0.98481){1}
  \psdot(-0.17365,-0.98481)
  \pnode(0.17365,-0.98481){2}
  \psdot(0.17365,-0.98481)
  \pnode(0.50000,-0.86603){3}
  \psdot(0.50000,-0.86603)
  \pnode(0.76604,-0.64279){4}
  \psdot(0.76604,-0.64279)
  \pnode(0.93969,-0.34202){5}
  \psdot(0.93969,-0.34202)
  \pnode(1.00000,0.00000){6}
  \psdot(1.00000,0.00000)
  \pnode(0.93969,0.34202){7}
  \psdot(0.93969,0.34202)
  \pnode(0.76604,0.64279){8}
  \psdot(0.76604,0.64279)
  \pnode(0.50000,0.86603){9}
  \psdot(0.50000,0.86603)
  \pnode(0.17365,0.98481){10}
  \psdot(0.17365,0.98481)
  \pnode(-0.17365,0.98481){11}
  \psdot(-0.17365,0.98481)
  \pnode(-0.50000,0.86603){12}
  \psdot(-0.50000,0.86603)
  \pnode(-0.76604,0.64279){13}
  \psdot(-0.76604,0.64279)
  \pnode(-0.93969,0.34202){14}
  \psdot(-0.93969,0.34202)
  \pnode(-1.00000,-0.00000){15}
  \psdot(-1.00000,-0.00000)
  \pnode(-0.93969,-0.34202){16}
  \psdot(-0.93969,-0.34202)
  \pnode(-0.76604,-0.64279){17}
  \psdot(-0.76604,-0.64279)
  \pnode(-0.50000,-0.86603){18}
  \psdot(-0.50000,-0.86603)
  \uput[-90.000](0.17364818,-0.98480775){\small $2$}
  \uput[-64.286](0.50000000,-0.86602540){\small $3$}
  \uput[-38.571](0.76604444,-0.64278761){\small $4$}
  \uput[-12.857](0.93969262,-0.34202014){\small $5$}
  \uput[12.857](1.00000000,0.00000000){\small $6$}
  \uput[38.571](0.93969262,0.34202014){\small $7$}
  \uput[64.286](0.76604444,0.64278761){\small $8$}
  \uput[90.000](0.50000000,0.86602540){\small $9$}
  \uput[115.714](0.17364818,0.98480775){\small $10$}
  \uput[-90.000](-0.17364818,-0.98480775){\small $1$}
  \uput[90.000](-0.17364818,0.98480775){\small $11$}
  \uput[115.714](-0.50000000,0.86602540){\small $12$}
  \uput[141.429](-0.76604444,0.64278761){\small $13$}
  \uput[167.143](-0.93969262,0.34202014){\small $14$}
  \uput[192.857](-1.00000000,-0.00000000){\small $15$}
  \uput[218.571](-0.93969262,-0.34202014){\small $16$}
  \uput[244.286](-0.76604444,-0.64278761){\small $17$}
  \uput[270.000](-0.50000000,-0.86602540){\small $18$}
  \ncline[linecolor=blue,linewidth=.015,ArrowInside=->,ArrowInsidePos=.7]{1}{2}
  \ncline[linewidth=.015,linecolor=magenta]{1}{14}
  \ncline[linewidth=.015,linecolor=magenta]{1}{16}
  \ncline[linewidth=.015,linecolor=magenta]{1}{18}
  \ncline[linewidth=.015,linecolor=red]{2}{3}
  \ncline[linewidth=.015,linecolor=red]{2}{9}
  \ncline[linewidth=.015,linecolor=red]{3}{4}
  \ncline[linewidth=.015,linecolor=red]{4}{5}
  \ncline[linewidth=.015,linecolor=red]{4}{7}
  \ncline[linewidth=.015,linecolor=red]{4}{9}
  \ncline[linewidth=.015,linecolor=red]{5}{6}
  \ncline[linewidth=.015,linecolor=red]{6}{7}
  \ncline[linecolor=red,linewidth=.015]{7}{8}
  \ncline[linecolor=red,linewidth=.015]{8}{9}
  \ncline[linecolor=green,linewidth=.015]{9}{10}
  \ncline[linecolor=green,linewidth=.015]{9}{14}
  \ncline[linecolor=green,linewidth=.015]{10}{11}
  \ncline[linecolor=green,linewidth=.015]{11}{12}
  \ncline[linecolor=green,linewidth=.015]{11}{14}
  \ncline[linecolor=green,linewidth=.015]{12}{13}
  \ncline[linecolor=green,linewidth=.015]{13}{14}
  \ncline[linecolor=magenta,linewidth=.015]{14}{15}
  \ncline[linecolor=magenta,linewidth=.015]{15}{16}
  \ncline[linecolor=magenta,linewidth=.015]{16}{17}
  \ncline[linecolor=magenta,linewidth=.015]{17}{18}
\end{pspicture}
  \caption{Expressing a quadrangular dissection as
    $\Upsilon({\magenta q_{\mathrm{l}}, {\green q_{\mathrm{m}}}, {\red
        q_{\mathrm{r}}}})$.}
  \label{fig:quadups}
\end{figure}

From the description of the fusion of quadrangular dissections it is
clear that for $q\in \mathcal{Q}$ its dual $q^{*}$ is obtained by
reflecting across the perpendicular bisector of the root edge $1\,2$;
see Figure~\ref{fig:q12ndu} for an example.

\begin{figure}[ht]
  \centering
\psset{unit=.8}
\begin{pspicture}(-8,-3)(8,3)
  \rput(-4,0){
    \psset{unit=2.4,arrowsize=.1}
    \begin{pspicture}(-1.26593, -1.26593)(1.26593, 1.26593)
      \pnode(-0.25882,-0.96593){1}
      \psdot(-0.25882,-0.96593)
      \pnode(0.25882,-0.96593){2}
      \psdot(0.25882,-0.96593)
      \pnode(0.70711,-0.70711){3}
      \psdot(0.70711,-0.70711)
      \pnode(0.96593,-0.25882){4}
      \psdot(0.96593,-0.25882)
      \pnode(0.96593,0.25882){5}
      \psdot(0.96593,0.25882)
      \pnode(0.70711,0.70711){6}
      \psdot(0.70711,0.70711)
      \pnode(0.25882,0.96593){7}
      \psdot(0.25882,0.96593)
      \pnode(-0.25882,0.96593){8}
      \psdot(-0.25882,0.96593)
      \pnode(-0.70711,0.70711){9}
      \psdot(-0.70711,0.70711)
      \pnode(-0.96593,0.25882){10}
      \psdot(-0.96593,0.25882)
      \pnode(-0.96593,-0.25882){11}
      \psdot(-0.96593,-0.25882)
      \pnode(-0.70711,-0.70711){12}
      \psdot(-0.70711,-0.70711)
      \uput[-90.000](0.25881905,-0.96592583){\small $2$}
      \uput[-54.000](0.70710678,-0.70710678){\small $3$}
      \uput[-18.000](0.96592583,-0.25881905){\small $4$}
      \uput[18.000](0.96592583,0.25881905){\small $5$}
      \uput[54.000](0.70710678,0.70710678){\small $6$}
      \uput[90.000](0.25881905,0.96592583){\small $7$}
      \uput[-90.000](-0.25881905,-0.96592583){\small $1$}
      \uput[90.000](-0.25881905,0.96592583){\small $8$}
      \uput[126.000](-0.70710678,0.70710678){\small $9$}
      \uput[162.000](-0.96592583,0.25881905){\small $10$}
      \uput[198.000](-0.96592583,-0.25881905){\small $11$}
      \uput[234.000](-0.70710678,-0.70710678){\small $12$}
      \ncline[ArrowInside=->,ArrowInsidePos=.7]{1}{2}
      \ncline{1}{12}
      \ncline{2}{3}
      \ncline{2}{5}
      \ncline{3}{4}
      \ncline{4}{5}
      \ncline{5}{6}
      \ncline{5}{12}
      \ncline{6}{7}
      \ncline{7}{8}
      \ncline{7}{10}
      \ncline{7}{12}
      \ncline{8}{9}
      \ncline{9}{10}
      \ncline{10}{11}
      \ncline{11}{12}
      \psline[linestyle=dashed,linecolor=blue,linewidth=.02](0,-1.2)(0,1.2)
    \end{pspicture}}
  \rput(4,0){
    \psset{unit=2.4,arrowsize=.1}
    \begin{pspicture}(-1.26593, -1.26593)(1.26593, 1.26593)
      \pnode(-0.25882,-0.96593){1}
      \psdot(-0.25882,-0.96593)
      \pnode(0.25882,-0.96593){2}
      \psdot(0.25882,-0.96593)
      \pnode(0.70711,-0.70711){3}
      \psdot(0.70711,-0.70711)
      \pnode(0.96593,-0.25882){4}
      \psdot(0.96593,-0.25882)
      \pnode(0.96593,0.25882){5}
      \psdot(0.96593,0.25882)
      \pnode(0.70711,0.70711){6}
      \psdot(0.70711,0.70711)
      \pnode(0.25882,0.96593){7}
      \psdot(0.25882,0.96593)
      \pnode(-0.25882,0.96593){8}
      \psdot(-0.25882,0.96593)
      \pnode(-0.70711,0.70711){9}
      \psdot(-0.70711,0.70711)
      \pnode(-0.96593,0.25882){10}
      \psdot(-0.96593,0.25882)
      \pnode(-0.96593,-0.25882){11}
      \psdot(-0.96593,-0.25882)
      \pnode(-0.70711,-0.70711){12}
      \psdot(-0.70711,-0.70711)
      \uput[-90.000](0.25881905,-0.96592583){\small $2$}
      \uput[-54.000](0.70710678,-0.70710678){\small $3$}
      \uput[-18.000](0.96592583,-0.25881905){\small $4$}
      \uput[18.000](0.96592583,0.25881905){\small $5$}
      \uput[54.000](0.70710678,0.70710678){\small $6$}
      \uput[90.000](0.25881905,0.96592583){\small $7$}
      \uput[-90.000](-0.25881905,-0.96592583){\small $1$}
      \uput[90.000](-0.25881905,0.96592583){\small $8$}
      \uput[126.000](-0.70710678,0.70710678){\small $9$}
      \uput[162.000](-0.96592583,0.25881905){\small $10$}
      \uput[198.000](-0.96592583,-0.25881905){\small $11$}
      \uput[234.000](-0.70710678,-0.70710678){\small $12$}
      \ncline[ArrowInside=->,ArrowInsidePos=.7]{1}{2}
      \ncline{1}{10}
      \ncline{1}{12}
      \ncline{2}{3}
      \ncline{3}{4}
      \ncline{3}{8}
      \ncline{3}{10}
      \ncline{4}{5}
      \ncline{5}{6}
      \ncline{5}{8}
      \ncline{6}{7}
      \ncline{7}{8}
      \ncline{8}{9}
      \ncline{9}{10}
      \ncline{10}{11}
      \ncline{11}{12}
    \end{pspicture}}
\end{pspicture}
  \caption{A quadrangular dissection of a dodecagon and its dual.}
  \label{fig:q12ndu}
\end{figure}
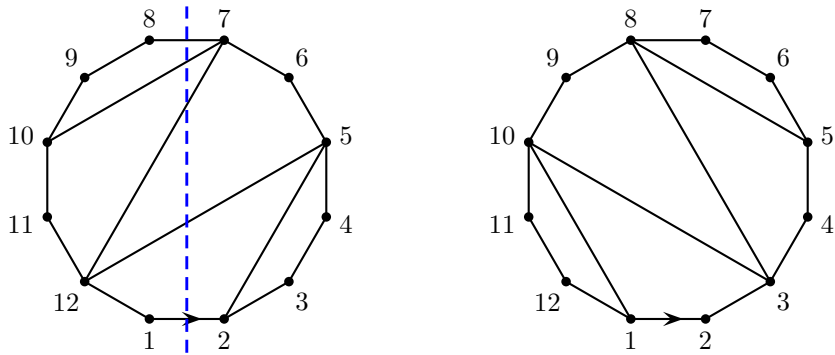

\subsection{Non-Crossing Trees as Properly Embedded Graphs}
\label{sec:pegs}

A non-crossing tree is a tree properly embedded (\emph{pegged}) in a
disk.  The concept of graphs properly embedded in an oriented surface
with boundary, and their duality, was developed
in~\cite{Apos2018arXivApril}. We review the basic
definitions with an eye to the application of the general theory to
the case of trees, so that all our examples will in fact be related
to non-crossing trees.  Most of the concepts are analogous to concepts
in the standard theory of cellularly embedded graphs in closed
surfaces, the reader may consult~\cite{Apos2018arXivApril} for
details.  

A \emph{Properly Embedded Graph} (\emph{peg} for short) is a graph
embedded in a compact oriented surface with boundary in such a way that:
\begin{itemize}
\item the vertices of the graph lie on the boundary of the surface and
  the interior of the edges in the interior of the surface,
\item removing the graph breaks the surface into simply connected
  \emph{regions} and its boundary into \emph{arcs},
\item each region contains exactly one arc in its boundary.
\end{itemize}
 We will refer to a proper embedding as \emph{pegging},
and the graph will be said to be \emph{pegged} into the surface.  For
example in Figure~\ref{fig:exncunl} we see a tree (in green) pegged
into a disk.

We are really interested in pegs up to homeomorphisms of the surface
and we will abuse the language and use peg to refer to an equivalence
class of properly embedded graphs where two pegs are equivalent if they
differ by a homeomorphism.  By an \emph{oriented peg} we mean an
equivalence class of properly embedded graphs where two pegs are
equivalent when they differ by an \emph{orientation preserving}
homeomorphism of the surface. When we want to emphasize that
whether the homeomorphism is orientation preserving or not is
irrelevant we will talk about \emph{unoriented pegs}.  

A \emph{labeled} peg is a peg with vertices labeled by $[n]$, where
$n$ is the order of the graph and homeomorphisms between labeled pegs
are required to preserve labels.

\begin{rem}
  \label{rem:chipeg}
  It is a consequence of the definition that if a graph is pegged in a
surface then the surface homotopically retracts to the graph, and in
particular its Euler characteristic is equal to the Euler
characteristic of the graph.  Since the disk is the only oriented
surface with Euler characteristic $1$ it follows that a graph pegged
in a disk is a tree, and if a tree is pegged in a surface then the
surface is a disk.
\end{rem}

\begin{defn}
  \label{defn:nctrees} A \emph{non-crossing tree} (\emph{nc-tree} for
  short) is a labeled tree pegged in an disk.  For concreteness
  (unless specified otherwise) we assume that all nc-trees are pegged
  in the \emph{standard disk} i.e. the unit disk in $\mathbb{C}$, their
  vertices form a regular polygon, and their labels are increasing in
  the counterclockwise direction.  We denote the set of nc-trees with
  $m$ edges by $\mathcal{N}_m$, and let
  $\mathcal{N} = \bigcup_{m\ge 0} \mathcal{N}_m$.

  An \emph{unlabeled nc-tree} is an unlabeled tree pegged in a disk
  and we denote by $\widetilde{\mathcal{N}}_m$ the set of unlabeled
  nc-trees with $m$ edges and let
  $\widetilde{\mathcal{N}} = \bigcup_{m\ge 0}
  \widetilde{\mathcal{N}}_m$.

  An \emph{oriented nc-tree} is an oriented peg whose underlying graph
  is a tree, we denote by $\mathcal{N}_m'$ the set of oriented
  nc-trees with $m$ edges and let
  $\mathcal{N}' = \bigcup_{m\ge 0} \mathcal{N}_m'$.
\end{defn}

\begin{rem}
  \label{rem:dnact}
  
  The symmetry group of the regular $n$-gon is
  $\mathrm{D}_n = \left\langle r, c \right\rangle$, the dihedral group
  with $2n$ elements, where $r$ stands for the reflection across the
  diameter of the circumscribed circle of the polygon that passes
  through the vertex $1$, and $c$ is counterclockwise rotation by
  $2\pi/n$ radians. If $n = m+1$ then $\mathrm{D}_n$ acts on
  $\mathcal{N}_m$, by rotating and reflecting the edges: for
  $g\in \mathrm{D}_n$, $g(t)$ has an edge $\left(g(i),g(j)\right)$ if
  and only if $t$ has an edge $(i,j)$. Then
  $\widetilde{\mathcal{N}}_m$ is the set of orbits of this
  action,while $\mathcal{N}'_m$ is the set of orbits of the action of
  the cyclic subgroup $\left\langle c \right\rangle$.

  In what follows we will occasionally use the notation $\bar{t}$ to
  stand for $r(t)$.
\end{rem}

Given a peg $\Gamma$, the orientation of the surface induces a cyclic
order on the set of vertices that lie on a given connected component
of the boundary, and this determines an element of
$\mu \left( \Gamma \right) \in \mathrm{S}_V$ called the
\emph{monodromy} of the peg.  Of course, if $\Gamma$ is a labeled peg
of order $n$, then $\mu \left( \Gamma \right)$ can be considered an
element of $\mathrm{S}_n$.  Since the disk has only one boundary
component, for an nc-tree $t$ we have that $\mu(t)$ is an $n$-cycle
$\zeta$, and our convention for the labels means that
$\zeta = (1\,2\,\ldots n)$.

The \emph{mind-body dual peg}\footnote{For an explanation of the term
  \emph{mind-body} see Section~2.3 of~\cite{Apos2018arXivApril}.} of a
graph $\Gamma$ pegged in a surface $F$ is the peg $\Gamma^{*}$ pegged
in $F^{\intercal}$, that is, $F$ endowed with the opposite
orientation, and defined as follows:

\begin{itemize}
\item The vertices of $\Gamma^{*}$ are in one-to-one correspondence
  with the regions of $\Gamma$; when we draw $\Gamma^{*}$ we place
  its vertices on the arcs of the corresponding regions.
\item The edges of $\Gamma^{*}$ are in one-to-one correspondence
  with the edges of $\Gamma$, the edge $e^{*}$ that corresponds to
  the edge $e$ connects the vertices of $\Gamma^{*}$ that correspond
  to the two regions of $\Gamma$ that $e$ lies in the boundary of.
\end{itemize}

Clearly $\left( \Gamma^{*} \right)^{*} = \Gamma$.  An example of the
mind-body dual for an unlabeled nc-tree is shown in
Figure~\ref{fig:exncunl}.

There is a natural correspondence $e\mapsto e^{*}$ between the edges
of $\Gamma$ and $\Gamma^{*}$ but no such natural correspondence exists
between their vertices, so in order to define the dual of a labeled
peg as a labeled peg we have to chose a correspondence
$v \mapsto v^{*}$ between the vertices of $\Gamma$ and those of
$\Gamma^{*}$.  There are two canonical such choices: each vertex of
$\Gamma$ lies in the boundary of two arcs\footnote{For general pegs
  these two arcs could be the same, but this can't happen for
  nc-trees, except in the degenerate case of the tree with no edges.},
one preceding it and one following it in the cyclic order induced by
the orientation, and each of these arcs contains exactly one vertex of
$\Gamma^{*}$.  Our definition of $\Gamma^{*}$ is obtained by making
the first choice, that is $v^{*}$ is the vertex of $\Gamma^{*}$ that
lies in the arc following $v$. When need arises we will denote the
dual obtained by making the second choice by $\Gamma^{\bar{*}}$.  See
Figure~\ref{fig:exnclab} for an example, for one labeling of the
unlabeled nc-tree $t$ of Figure~\ref{fig:exncunl}.  We emphasize that
the nc-trees on the right hand side are pegged in the disk with the
\emph{opposite} (clockwise) orientation; in particular their labelings
do not follow the conventions of Definition~\ref{defn:nctrees} since
their vertices are decreasing if we go around the boundary circle
according to the orientation.  This fact is essential to ensuring that
$\left( t^{*} \right)^{*} = t$ and
$\left( t^{\bar{*}} \right)^{\bar{*}} = t$.

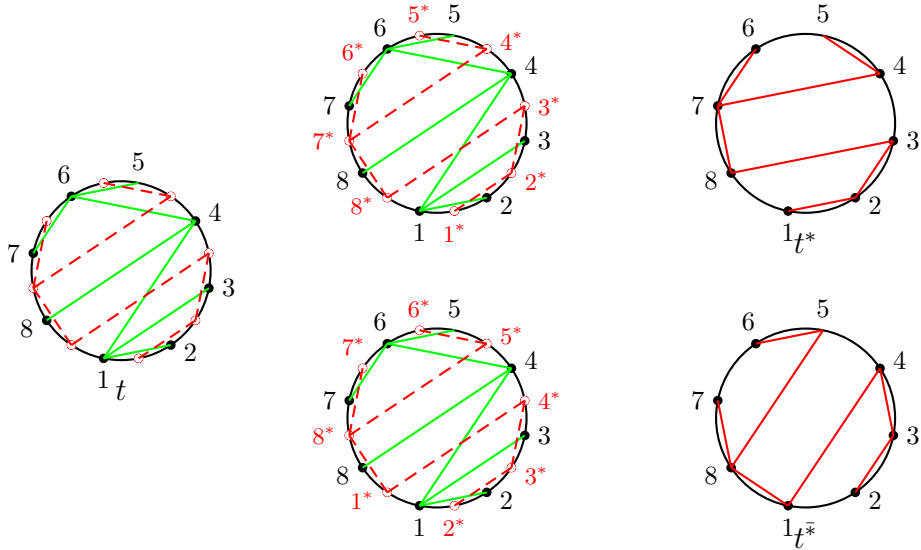
\begin{figure}[ht]
  \centering
  \psset{unit=.7}
  \begin{pspicture}(-8.5,-5.6)(10,5.3)
    \rput(-6,0){
      \psset{unit=1.7}
      \begin{pspicture}(-1.28079, -1.28079)(1.28079, 1.28079)
        \pscircle(0,0){1}
        \pnode(-0.19509,-0.98079){1}
        \psdot(-0.19509,-0.98079)
        \uput[-90](-0.19509,-0.98079){\small $1$}
        \pnode(0.19509,-0.98079){1d}
        \psdot[linecolor=red,dotstyle=o](0.19509,-0.98079)
        \pnode(0.55557,-0.83147){2}
        \psdot(0.55557,-0.83147)
        \uput[-20](0.55557,-0.83147){\small $2$}
        \pnode(0.83147,-0.55557){2d}
        \psdot[linecolor=red,dotstyle=o](0.83147,-0.55557)
        \pnode(0.98079,-0.19509){3}
        \psdot(0.98079,-0.19509)
        \uput[0](0.98079,-0.19509){\small $3$}
        \pnode(0.98079,0.19509){3d}
        \psdot[linecolor=red,dotstyle=o](0.98079,0.19509)
        \pnode(0.83147,0.55557){4}
        \psdot(0.83147,0.55557)
        \uput[20](0.83147,0.55557){\small $4$}
        \pnode(0.55557,0.83147){4d}
        \psdot[linecolor=red,dotstyle=o](0.55557,0.83147)
        \pnode(0.19509,0.98079){5}
        \uput[90](0.19509,0.98079){\small $5$}
        \pnode(-0.19509,0.98079){5d}
        \psdot[linecolor=red,dotstyle=o](-0.19509,0.98079)
        \pnode(-0.55557,0.83147){6}
        \psdot(-0.55557,0.83147)
        \uput[110](-0.55557,0.83147){\small $6$}
        \pnode(-0.83147,0.55557){6d}
        \psdot[linecolor=red,dotstyle=o](-0.83147,0.55557)
        \pnode(-0.98079,0.19509){7}
        \psdot(-0.98079,0.19509)
        \uput[180](-0.98079,0.19509){\small $7$}
        \pnode(-0.98079,-0.19509){7d}
        \psdot[linecolor=red,dotstyle=o](-0.98079,-0.19509)
        \pnode(-0.83147,-0.55557){8}
        \psdot(-0.83147,-0.55557)
        \uput[200](-0.83147,-0.55557){\small $8$}
        \pnode(-0.55557,-0.83147){8d}
        \psdot[linecolor=red,dotstyle=o](-0.55557,-0.83147)
        \ncline[linecolor=green]{1}{2}
        \ncline[linecolor=green]{1}{3}
        \ncline[linecolor=green]{1}{4}
        \ncline[linecolor=green]{4}{6}
        \ncline[linecolor=green]{4}{8}
        \ncline[linecolor=green]{5}{6}
        \ncline[linecolor=green]{6}{7}
        \ncline[linecolor=red, linestyle = dashed]{1d}{2d}
        \ncline[linecolor=red, linestyle = dashed]{2d}{3d}
        \ncline[linecolor=red, linestyle = dashed]{3d}{8d}
        \ncline[linecolor=red, linestyle = dashed]{4d}{5d}
        \ncline[linecolor=red, linestyle = dashed]{4d}{7d}
        \ncline[linecolor=red, linestyle = dashed]{6d}{7d}
        \ncline[linecolor=red, linestyle = dashed]{7d}{8d}
        \uput[90](0,-1.6){\large $t$}
      \end{pspicture}}
    \rput(0,2.8){
      \psset{unit=1.7}
      \begin{pspicture}(-1.28079, -1.28079)(1.28079, 1.28079)
        \pscircle(0,0){1}
        \pnode(-0.19509,-0.98079){1}
        \psdot(-0.19509,-0.98079)
        \uput[-90](-0.19509,-0.98079){\small $1$}
        \pnode(0.19509,-0.98079){1d}
        \psdot[linecolor=red,dotstyle=o](0.19509,-0.98079)
        \uput[-90](0.19509,-0.98079){\red \footnotesize $1^{*}$}
        \pnode(0.55557,-0.83147){2}
        \psdot(0.55557,-0.83147)
        \uput[-20](0.55557,-0.83147){\small $2$}
        \pnode(0.83147,-0.55557){2d}
        \psdot[linecolor=red,dotstyle=o](0.83147,-0.55557)
        \uput[-20](0.83147,-0.55557){\red \footnotesize $2^{*}$}
        \pnode(0.98079,-0.19509){3}
        \psdot(0.98079,-0.19509)
        \uput[0](0.98079,-0.19509){\small $3$}
        \pnode(0.98079,0.19509){3d}
        \psdot[linecolor=red,dotstyle=o](0.98079,0.19509)
        \uput[0](0.98079,0.19509){\red \footnotesize $3^{*}$}
        \pnode(0.83147,0.55557){4}
        \psdot(0.83147,0.55557)
        \uput[20](0.83147,0.55557){\small $4$}
        \pnode(0.55557,0.83147){4d}
        \psdot[linecolor=red,dotstyle=o](0.55557,0.83147)
        \uput[20](0.55557,0.83147){\red \footnotesize $4^{*}$}
        \pnode(0.19509,0.98079){5}
        \uput[90](0.19509,0.98079){\small $5$}
        \pnode(-0.19509,0.98079){5d}
        \psdot[linecolor=red,dotstyle=o](-0.19509,0.98079)
        \uput[90](-0.19509,0.98079){\red \footnotesize $5^{*}$}
        \pnode(-0.55557,0.83147){6}
        \psdot(-0.55557,0.83147)
        \uput[110](-0.55557,0.83147){\small $6$}
        \pnode(-0.83147,0.55557){6d}
        \psdot[linecolor=red,dotstyle=o](-0.83147,0.55557)
        \uput[110](-0.83147,0.55557){\red \footnotesize $6^{*}$}
        \pnode(-0.98079,0.19509){7}
        \psdot(-0.98079,0.19509)
        \uput[180](-0.98079,0.19509){\small $7$}
        \pnode(-0.98079,-0.19509){7d}
        \psdot[linecolor=red,dotstyle=o](-0.98079,-0.19509)
        \uput[180](-0.98079,-0.19509){\red \footnotesize $7^{*}$}
        \pnode(-0.83147,-0.55557){8}
        \psdot(-0.83147,-0.55557)
        \uput[200](-0.83147,-0.55557){\small $8$}
        \pnode(-0.55557,-0.83147){8d}
        \psdot[linecolor=red,dotstyle=o](-0.55557,-0.83147)
        \uput[200](-0.55557,-0.83147){\red \footnotesize $8^{*}$}
        \ncline[linecolor=green]{1}{2}
        \ncline[linecolor=green]{1}{3}
        \ncline[linecolor=green]{1}{4}
        \ncline[linecolor=green]{4}{6}
        \ncline[linecolor=green]{4}{8}
        \ncline[linecolor=green]{5}{6}
        \ncline[linecolor=green]{6}{7}
        \ncline[linecolor=red, linestyle = dashed]{1d}{2d}
        \ncline[linecolor=red, linestyle = dashed]{2d}{3d}
        \ncline[linecolor=red, linestyle = dashed]{3d}{8d}
        \ncline[linecolor=red, linestyle = dashed]{4d}{5d}
        \ncline[linecolor=red, linestyle = dashed]{4d}{7d}
        \ncline[linecolor=red, linestyle = dashed]{6d}{7d}
        \ncline[linecolor=red, linestyle = dashed]{7d}{8d}      
      \end{pspicture}}
    \rput(7,2.8){
      \psset{unit=1.7}
      \begin{pspicture}(-1.28079, -1.28079)(1.28079, 1.28079)
        \pscircle(0,0){1}
        \pnode(-0.19509,-0.98079){1}
        \psdot(-0.19509,-0.98079)
        \uput[-90](-0.19509,-0.98079){\small $1$}
        \pnode(0.55557,-0.83147){2}
        \psdot(0.55557,-0.83147)
        \uput[-20](0.55557,-0.83147){\small $2$}
        \pnode(0.98079,-0.19509){3}
        \psdot(0.98079,-0.19509)
        \uput[0](0.98079,-0.19509){\small $3$}
        \pnode(0.83147,0.55557){4}
        \psdot(0.83147,0.55557)
        \uput[20](0.83147,0.55557){\small $4$}
        \pnode(0.19509,0.98079){5}
        \uput[90](0.19509,0.98079){\small $5$}
        \pnode(-0.55557,0.83147){6}
        \psdot(-0.55557,0.83147)
        \uput[110](-0.55557,0.83147){\small $6$}
        \pnode(-0.98079,0.19509){7}
        \psdot(-0.98079,0.19509)
        \uput[180](-0.98079,0.19509){\small $7$}
        \pnode(-0.83147,-0.55557){8}
        \psdot(-0.83147,-0.55557)
        \uput[200](-0.83147,-0.55557){\small $8$}
        \ncline[linecolor=red]{1}{2}
        \ncline[linecolor=red]{2}{3}
        \ncline[linecolor=red]{3}{8}
        \ncline[linecolor=red]{4}{5}
        \ncline[linecolor=red]{4}{7}
        \ncline[linecolor=red]{6}{7}
        \ncline[linecolor=red]{7}{8}
        \uput[90](0,-1.6){\large $t^{*}$}
      \end{pspicture}}  
    \rput(0,-2.8){
      \psset{unit=1.7}
      \begin{pspicture}(-1.28079, -1.28079)(1.28079, 1.28079)
        \pscircle(0,0){1}
        \pnode(-0.19509,-0.98079){1}
        \psdot(-0.19509,-0.98079)
        \uput[-90](-0.19509,-0.98079){\small $1$}
        \pnode(0.19509,-0.98079){1d}
        \psdot[linecolor=red,dotstyle=o](0.19509,-0.98079)
        \uput[-90](0.19509,-0.98079){\red \footnotesize $2^{*}$}
        \pnode(0.55557,-0.83147){2}
        \psdot(0.55557,-0.83147)
        \uput[-20](0.55557,-0.83147){\small $2$}
        \pnode(0.83147,-0.55557){2d}
        \psdot[linecolor=red,dotstyle=o](0.83147,-0.55557)
        \uput[-20](0.83147,-0.55557){\red \footnotesize $3^{*}$}
        \pnode(0.98079,-0.19509){3}
        \psdot(0.98079,-0.19509)
        \uput[0](0.98079,-0.19509){\small $3$}
        \pnode(0.98079,0.19509){3d}
        \psdot[linecolor=red,dotstyle=o](0.98079,0.19509)
        \uput[0](0.98079,0.19509){\red \footnotesize $4^{*}$}
        \pnode(0.83147,0.55557){4}
        \psdot(0.83147,0.55557)
        \uput[20](0.83147,0.55557){\small $4$}
        \pnode(0.55557,0.83147){4d}
        \psdot[linecolor=red,dotstyle=o](0.55557,0.83147)
        \uput[20](0.55557,0.83147){\red \footnotesize $5^{*}$}
        \pnode(0.19509,0.98079){5}
        \uput[90](0.19509,0.98079){\small $5$}
        \pnode(-0.19509,0.98079){5d}
        \psdot[linecolor=red,dotstyle=o](-0.19509,0.98079)
        \uput[90](-0.19509,0.98079){\red \footnotesize $6^{*}$}
        \pnode(-0.55557,0.83147){6}
        \psdot(-0.55557,0.83147)
        \uput[110](-0.55557,0.83147){\small $6$}
        \pnode(-0.83147,0.55557){6d}
        \psdot[linecolor=red,dotstyle=o](-0.83147,0.55557)
        \uput[110](-0.83147,0.55557){\red \footnotesize $7^{*}$}
        \pnode(-0.98079,0.19509){7}
        \psdot(-0.98079,0.19509)
        \uput[180](-0.98079,0.19509){\small $7$}
        \pnode(-0.98079,-0.19509){7d}
        \psdot[linecolor=red,dotstyle=o](-0.98079,-0.19509)
        \uput[180](-0.98079,-0.19509){\red \footnotesize $8^{*}$}
        \pnode(-0.83147,-0.55557){8}
        \psdot(-0.83147,-0.55557)
        \uput[200](-0.83147,-0.55557){\small $8$}
        \pnode(-0.55557,-0.83147){8d}
        \psdot[linecolor=red,dotstyle=o](-0.55557,-0.83147)
        \uput[200](-0.55557,-0.83147){\red \footnotesize $1^{*}$}
        \ncline[linecolor=green]{1}{2}
        \ncline[linecolor=green]{1}{3}
        \ncline[linecolor=green]{1}{4}
        \ncline[linecolor=green]{4}{6}
        \ncline[linecolor=green]{4}{8}
        \ncline[linecolor=green]{5}{6}
        \ncline[linecolor=green]{6}{7}
        \ncline[linecolor=red, linestyle = dashed]{1d}{2d}
        \ncline[linecolor=red, linestyle = dashed]{2d}{3d}
        \ncline[linecolor=red, linestyle = dashed]{3d}{8d}
        \ncline[linecolor=red, linestyle = dashed]{4d}{5d}
        \ncline[linecolor=red, linestyle = dashed]{4d}{7d}
        \ncline[linecolor=red, linestyle = dashed]{6d}{7d}
        \ncline[linecolor=red, linestyle = dashed]{7d}{8d}      
      \end{pspicture}}
    \rput(7,-2.8){
      \psset{unit=1.7}
      \begin{pspicture}(-1.28079, -1.28079)(1.28079, 1.28079)
        \pscircle(0,0){1}
        \pnode(-0.19509,-0.98079){1}
        \psdot(-0.19509,-0.98079)
        \uput[-90](-0.19509,-0.98079){\small $1$}
        \pnode(0.55557,-0.83147){2}
        \psdot(0.55557,-0.83147)
        \uput[-20](0.55557,-0.83147){\small $2$}
        \pnode(0.98079,-0.19509){3}
        \psdot(0.98079,-0.19509)
        \uput[0](0.98079,-0.19509){\small $3$}
        \pnode(0.83147,0.55557){4}
        \psdot(0.83147,0.55557)
        \uput[20](0.83147,0.55557){\small $4$}
        \pnode(0.19509,0.98079){5}
        \uput[90](0.19509,0.98079){\small $5$}
        \pnode(-0.55557,0.83147){6}
        \psdot(-0.55557,0.83147)
        \uput[110](-0.55557,0.83147){\small $6$}
        \pnode(-0.98079,0.19509){7}
        \psdot(-0.98079,0.19509)
        \uput[180](-0.98079,0.19509){\small $7$}
        \pnode(-0.83147,-0.55557){8}
        \psdot(-0.83147,-0.55557)
        \uput[200](-0.83147,-0.55557){\small $8$}
        \ncline[linecolor=red]{2}{3}
        \ncline[linecolor=red]{3}{4}
        \ncline[linecolor=red]{4}{1}
        \ncline[linecolor=red]{5}{6}
        \ncline[linecolor=red]{5}{8}
        \ncline[linecolor=red]{7}{8}
        \ncline[linecolor=red]{8}{1}
        \uput[90](0,-1.6){\large $t^{\bar{*}}$}
      \end{pspicture}}  
  \end{pspicture}
  \caption{The two mind-body duals of a labeled tree pegged in a disk.}
  \label{fig:exnclab}
\end{figure}

A peg defines two dual structures on its underlying graph: a
\emph{Local Edge Order} (\emph{leo} for short) and a \emph{Perfect
  Trail Double Cover} (\emph{PTDC} for short), that are analogous to a rotation scheme and a
Cycle Double Cover for cellularly embedded graphs, respectively
(see~\cite{gross1987topological} or~\cite{LandoZvonkin2004} for basic
facts and definitions for cellularly embedded graphs).

A leo is simply an assignment of a linear order to the star of each
vertex of $\Gamma$, while a PTDC is is a collection of positive length
trails $\mathcal{T}$ such that:
\begin{itemize}
\item each edge of $\Gamma$ belongs to exactly two trails of
  $\mathcal{T}$,
\item each vertex is the endpoint of exactly two trails of
  $\mathcal{T}$, and we can orient the trails of $\mathcal{T}$ in such
  a way that each \emph{oriented} edge of $\Gamma$ belongs to exactly
  one trail,
\item each vertex $v$ is the beginning of exactly one trail
  $\overrightarrow{v}$ and the end of exactly one trail
  $\overleftarrow{v}$.
\item Finally, we require that unless $v$ is a leaf
the first edge of $\overrightarrow{v}$ is different than the last edge
of $\overleftarrow{v}$.
\end{itemize}

Given a peg its leo is determined by the orientation of the surface:
for every vertex $v$ start slightly ahead of $v$ in the boundary of
the surface and then transverse a positively oriented loop around the
vertex in the interior of the surface and order the edges incident to
$v$ in the order you encounter them.  The PTDC is the collection of
paths that lead from a vertex $v$ to the next: since each region
contains exactly one arc in its boundary there is a path in $\Gamma$
that leads from $v$ to the next vertex, and we define
$\overrightarrow{v}$ to be that path.

The two structures are dual in the following sense: both a leo and a
PTDC can be thought as an assignment of a list of edges to each
vertex. Indeed, the ordering of the star of each vertex can be given by
listing the edges in order, while the trail starting at each vertex
can be described as a list of edges.  Mind-body duality transforms the
lists coming from the leo of $\Gamma$ to the lists coming from the
PTDC of $\Gamma^{*}$, and vice versa.  This can be seen in
Figure~\ref{fig:mainex}, the edges that constitute the trail starting
at a given vertex are exactly the duals of the edges that are incident
to that vertex.

Conversely, the peg can be recovered given the leo or the PTDC of the
graph by gluing $2$-cells to the graph in a procedure analogous to the
way that one obtains a cellular embedding in a closed surface given a
rotation scheme or a Cycle Double Cover. For example we can see in
Figure~\ref{fig:mainex}, that there is a half-disk attached to the tree
along each trail of the PTDC. For details
see~\cite{Apos2018arXivApril}, Section 4.

\begin{figure}[ht]
  \centering
  \begin{pspicture}(-8.5,-12.6)(8.5,3.6)
    \rput(-5,0){
      \psset{unit=3}
      \begin{pspicture}(-1.28079, -1.28079)(1.28079, 1.28079)
        \pscircle(0,0){1}
        \pnode(-0.19509,-0.98079){1}
        \psdot(-0.19509,-0.98079)
        \uput[-90](-0.19509,-0.98079){$1$}
        \pnode(0.19509,-0.98079){1d}
        \psdot[linecolor=red,dotstyle=o](0.19509,-0.98079)
        \pnode(0.55557,-0.83147){2}
        \uput[-90](0.19509,-0.98079){\red \small $1^{*}$}
        \psdot(0.55557,-0.83147)
        \uput[-20](0.55557,-0.83147){$2$}
        \pnode(0.83147,-0.55557){2d}
        \psdot[linecolor=red,dotstyle=o](0.83147,-0.55557)
        \uput[-20](0.83147,-0.55557){\red \small $2^{*}$}
        \pnode(0.98079,-0.19509){3}
        \psdot(0.98079,-0.19509)
        \uput[0](0.98079,-0.19509){$3$}
        \pnode(0.98079,0.19509){3d}
        \psdot[linecolor=red,dotstyle=o](0.98079,0.19509)
        \uput[0](0.98079,0.19509){\red \small $3^{*}$}
        \pnode(0.83147,0.55557){4}
        \psdot(0.83147,0.55557)
        \uput[20](0.83147,0.55557){$4$}
        \pnode(0.55557,0.83147){4d}
        \psdot[linecolor=red,dotstyle=o](0.55557,0.83147)
        \uput[20](0.55557,0.83147){\red \small $4^{*}$}
        \pnode(0.19509,0.98079){5}
        \uput[90](0.19509,0.98079){$5$}
        \pnode(-0.19509,0.98079){5d}
        \psdot[linecolor=red,dotstyle=o](-0.19509,0.98079)
        \uput[90](-0.19509,0.98079){\red \small $5^{*}$}
        \pnode(-0.55557,0.83147){6}
        \psdot(-0.55557,0.83147)
        \uput[110](-0.55557,0.83147){$6$}
        \pnode(-0.83147,0.55557){6d}
        \psdot[linecolor=red,dotstyle=o](-0.83147,0.55557)
        \uput[110](-0.83147,0.55557){\red \small $6^{*}$}
        \pnode(-0.98079,0.19509){7}
        \psdot(-0.98079,0.19509)
        \uput[180](-0.98079,0.19509){$7$}
        \pnode(-0.98079,-0.19509){7d}
        \psdot[linecolor=red,dotstyle=o](-0.98079,-0.19509)
        \uput[180](-0.98079,-0.19509){\red \small $7^{*}$}
        \pnode(-0.83147,-0.55557){8}
        \psdot(-0.83147,-0.55557)
        \uput[200](-0.83147,-0.55557){$8$}
        \pnode(-0.55557,-0.83147){8d}
        \psdot[linecolor=red,dotstyle=o](-0.55557,-0.83147)
        \uput[200](-0.55557,-0.83147){\red \small $8^{*}$}
        \ncline[linecolor=blue]{1}{2}
        \ncline[linecolor=blue]{1}{3}
        \ncline[linecolor=blue]{1}{4}
        \ncline[linecolor=blue]{4}{6}
        \ncline[linecolor=blue]{4}{8}
        \ncline[linecolor=blue]{5}{6}
        \ncline[linecolor=blue]{6}{7}
        \ncline[linecolor=red, linestyle = dashed]{1d}{2d}
        \ncline[linecolor=red, linestyle = dashed]{2d}{3d}
        \ncline[linecolor=red, linestyle = dashed]{3d}{8d}
        \ncline[linecolor=red, linestyle = dashed]{4d}{5d}
        \ncline[linecolor=red, linestyle = dashed]{4d}{7d}
        \ncline[linecolor=red, linestyle = dashed]{6d}{7d}
        \ncline[linecolor=red, linestyle = dashed]{7d}{8d}      
      \end{pspicture}}
    \rput(2,2){The LEO:}
    \rput(2,-.5){%
      \begin{tabular}{ll}
        $1:$ & $1\,2, 1\,3, 1\,4$\\
        $2:$ & $2\,1$\\
        $3:$ & $3\,1$\\
        $4:$ & $4\,6, 4\,8, 4\,1$\\
        $5:$ & $5\,6$\\
        $6:$ & $6\,7, 6\,4, 6\,5$\\
        $7:$ & $7\,6$\\
        $8:$ & $8\,4$
      \end{tabular}}    
    \rput(6.5,2){The PTDC:}
    \rput(6.5,-1){%
      {\def\arraystretch{1.2}\tabcolsep=3pt
        \begin{tabular}{lll}
          $\overrightarrow{1}= 1\,2$ & $= \overleftarrow{2}$\\
          $\overrightarrow{2}= 2\,1, 1\,3$ & $= \overleftarrow{3}$\\
          $\overrightarrow{3}= 3\,1, 1\,4$ & $= \overleftarrow{4}$\\
          $\overrightarrow{4}= 4\,6, 6\,5$ & $= \overleftarrow{5}$\\
          $\overrightarrow{5}= 5\,6$ & $= \overleftarrow{6}$\\
          $\overrightarrow{6}= 6\,7$ & $= \overleftarrow{7}$\\
          $\overrightarrow{7}= 7\,6, 6\,4, 4\,8$ & $= \overleftarrow{8}$\\
          $\overrightarrow{8}= 8\,4, 4\,1$ & $= \overleftarrow{1}$\\                        
        \end{tabular}}}
    \rput(-5,-9){
      \psset{unit=3}
      \begin{pspicture}(-1.28079, -1.28079)(1.28079, 1.28079)
        \pscircle(0,0){1}
        \pnode(-0.19509,-0.98079){1}
        \psdot(-0.19509,-0.98079)
        \uput[-90](-0.19509,-0.98079){$1$}
        \pnode(0.19509,-0.98079){1d}
        \psdot[linecolor=red,dotstyle=o](0.19509,-0.98079)
        \pnode(0.55557,-0.83147){2}
        \uput[-90](0.19509,-0.98079){\red \small $2^{*}$}
        \psdot(0.55557,-0.83147)
        \uput[-20](0.55557,-0.83147){$2$}
        \pnode(0.83147,-0.55557){2d}
        \psdot[linecolor=red,dotstyle=o](0.83147,-0.55557)
        \uput[-20](0.83147,-0.55557){\red \small $3^{*}$}
        \pnode(0.98079,-0.19509){3}
        \psdot(0.98079,-0.19509)
        \uput[0](0.98079,-0.19509){$3$}
        \pnode(0.98079,0.19509){3d}
        \psdot[linecolor=red,dotstyle=o](0.98079,0.19509)
        \uput[0](0.98079,0.19509){\red \small $4^{*}$}
        \pnode(0.83147,0.55557){4}
        \psdot(0.83147,0.55557)
        \uput[20](0.83147,0.55557){$4$}
        \pnode(0.55557,0.83147){4d}
        \psdot[linecolor=red,dotstyle=o](0.55557,0.83147)
        \uput[20](0.55557,0.83147){\red \small $5^{*}$}
        \pnode(0.19509,0.98079){5}
        \uput[90](0.19509,0.98079){$5$}
        \pnode(-0.19509,0.98079){5d}
        \psdot[linecolor=red,dotstyle=o](-0.19509,0.98079)
        \uput[90](-0.19509,0.98079){\red \small $6^{*}$}
        \pnode(-0.55557,0.83147){6}
        \psdot(-0.55557,0.83147)
        \uput[110](-0.55557,0.83147){$6$}
        \pnode(-0.83147,0.55557){6d}
        \psdot[linecolor=red,dotstyle=o](-0.83147,0.55557)
        \uput[110](-0.83147,0.55557){\red \small $7^{*}$}
        \pnode(-0.98079,0.19509){7}
        \psdot(-0.98079,0.19509)
        \uput[180](-0.98079,0.19509){$7$}
        \pnode(-0.98079,-0.19509){7d}
        \psdot[linecolor=red,dotstyle=o](-0.98079,-0.19509)
        \uput[180](-0.98079,-0.19509){\red \small $8^{*}$}
        \pnode(-0.83147,-0.55557){8}
        \psdot(-0.83147,-0.55557)
        \uput[200](-0.83147,-0.55557){$8$}
        \pnode(-0.55557,-0.83147){8d}
        \psdot[linecolor=red,dotstyle=o](-0.55557,-0.83147)
        \uput[200](-0.55557,-0.83147){\red \small $1^{*}$}
        \ncline[linecolor=blue]{1}{2}
        \ncline[linecolor=blue]{2}{3}
        \ncline[linecolor=blue]{3}{8}
        \ncline[linecolor=blue]{4}{5}
        \ncline[linecolor=blue]{4}{7}
        \ncline[linecolor=blue]{6}{7}
        \ncline[linecolor=blue]{7}{8}
        \ncline[linecolor=red, linestyle = dashed]{1d}{8d}
        \ncline[linecolor=red, linestyle = dashed]{8d}{2d}
        \ncline[linecolor=red, linestyle = dashed]{8d}{3d}
        \ncline[linecolor=red, linestyle = dashed]{3d}{7d}
        \ncline[linecolor=red, linestyle = dashed]{3d}{5d}
        \ncline[linecolor=red, linestyle = dashed]{4d}{5d}
        \ncline[linecolor=red, linestyle = dashed]{6d}{5d}      
      \end{pspicture}}
    \rput(2,-6.6){The LEO:}
    \rput(2,-9.1){%
      \begin{tabular}{ll}
        $1:$ & $1\,2$\\
        $2:$ & $2\,1, 2\,3$\\
        $3:$ & $3\,2, 3\,8$\\
        $4:$ & $4\,7, 4\,5$\\
        $5:$ & $5\,4$\\
        $6:$ & $6\,7$\\
        $7:$ & $7\,6, 7\,4, 7\,8$\\
        $8:$ & $8\,7, 8\,3$
      \end{tabular}}    
    \rput(6.5,-6.6){The PTDC:}
    \rput(6.5,-9.6){%
      {\def\arraystretch{1.2}\tabcolsep=3pt
        \begin{tabular}{lll}
          $\overrightarrow{1}= 1\,2, 2\,3, 3\,8$ & $= \overleftarrow{8}$\\
          $\overrightarrow{2}= 2\,1$ & $= \overleftarrow{1}$\\
          $\overrightarrow{3}= 3\,2$ & $= \overleftarrow{2}$\\
          $\overrightarrow{4}= 4\,7, 7\,8, 8\,3$ & $= \overleftarrow{3}$\\
          $\overrightarrow{5}= 5\,4$ & $= \overleftarrow{4}$\\
          $\overrightarrow{6}= 6\,7, 7\,4, 4\,5$ & $= \overleftarrow{5}$\\
          $\overrightarrow{7}= 7\,6$ & $= \overleftarrow{6}$\\
          $\overrightarrow{8}= 8\,7$ & $= \overleftarrow{7}$\\                        
        \end{tabular}}}   
  \end{pspicture}
  \caption{Leos, PTDCs, and duality.}
  \label{fig:mainex}
\end{figure}

Pegs and their duality are closely related to factorizations of
permutations into products of transpositions, indeed there is an
obvious bijective correspondence\footnote{First observed by D\'enes
  in~\cite{Denes1959}.} between factorizations of permutations of
$\mathrm{S}_n$ into a product of $m$ transpositions and edge-labeled
graphs of size $m$ with vertex set $[n]$, where as usual $n=m+1$.
Indeed such a factorization $\rho$ can be viewed as a sequence of $m$
transpositions $\rho = (\tau_1,\ldots,\tau_m)$, and the corresponding
graph has an edge labeled $i$ connecting $k$ and $l$ if and only if
$\tau_i = (k\,l)$.  For example the e-v-tree that corresponds to the
factorization
$\rho = (6\,7), (4\,6), (5\,6), (4\,8), (1\,2), (1\,3), (1\,4)$ of the
$8$-cycle $(1\,2\,\ldots\,8)$ is shown in left side of
Figure~\ref{fig:evtreeex}.  It is useful to consider factorizations up
to conjugation (that is we consider factorizations $(\tau_i)$ and
$(\tau_i^{\pi})$ the same for any permutation $\pi$) and these
correspond to e-graphs. 

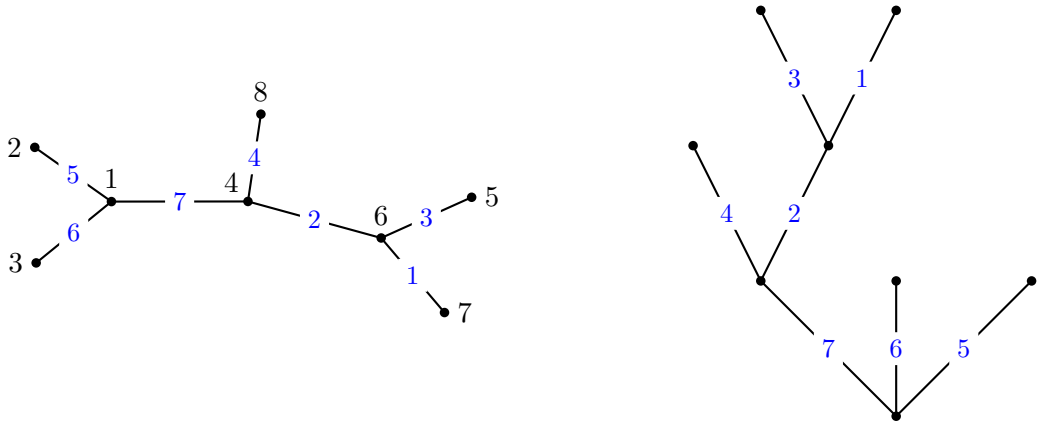
\begin{figure}[ht]
  \centering
\psset{unit=.9}
  \begin{pspicture}(-6.8,-3.3)(9,3.3)
    \rput(6,0){
      \psset{unit=2}
      \begin{pspicture}(-0.30000, -0.30000)(2.80000, 3.30000)
        \pnode(1.50000,0.00000){1}
        \psdot(1.50000,0.00000)
        \pnode(2.50000,1.00000){2}
        \psdot(2.50000,1.00000)
        \pnode(1.50000,1.00000){3}
        \psdot(1.50000,1.00000)
        \pnode(0.50000,1.00000){4}
        \psdot(0.50000,1.00000)
        \pnode(1.50000,3.00000){5}
        \psdot(1.50000,3.00000)
        \pnode(1.00000,2.00000){6}
        \psdot(1.00000,2.00000)
        \pnode(0.50000,3.00000){7}
        \psdot(0.50000,3.00000)
        \pnode(0.00000,2.00000){8}
        \psdot(0.00000,2.00000)
        \ncline{1}{2}
        \ncput*{\small \blue $5$}
        \ncline{1}{3}
        \ncput*{\small \blue $6$}
        \ncline{1}{4}
        \ncput*{\small \blue $7$}
        \ncline{4}{6}
        \ncput*{\small \blue $2$}    
        \ncline{4}{8}
        \ncput*{\small \blue $4$}    
        \ncline{5}{6}
        \ncput*{\small \blue $1$}    
        \ncline{6}{7}
        \ncput*{\small \blue $3$}    
      \end{pspicture}}
    \rput(-3,0){  
      \psset{unit=3.3}
      \begin{pspicture}(0.39137, -0.76203)(2.94686, 0.72671)
        \pnode(1.03451,0.03505){1}
        \psdot(1.03451,0.03505)
        \uput[90](1.03451,0.03505){$1$}
        \pnode(0.69137,0.27796){2}
        \uput[180](0.69137,0.27796){$2$}
        \psdot(0.69137,0.27796)
        \pnode(0.69751,-0.23965){3}
        \uput[180](0.69751,-0.23965){$3$}
        \psdot(0.69751,-0.23965)
        \pnode(1.64559,0.03562){4}
        \psdot(1.64559,0.03562)
        \uput[135](1.64559,0.03562){$4$}
        \pnode(2.64686,0.05437){5}
        \psdot(2.64686,0.05437)
        \uput[0](2.64686,0.05437){$5$}
        \pnode(2.24191,-0.12803){6}
        \psdot(2.24191,-0.12803)
        \uput[90](2.24191,-0.12803){$6$}
        \pnode(2.52534,-0.46203){7}
        \psdot(2.52534,-0.46203)
        \uput[0](2.52534,-0.46203){$7$}
        \pnode(1.70360,0.42671){8}
        \psdot(1.70360,0.42671)
        \uput[90](1.70360,0.42671){$8$}
        \ncline{1}{2}
        \ncput*{\small \blue $5$}
        \ncline{1}{3}
        \ncput*{\small \blue $6$}
        \ncline{1}{4}
        \ncput*{\small \blue $7$}
        \ncline{4}{6}
        \ncput*{\small \blue $2$}    
        \ncline{4}{8}
        \ncput*{\small \blue $4$}    
        \ncline{5}{6}
        \ncput*{\small \blue $3$}    
        \ncline{6}{7}
        \ncput*{\small \blue $1$}    
      \end{pspicture}}      
  \end{pspicture}
  \caption{The e-v-tree (left) and the the rooted e-tree (right) that
    corresponds to the factorization of our running example.}
  \label{fig:evtreeex}
\end{figure}

In fact a factorization (or e-v-graph) determines a labeled peg, and
an e-graph determines an unlabeled peg.  Indeed the edge-labels induce
a linear order of the edges which restricts to a linear order at the
star of each vertex; equivalently the trajectories of the vertices
determining a PTDC\footnote{Alternatively we can obtain the peg as the
  total space of a branched covering of the disk,
  see~\cite{Apos2018arXivApril}, Section 4.4.}.  The mind-body duality
then can be transferred to factorizations to define $\rho^{*}$ and
$\rho^{\bar{*}}$, and we have the following explicit
formulas:\footnote{Recall that we use left and right exponential
  notation for conjugation. Since transpositions are involutions the
  distinction is mute in the current context, however it is useful in more
  general contexts.  }
\begin{align}
  \rho^{*} &= \tau_1,\prescript{\tau_1}{}\tau_2,\ldots, \prescript{\tau_1\ldots\tau_{m-1}}{}\tau_m \label{eq:dualeq1}
  \\
  \rho^{\bar{*}} &= \tau_1^{\tau_2\tau_3\ldots\tau_m}, \tau_2^{\tau_3\ldots\tau_m}, \ldots, \tau_{m-1}^{\tau_m}, \tau_m.\label{eq:dualeq2}
\end{align}
For the factorization of our example we have
$\rho^{*} = (6\,7), (4\,7), (4\,5), (7\,8), (1\,2), \allowbreak (2\,3), (3\,8)$
and
$\rho^{\bar{*}} = (7\,8), (5\,8), (5\,6), (1\,8), (2\,3), (3\,4),
(1\,4)$, and we emphasize that these are factorizations of the inverse
cycle $(8\,7\,\ldots\,1)$.

These formulas are best understood via the \emph{Hurwitz action} of
the braid group on factorizations.  Recall that $\mathrm{B}_m$, the
braid group with $m$ strands, is the group generated by $m-1$
generators $\sigma_1,\ldots, \sigma_{m-1}$ subject to the relations
$\sigma_i \sigma_{i+1} \sigma_i = \sigma_{i+1} \sigma_i \sigma_{i+1}$,
for $i=1,\ldots,m-2$ and $\sigma_i \sigma_j = \sigma_j\sigma_i$ if
$|i-j|\ge 2$.  One of the basic incarnations of $\mathrm{B}_m$ is as a
group of automorphims of $\mathrm{F}_m$ the free group with $m$
generators: if $x_1,\ldots,x_m$ are the generators of $\mathrm{F}_m$
then the action of the generator $\sigma_i$ is given by
$\sigma_i\, x_j = x_j$ for $j\ne i,i+1$, while
$\sigma_i \, x_i = \prescript{x_i}{} x_{i+1}$ and
$\sigma_i \, x_{i+1} = x_i$.  It follows that $\mathrm{B}_m$ acts on
the right on the set of homomorphisms $\mathrm{F}_m \to G$, for any
group $G$ and in particular for $G$ a symmetric group.  A
factorization $\rho$ is a sequence of elements in a symmetric group,
and therefore can be construed as a representation of $\mathsf{F}_m$
to that group.  So we have a right action of $\mathrm{B}_m$ on the set
of all factorizations in any symmetric group, this action is called
the \emph{Hurwitz action}.  If $\rho = \tau_1,\ldots,\tau_m$ is a
factorization, then for the $i$th generator of $\mathrm{B}_m$ we have
that $\rho \sigma_i = \tau_1', \ldots, \tau_m'$, where
$\tau_i' = \prescript{\tau_i}{}\tau_{i+1}$, $\tau_{i+1}' = \tau_i$,
and $\tau_j' = \tau_j$ for $j\neq i,i+1$.

The braid
$\Delta_m = \sigma_1\ldots \sigma_m \sigma_1\ldots \sigma_{m-1} \ldots
\sigma_1\sigma_2 \sigma_{1} \in \mathrm{B}_m$ is called the
\emph{Garside element} of $\mathrm{B}_m$ and  it plays an important
role in the theory of Braid Groups, for example it is a square root of
the generator of the center of $\mathrm{B}_m$.  Its importance for the
present work is that formulas~\eqref{eq:dualeq1}
and~\eqref{eq:dualeq2} can be written as
\begin{align}
  \rho^{*} &= \left( \rho \Delta_m \right)^{\intercal} \label{eq:dualeqgar1}\\
  \rho^{\bar{*}} &= \left( \rho \Delta_m^{-1} \right)^{\intercal}\label{eq:dualeqgar2}
\end{align}
where for a factorization $\rho$, $\rho^{\intercal}$ stands for the
factorization red backwards: if $\rho = (\tau_1,\ldots,\tau_m)$
then $\rho^{\intercal} = (\tau_m,\ldots,\tau_1)$.

Using the bijection between factorizations in $\mathrm{S}_n$ and
e-graphs on $[n]$ we can transfer this to a $\mathrm{B}_m$-action on
the set of e-labeled graphs on $[n]$ with $m$ edges.  It is easily
seen that if $\Gamma$ is an e-v-graph then $\Gamma \sigma_i$ is
obtained from $\Gamma$ by interchanging the labels of the $i$-th and
$(i+1)$-th edge and then ``sliding'' the $(i+1)$-th edge along the
$i$-th, while $\Gamma \sigma_i^{-1}$ is obtained by interchanging the
$i$-th and $(i+1)$-th labels and then sliding the $i$-th edge along
the $(i+1)$-th.  We interpret a slide of an edge along a non-adjacent
edge to have no effect.  This action on e-v-labeled graphs, which
we'll also call the \emph{Hurwitz action}, is shown in
figure~\ref{fig:hurfacg}, where only the edges labeled $i$ and $i+1$
are shown since the other edges are not affected.

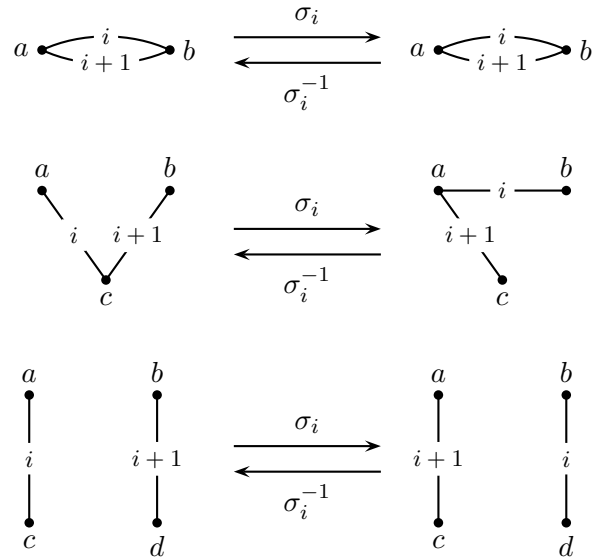
\begin{figure}[htbp]
  \centering
  \psset{unit=1.7}
  \begin{pspicture}(-1.2,-1)(4.85,4.3)
    \rput(-.2,3.5){
      \rput(0, 0){\rnode{a}{\psdot(0,0)}}
      \uput[180](0,0){$a$}
      \rput(1, 0){\rnode{b}{\psdot(0,0)}}
      \uput[0](1,0){$b$}
      \ncarc[arcangle=25]{a}{b}
      \ncput*{\footnotesize $i$}
      \ncarc[arcangle=-25]{a}{b}
      \ncput*{\footnotesize $i+1$}}
    \psline[arrowsize=.09]{->}(1.3,3.6)(2.45,3.6)
    \psline[arrowsize=.09]{<-}(1.3,3.4)(2.45,3.4)
    \uput[90](1.875,3.6){$\sigma_{i}$}
    \uput[-90](1.875,3.4){$\sigma_i^{-1}$}
    \rput(2.9,3.5){
      \rput(0, 0){\rnode{a}{\psdot(0,0)}}
      \uput[180](0,0){$a$}
      \rput(1, 0){\rnode{b}{\psdot(0,0)}}
      \uput[0](1,0){$b$}
      \ncarc[arcangle=25]{a}{b}
      \ncput*{\footnotesize $i$}
      \ncarc[arcangle=-25]{a}{b}
      \ncput*{\footnotesize $i+1$}}
    \rput(-.2,2.4){
      \rput(0, 0){\rnode{a}{\psdot(0,0)}}
      \uput[90](0,0){$a$}
      \rput(1, 0){\rnode{b}{\psdot(0,0)}}
      \uput[90](1,0){$b$}
      \rput(.5, -.7){\rnode{c}{\psdot(0,0)}}
      \uput[-90](.5,-.7){$c$}
      \ncline{a}{c}
      \ncput*{\footnotesize $i$}
      \ncline{b}{c}
      \ncput*{\footnotesize $i+1$}}
    \psline[arrowsize=.09]{->}(1.3,2.1)(2.45,2.1)
    \psline[arrowsize=.09]{<-}(1.3,1.9)(2.45,1.9)
    \uput[90](1.875,2.1){$\sigma_{i}$}
    \uput[-90](1.875,1.9){$\sigma_i^{-1}$}
    \rput(2.9,2.4){
      \rput(0, 0){\rnode{a}{\psdot(0,0)}}
      \uput[90](0,0){$a$}
      \rput(1, 0){\rnode{b}{\psdot(0,0)}}
      \uput[90](1,0){$b$}
      \rput(.5, -.7){\rnode{c}{\psdot(0,0)}}
      \uput[-90](.5,-.7){$c$}
      \ncline{a}{c}
      \ncput*{\footnotesize $i+1$}
      \ncline{a}{b}
      \ncput*{\footnotesize $i$}}
    \rput(-.3,.8){
      \rput(0, 0){\rnode{a}{\psdot(0,0)}}
      \uput[90](0,0){$a$}
      \rput(1, 0){\rnode{b}{\psdot(0,0)}}
      \uput[90](1,0){$b$}
      \rput(0, -1){\rnode{c}{\psdot(0,0)}}
      \uput[-90](0,-1){$c$}
      \rput(1, -1){\rnode{d}{\psdot(0,0)}}
      \uput[-90](1,-1){$d$}
      \ncline{a}{c}
      \ncput*{\footnotesize $i$}
      \ncline{b}{d}
      \ncput*{\footnotesize $i+1$}}
    \psline[arrowsize=.09]{->}(1.3,.4)(2.45,.4)
    \psline[arrowsize=.09]{<-}(1.3,.2)(2.45,.2)
    \uput[90](1.875,.4){$\sigma_{i}$}
    \uput[-90](1.875,.2){$\sigma_i^{-1}$}
    \rput(2.9,.8){
      \rput(0, 0){\rnode{a}{\psdot(0,0)}}
      \uput[90](0,0){$a$}
      \rput(1, 0){\rnode{b}{\psdot(0,0)}}
      \uput[90](1,0){$b$}
      \rput(0, -1){\rnode{c}{\psdot(0,0)}}
      \uput[-90](0,-1){$c$}
      \rput(1, -1){\rnode{d}{\psdot(0,0)}}
      \uput[-90](1,-1){$d$}
      \ncline{a}{c}
      \ncput*{\footnotesize $i+1$}
      \ncline{b}{d}
      \ncput*{\footnotesize $i$}}
  \end{pspicture}
  \caption{The Hurwitz action on e-v-graphs.}
  \label{fig:hurfacg}
\end{figure}

Notice that this action descends at the level of e-labeled graphs
(just forget the v-labels in Figure~\ref{fig:hurfacg}). We will still
call it the Hurwitz action since no confusion is likely to arise, and
we use formulas~\eqref{eq:dualeqgar1} and ~\eqref{eq:dualeqgar2} to
define mind-body duality for labeled graphs.

For a fixed $n$-cycle $\zeta$ (say $\zeta = (1\,2\,\ldots\,n)$) denote
by $\mathcal{F}_m$ the set of minimal transitive factorizations of
$\zeta$, or equivalently the set of e-v-trees with monodromy $\zeta$,
and by $\mathcal{E}_m$ the set of e-trees of size $m$. There is a
commutative diagram of projection:
\begin{equation}
   \label{eq:commproj}
\begin{psmatrix}[mnode=R,colsep=2cm,rowsep=2cm]
  \mathcal{F}_m & \mathcal{N}_m \\
  \mathcal{E}_{m} & \widetilde{\mathcal{N}}_{m} 
\end{psmatrix}
\psset{nodesep=0.3cm}
\ncLine[arrowsize=.2]{->}{1,1}{1,2}
\Aput{p}
\ncLine[arrowsize=.2]{->}{1,1}{2,1}
\ncLine[arrowsize=.2]{->}{2,1}{2,2}
\Aput{\bar{p}}
\ncLine[arrowsize=.2]{->}{1,2}{2,2}
\end{equation}
where the vertical arrows are given by forgetting the v-labels and the
horizontal by forgetting the e-labels and remembering only the leos
they induce.

The following theorem was proven in~\cite{Longyear1989b}
and~\cite{Eidswick1989} independently.  See the remarks about the
proof of Proposition~\ref{prop:pegrec}, for a proof using the theory
of pegs.

\begin{thm}
  \label{thm:pfibers}
  Two factorizations belong to the same fiber of $p$ if and only if
  they differ by a sequence of interchanges of consecutive commuting
  factors.  In particular, the set of minimal transitive
  factorizations of an $n$-cycle, up to commutation of adjacent
  factors, is in bijection with $\mathcal{N}_m$, and is therefore
  counted by $\nu_m$.
\end{thm}

A single such interchange of, say, the $i$-th and $(i+1)$-th factor,
can be effected by the action of a braid generator $\sigma_i$, and
since $\Delta_m \sigma_i = \sigma_{m-i} \Delta_m$ it follows that the
action of $\Delta_m$ on $\mathcal{F}_m$ (respectively $\mathcal{E}_m$)
descends to a map $\kappa \co \mathcal{N}_m \to \mathcal{N}_m$,
(respectively
$\tilde{\kappa} \co \widetilde{\mathcal{N}}_m \to
\widetilde{\mathcal{N}}_m$), and this map is the ``dual'' for nc-trees
defined in~\cite{Hernando1999}.  It was proved
in~\cite{Apos2018arXivApril} that $\Delta_m^2$, the central element of
$\mathrm{B}_{m}$, acts on an e-v-graph $\Gamma$ by relabeling its
vertices according to its monodromy $\mu \left( \Gamma \right)$, and
trivially on an e-graph.  Since the monodromy of an e-v-tree is a
cycle we have that $\kappa^2$ is a rotation by $2\pi/n$
radians\footnote{What we called $c$ in Remark~\ref{rem:dnact}.}, while
$\tilde{\kappa}^2 = \mathrm{id}$.  We will call $\kappa(t)$ the
\emph{complement} of $t$. The mind-body dualities $*,\bar{*}$ descend
to maps $\mathcal{N}_m \to \mathcal{N}_m^{\intercal}$ and as a
consequence of Equation~\eqref{eq:dualeqgar1} we have that
$\kappa(t) = \left( t^{*} \right)^{\intercal}$ and
$\kappa^{-1}(t) = \left( t^{\bar{*}} \right)^{\intercal}$.

We want to define (involutory) dualities
$$*,\bar{*}\co \mathcal{N}_m \to \mathcal{N}_m$$
that lift the mind-body duality for unlabeled trees, and as indicated
in Section~5.2 of~\cite{Apos2018arXivApril} this can be done by
projecting the (pullback) of mind-body duality for rooted e-trees rather
than for factorizations.  We explain that next.

For any cyclic permutation $\zeta\in \mathrm{S}_n$ there is a
bijection\footnote{Essentially due to Moszowski
  (\cite{Moszkowski1989}).}
$$
f_{\zeta}\co \mathcal{F}^{\zeta} \to \mathcal{E}_m^{*}
$$
from minimal factorizations of $\zeta$ to \emph{rooted} e-trees with
$m$ edges.  For a factorization $\rho$, $f_{\zeta}(\rho)$ is the
rooted e-tree obtained from the corresponding e-v-tree by declaring
the vertex labeled $1$ to be the root and forgetting the vertex
labels. Conversely given a rooted e-tree $t$ its monodromy is a cyclic
permutation in $\mathcal{S}_V$ and once we label the root of $t$ by
$1$ there is a unique way to label the rest of the vertices so that
$\mu(t)$ becomes $\zeta$. For example the rooted e-tree that
corresponds to our example factorization of $(1\,2\,\ldots\, 8)$ is
shown in the right side of Figure~\ref{fig:evtreeex}.

We can extend the braid action to rooted e-graphs by just letting the
root stay the same, and so we can define mind-body dualities
${*, \bar{*} \co \mathcal{E}_m^{*} \to \mathcal{E}_m^{*}}$, by
Equations~\eqref{eq:dualeqgar1} and~\eqref{eq:dualeqgar2}.  It can
be easily checked that the following diagram commutes:

\begin{equation}
  \label{eq:mbdiag}
  \begin{psmatrix}[mnode=R,colsep=2cm,rowsep=2cm]
  \mathcal{F}^{\zeta} & \mathcal{E}_m^{*}\\
  \mathcal{F}^{\zeta^{-1}} & \mathcal{E}_m^{*}
\end{psmatrix}
\psset{nodesep=0.3cm}
\ncLine[arrowsize=.2]{->}{1,1}{1,2}
\Aput{f_{\zeta}}
\ncLine[arrowsize=.2]{->}{1,1}{2,1}
\Bput{*, \bar{*}}
\ncLine[arrowsize=.2]{->}{1,2}{2,2}
\Aput{\bar{*}, *}
\ncLine[arrowsize=.2]{->}{2,2}{2,1}
\Bput{f_{\zeta^{-1}}^{-1}}
\end{equation}

By using $f_{\zeta}^{-1}$ instead of $f_{\zeta^{-1}}^{-1}$ in the
bottom row of this commutative diagram we obtain involutory
dualities $\mathcal{F}^{\zeta} \to \mathcal{F}^{\zeta}$, and we
can then project those to $\mathcal{N}_m$ to get involutory
dualities that lift the duality of unlabeled nc-trees.

\begin{defn}
  \label{defn:mbdinv}
  The \emph{nc-dualities}
  $*,\bar{*}\co \mathcal{N}_m \to \mathcal{N}_m$ are defined via the
  following commutative diagram:

\begin{equation}
  \label{eq:mbdiagnc}
  \begin{psmatrix}[mnode=R,colsep=2cm,rowsep=2cm]
     \mathcal{E}_m^{*} & \mathcal{N}_m  \\
     \mathcal{E}_m^{*} & \mathcal{N}_m 
  \end{psmatrix}
  \psset{nodesep=0.3cm}
  \ncLine[arrowsize=.2]{->}{1,1}{1,2}
  \ncLine[arrowsize=.2]{->}{1,1}{2,1}
  \Bput{*, \bar{*}}
  \ncLine[arrowsize=.2]{->}{1,2}{2,2}
  \Aput{\bar{*}, *}
  \ncLine[arrowsize=.2]{->}{2,1}{2,2}
\end{equation}
where the horizontal arrows are $p\circ f_{\zeta}^{-1}$.
From now on, unless explicitly mentioned, the term \emph{duality}, in
the context of nc-trees, will refer to these involutory dualities.
\end{defn}

It is easy to see that $t^{*} = \widebar{\kappa(t)}$, i.e. the nc-tree
obtained by reflecting $\kappa(t)$ across the diameter that passes
through the vertex labeled $1$.  For example the nc-dual of the
nc-tree on the left side of Figure~\ref{fig:compl} is shown in the
right side of Figure~\ref{fig:ncdudef}.  Since $\kappa$ has order $2n$
we have:

\begin{prop}
  \label{prop:kapparef}
  Let $r,s,\kappa \co \mathcal{N}_m \to \mathcal{N}_m$ stand for
  nc-duality, reflection across the diameter that passes through $1$,
  and the complement respectively.  Then
  $$ r = s\circ \kappa $$
  and therefore the group generated by the involutions $r,s$ is
  isomorphic to $\mathrm{D}_{2n}$ the dihedral group with $4n$
  elements.
\end{prop}

It turns out that $\mathcal{N}$ is a free $*$-magma, generated by the
nc-tree $\lambda$ consisting of a single vertex labeled $1$ and no
edges. This will be exposed in Section~\ref{sec:pmagma} after we have
introduced in Section~\ref{sec:pcdd} the set $\mathcal{P}$ of
\emph{flagged Perfectly Chain Decomposed Ditrees}.

We end this subsection by introducing a bijection analogous to
$f_{\zeta}$ between labeled nc-trees and a class of ordered
trees in the next section, and explaining the connection of $\kappa$
with the Kreweras complement in Section~\ref{sec:kreweras}.

\subsubsection{The set of Bipartisan Trees}
\label{sec:bipart}
Recall that an \emph{ordered tree} is a rooted tree where the children
of every vertex have been given a linear order.

\begin{defn}
  \label{defn:bipart}
  Let $v$ be a non-root vertex in a rooted tree.  The \emph{trunk} of $v$
  is the last edge $e_v$ in the unique path from the root to $v$.

  A \emph{bipartisan tree} is an ordered tree, where the children of
  every vertex are partitioned into two classes, the \emph{left
    children} and the \emph{right children}, in such a way that the
  right children are less than the left children.

  The set of bipartisan trees with $m$ edges is denoted by $\mathcal{BO}_m$.
\end{defn}

\begin{prop}
  \label{prop:bipart}
  There is a bijection $f\co \mathcal{N}_m \to \mathcal{BO}_m$.
\end{prop}
\begin{proof}
  Projecting Moszkowski's $f_{\zeta}$ gives a bijection from the set
$\mathcal{N}_m$ to the set of rooted trees with leos.  Now given a
rooted tree with a leo, for a vertex $v$ with trunk $e_v$ let
$e_{l_1} < \cdots <e_{l_k} < e_v < e_{r_1} < \cdots e_{r_s}$ be the
edges incident to $v$ ordered according to the leo at $v$.  Call
$l_1,\ldots, l_k$ the left children and $r_1,\ldots, r_s$ the right children
and put them in the order $r_1,\ldots,r_s,l_1,\ldots,l_k$ to obtain a
bipartisan tree $f(t)$.

Conversely, given a bipartisan tree $t$ we obtain the leo of a
non-root vertex by declaring the trunks of the left children are
ordered according to the order of the ordered tree, followed by the
trunk of $v$, and then by the trunks of the right children again in
the order given by the ordered tree. The edges incident to the root
are ordered according to the order of their endpoints.
\end{proof}

See Figure~\ref{fig:bipartex} as an example where the bipartisan tree
corresponding to the nc-tree of Figure~\ref{fig:exncunl} is shown, and
the leo structure at every vertex is indicated by oriented arcs around
that vertex.

\begin{figure}[ht]
  \centering
        \psset{unit=1.3}
      \begin{pspicture}(-1.30000, -0.30000)(2.80000, 3.30000)
        \pnode(1.50000,0.00000){1}
        \psdot(1.50000,0.00000)
        \pnode(2.50000,1.00000){2}
        \psdot(2.50000,1.00000)
        \pnode(1.50000,1.00000){3}
        \psdot(1.50000,1.00000)
        \pnode(0.50000,1.00000){4}
        \psdot(0.50000,1.00000)
        \pnode(0.50000,3.00000){5}
        \psdot(0.50000,3.00000)
        \pnode(0.00000,2.00000){6}
        \psdot(0.00000,2.00000)
        \pnode(-0.50000,3.00000){7}
        \psdot(-0.50000,3.00000)
        \pnode(-1.00000,2.00000){8}
        \psdot(-1.00000,2.00000)
        \ncline{1}{2}
        \ncline{1}{3}
        \ncline{1}{4}
        \ncline{4}{6}
        \ncline{4}{8}
        \ncline{5}{6}
        \ncline{6}{7}
        \psarc[linecolor=red,linewidth=.01]{->}(0,2){.2}{116.57}{423.57}
        \psarc[linecolor=red,linewidth=.01]{->}(.5,1){.2}{116.57}{315}
        \psarc[linecolor=red,linewidth=.01]{->}(1.5,0){.2}{45}{135}
      \end{pspicture}
  \caption{The bipartisan tree representing the nc-tree of Figure~\ref{fig:exncunl}.}
  \label{fig:bipartex}
\end{figure}
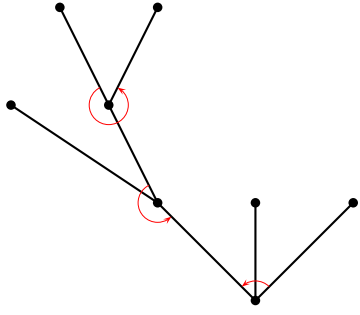

\subsubsection{ The Kreweras complement on the lattice of non-crossing
  partitions}
\label{sec:kreweras}
The lattice of non-crossing partitions is well studied and we refer
the reader to~\cite{Armstrong2009} for the basic definitions and the
extensive bibliography.  In this subsection we show that the map
$\kappa\co \mathcal{N}_m\to \mathcal{N}_m$ induced by the action of
the Garside element on the set of rooted e-trees, has an
interpretation in terms of the Kreweras complement on the lattice
$\mathcal{NC}_n$ of non-crossing partitions of a set of $n$ elements,
where as usual $n=m+1$.

Let $\mathcal{G} = \mathrm{Cay}(\mathrm{S}_n,T)$ be the Calyley graph
of the symmetric group with respect to the generating set $T$ of all
transpositions.  Define $\pi_1 \le \pi_2$ if there is a geodesic path
(with respect to the word length metric) in $\mathcal{G}$ from the
identity to $\pi_2$ that passes through $\pi_1$. This gives a partial
order in $\mathrm{S}_n$ called the \emph{strong order}. The lattice of
non-crossing partitions $\mathcal{NC}_n$ is (isomorphic to) the
interval $[ \mathrm{id}, \zeta] \subset \mathrm{S}_n$ in the strong
order.  This is a complemented lattice and one of its complements, the
so called Kreweras complement, is given by the formula\footnote{This
  is different than the formula in~\cite{Armstrong2009}.  The
  inconsistency is due to different conventions on how to multiply
  transpositions and how exactly the braid group acts.  With our
  conventions Armstrong's formula would give $K^{-1}$ that corresponds
  to the action of $\Delta_m^{-1}$. Of course, $K^{-1}$ is also a
  complement in the lattice.}

$$K(\pi) = \zeta \pi^{-1}.$$

There is a bijection between $\mathcal{C}$, the set of maximal
increasing chains in $\mathcal{NC}_n$ and $\mathcal{F}_m$ the set
minimal transitive factorizations of $\zeta$.  Indeed they both
determine a geodesic path (i.e. a path of minimal distance) in
$\mathcal{G}$ from $\mathrm{id}$ to $\zeta$, and the labels of the
vertices of that path give a maximal chain
$ c = \left( \mathrm{id} = \pi_0 < \pi_1 < \ldots < \pi_m = \zeta
\right)$, while the labels of the edges give a factorization
$\rho = \tau_1, \ldots, \tau_m$.  More precisely, we have two inverse
bijections:
\begin{align*}
  \partial \co \mathcal{C} \to \mathcal{F}_m, \quad&  c  \mapsto \pi_0^{-1}\pi_1, \pi_1^{-1}\pi_2, \ldots, \pi_{m-1}^{-1}\pi_m \\
  \int \co \mathcal{F}_{m} \to \mathcal{C}, \quad & \rho  \mapsto \mathrm{id}, \tau_1, \tau_1\tau_2, \ldots, \tau_1\tau_2\ldots\tau_m.
\end{align*}
Since $K$ is an anti-isomorphism of $\mathcal{NC}_n$ it maps maximal
increasing chains to maximal decreasing chains and we can define a map
$$
\kappa \co \mathcal{C} \to \mathcal{C}, \quad \kappa(c) = K(\pi_m), K(\pi_{m-1}),\ldots,K(\pi_1),K(\pi_0).
$$

It turns out that the action of the Garside element $\Delta_m$ is given by
$\kappa$ interpreted as a map between factorizations.

\begin{prop}
  \label{prop:kreweqdel} For a factorization $\rho$ we have
$$ \rho \Delta = \partial \kappa \left( \int \rho \right). $$
\end{prop}
\begin{proof}
  Let $\int \rho = \pi_{0}, \pi_1,\ldots,\pi_m$ then (see Equation~\ref{eq:dualeqgar1})
  we have that
  $$ \rho \Delta = \prescript{\pi_{m-1}}{}\tau_m, \prescript{\pi_{m-2}}{}\tau_{m-1},
  \ldots, \prescript{\pi_1}{}\tau_2, \prescript{\pi_o}{}\tau_1$$
  and so since $\pi_j =  \pi_{j-1}t_j$ for $j=1,\ldots,m$ we see that
  $$ \int \rho \Delta = \mathrm{id}, \prescript{\pi_{m-1}}{}\tau_m,
  \prescript{\pi_{m-2}}{} (\tau_{m-1}\tau_{m}), \ldots,
   \prescript{\pi_1}{} ( \tau_2\ldots \tau_m),
   \prescript{\pi_0}{} ( \tau_1\ldots \tau_m) = \pi_m.
   $$
   Now since $\zeta = \pi_m = \tau_1\ldots \tau_m$ we have that
   $$ K(\pi_j) = \tau_1\ldots \tau_m \tau_j\ldots \tau_1
   = \prescript{\pi_0}{} ( \tau_{j+1}\ldots \tau_m).$$
   Thus
   $$ \int \rho \Delta = \kappa \left( \int \rho \right) $$
   as we needed.
 \end{proof}

\subsection{Perfectly Chain Decomposed Ditrees}
\label{sec:pcdd}

The \emph{medial digraph} of a peg $\Gamma$ is the analogue of medial
graphs in the theory of cellularly embedded graphs.  Essentially the
medial digraph of of $\Gamma$ is the digraph
$\mathcal{M} \left( \Gamma \right)$ obtained by putting together the
Hasse diagrams of all the local edge orderings: its vertices are the
edges of $\Gamma$ and there is an edge from $e_1$ to $e_2$ if and only
if $e_1 \le e_2$ in the leo of a vertex of $\Gamma$. Each edge of
$\Gamma$ is incident to two vertices and is preceded (or followed) by
at most one edge at the leo of each of those vertices. It follows that
the in and out degrees at every vertex of the medial digraph is at
most $2$.  Conversely, every digraph that satisfies these degree
restrictions is the medial digraph of a peg, see
Item~\ref{item:pegrec} in Proposition~\ref{prop:pegrec}. For example
in the top of Figure~\ref{fig:pcddex} we see the medial digraphs of
the pair of dual pegs of Figure~\ref{fig:mainex}. Notice that the two
medial digraphs are isomorphic and this is true in general: the map
$e \mapsto e^{*}$ defines an isomorphism between the medial digraphs
of dual pegs.  The local linear order at the star of each vertex gives
a chain in the medial digraph, and in Figure~\ref{fig:pcddex} the
chains that come from different vertices are indicated by different
colors. We remark that the peg can be can be reconstructed from
its medial digraph once this decomposition into chains is known.  This
observation is important for what follows so we develop it in some
detail.

\begin{defn}
  \label{defn:PCD}   
  A \emph{medial digraph} is a digraph with the in and out degrees of
  all vertices at most two. A \emph{Perfect Chain Decomposition}
  (\emph{PCD} for short) of a medial digraph is a decomposition
  $\mathcal{C}$ of its edges into chains with the property that every
  vertex belongs to exactly two chains.  We emphasize that chains of
  length zero consisting of a single vertex are
  allowed\footnote{Actually when $\Gamma$ is a tree they are
    required!}.  For a chain $c\in \mathcal{C}$ we use the notation
  $\alpha(c)$ (resp. $\omega(c)$) to stand for the first (resp. last)
  vertex of $c$.

  A vertex of a medial digraph is called \emph{internal} if both its
  in and out degree are at least $1$. Notice that constructing a PCD
  on a medial digraph $d$ involves a binary choice at every internal
  vertex, namely which incoming edge to connect to which outgoing
  edge.  The \emph{dual} $\mathcal{C}^{*} $ of a PCD $\mathcal{C}$ is
  the PCD obtained from $\mathcal{C}$ when the opposite choice of such
  connections is made at every internal vertex. 
\end{defn}

The following summarizes the main results for PCDs on medial
digraphs from~\cite{Apos2018arXivApril}:

\begin{prop}
  \label{prop:pegrec}
  We have:
  \begin{enumerate}
  \item \label{item:medchi}
    The Euler characteristic of $\mathcal{M}\left( \Gamma \right)$
    equals the Euler characteristic of $\Gamma$. In particular for
    a non-crossing tree $t$ we have that the underlying
    graph of $\mathcal{M}(t)$ is a tree.
  \item \label{item:pegrec}
    The leo of a peg $\Gamma$ induces a PCD on its medial digraph
    $\mathcal{M}\left( \Gamma \right)$, and the peg can be reconstructed
    from that PCD.
  \item Mind-body dual pegs have isomorphic medial digraphs and they
    induce dual PCDs.
  \item \label{item:dag}
    A peg $\Gamma$ comes from a factorization if and only if its
    medial digraph is a dag.\footnote{Directed Acyclic Graph.  This
      observation is essentially due to~\cite{DulPen1993}.  The
      definition of medial digraphs was inspired in part from that
      paper.}  In particular, by Item~\ref{item:medchi}, any leo on a
    tree comes from a factorization.
  \end{enumerate}
\end{prop}

\begin{proof}[Remarks on the proof]
  For detailed proofs consult~\cite{Apos2018arXivApril}.  Regarding
  Item~\ref{item:pegrec}, the peg that corresponds to a PCD
  $\mathcal{C}$ on a medial digraph $d$ has a vertex $v_c$ for any
  chain $c \in \mathcal{C}$ and each vertex $w$ of $d$ gives an edge
  $e_w$ connecting $v_{c_1}$ and $v_{c_2}$ where the $c_1$ and $c_2$
  are the two chains that $w$ belongs to.  Clearly an edge $e_w$
  belongs to the star of a vertex $v_{c}$ if and only if $w$ is
  contained in $c$, and so the order of the vertices of $c$ gives a
  linear order at the star of each vertex endowing the resulting graph
  with a leo.

  In the bottom half of Figure~\ref{fig:pcddex} we see the PCDs on the
  medial ditrees\footnote{Recall that a ditree is a digraph whose
    underlying graph is a tree.} of the pair of mind-body dual
  nc-trees of Figure~\ref{fig:mainex}.  When drawing medial ditrees we
  omit arrows and use the convention that all edges are directed
  upwards, and we follow the same convention when we draw the chains
  of a PCD.

  Regarding Item~\ref{item:dag}, notice that the edges of a peg that
  comes from a factorization are totally ordered by their labels, and
  that order gives a \emph{topological sort} in its medial
  digraph\footnote{That is a \emph{linear extension} of the poset
    whose Hasse diagram is the dag.}. Actually the set of all
  possible factorizations (up to conjugation) that give that peg is in
  bijection with the set of topological sorts of its medial digraph.
  We remark that it is relatively easy to prove (see for
  example~\cite{Ruskey1992}) that any two topological sorts of a dag
  differ by a sequence of adjacent transpositions, and this can be
  used to prove Theorem~\ref{thm:pfibers}.
\end{proof}

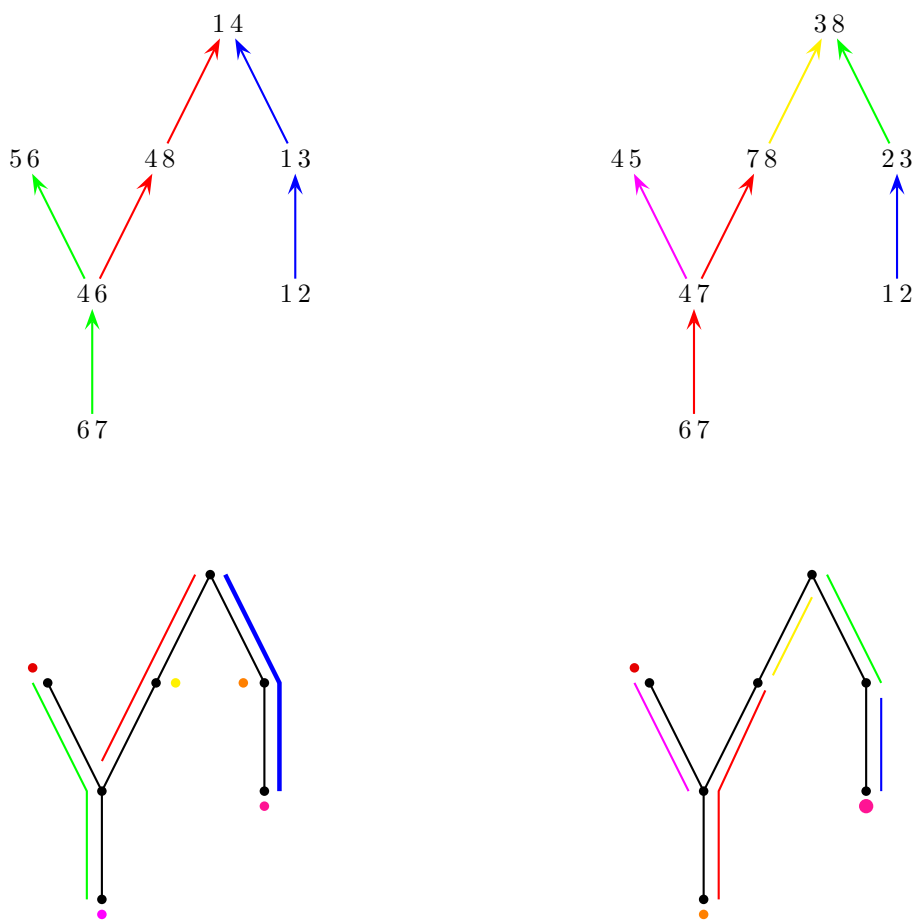
\begin{figure}[ht]
  \begin{pspicture}(-7,-9.5)(7,3.5)
    \rput(-4,0){
      \psset{unit=.025,arrowsize=8}
      \begin{pspicture}(26.70000, 17.70000)(171.30000, 234.30000)
        \psnode(63.00000,18.00000){67}{\small $6\,7$}
        \psnode(63.00000,90.00000){46}{\small $4\,6$}
        \psnode(27.00000,162.00000){56}{\small $5\,6$}
        \psnode(99.00000,162.00000){48}{\small $4\,8$}
        \psnode(171.00000,90.00000){12}{\small $1\,2$}
        \psnode(171.00000,162.00000){13}{\small $1\,3$}
        \psnode(135.00000,234.00000){14}{\small $1\,4$}
        \ncline[nodesep=2.5pt,linecolor=blue]{->}{12}{13}
        \ncline[nodesep=2.5pt,linecolor=blue]{->}{13}{14}
        \ncline[nodesep=2.5pt,linecolor=red]{->}{46}{48}
        \ncline[nodesep=2.5pt,linecolor=green]{->}{46}{56}
        \ncline[nodesep=2.5pt,linecolor=red]{->}{48}{14}
        \ncline[nodesep=2.5pt,linecolor=green]{->}{67}{46}
      \end{pspicture}}
    \rput(4,0){
      \psset{unit=.025,arrowsize=8}
      \begin{pspicture}(26.70000, 17.70000)(171.30000, 234.30000)
        \psnode(63.00000,18.00000){67}{\small $6\,7$}
        \psnode(63.00000,90.00000){46}{\small $4\,7$}
        \psnode(27.00000,162.00000){56}{\small $4\,5$}
        \psnode(99.00000,162.00000){48}{\small $7\,8$}
        \psnode(171.00000,90.00000){12}{\small $1\,2$}
        \psnode(171.00000,162.00000){13}{\small $2\,3$}
        \psnode(135.00000,234.00000){14}{\small $3\,8$}
        \ncline[nodesep=2.5pt,linecolor=blue]{->}{12}{13}
        \ncline[nodesep=2.5pt,linecolor=green]{->}{13}{14}
        \ncline[nodesep=2.5pt,linecolor=red]{->}{46}{48}
        \ncline[nodesep=2.5pt,linecolor=magenta]{->}{46}{56}
        \ncline[nodesep=2.5pt,linecolor=yellow]{->}{48}{14}
        \ncline[nodesep=2.5pt,linecolor=red]{->}{67}{46}
      \end{pspicture}}
    \rput(-4,-6.8){%
      \psset{unit=.02,arrowsize=9}
      \begin{pspicture}(26.70000, 17.70000)(171.30000, 234.30000)
        \pnode(63.00000,18.00000){1}
        \psdot(63.00000,18.00000)
        \pnode(63.00000,90.00000){2}
        \psdot(63.00000,90.00000)
        \pnode(27.00000,162.00000){3}
        \psdot(27.00000,162.00000)
        \pnode(99.00000,162.00000){4}
        \psdot(99.00000,162.00000)
        \pnode(171.00000,90.00000){5}
        \psdot(171.00000,90.00000)
        \pnode(171.00000,162.00000){6}
        \psdot(171.00000,162.00000)
        \pnode(135.00000,234.00000){7}
        \psdot(135.00000,234.00000)
        \ncline{1}{2}
        \ncline{2}{3}
        \ncline{2}{4}
        \ncline{4}{7}
        \ncline{5}{6}
        \ncline{6}{7}
        \psline[linecolor=green,arcangle=25](53.00000,18.00000)(53.00000,90.00000)(17.00000,162.00000)
        \psline[linecolor=red](63.00000,110.00000)(89.00000,162.00000)(125.00000,234.00000)
        \psline[linecolor=blue,arcangle=25,linewidth=3](181.00000,90.00000)(181.00000,162.00000)(145.00000,234.00000)
        \psdot[linecolor=magenta](63.00000,8.00000)
        \psdot[linecolor=yellow](112.00000,162.00000)
        \psdot[linecolor=deeppink](171.00000,80.00000)
        \psdot[linecolor=orange](157.00000,162.00000)
        \psdot[linecolor=darkred](17.00000,172.00000)
      \end{pspicture}}
    \rput(4,-6.8){%
      \psset{unit=.02,arrowsize=9}
      \begin{pspicture}(26.70000, 17.70000)(171.30000, 234.30000)
        \pnode(63.00000,18.00000){1}
        \psdot(63.00000,18.00000)
        \pnode(63.00000,90.00000){2}
        \psdot(63.00000,90.00000)
        \pnode(27.00000,162.00000){3}
        \psdot(27.00000,162.00000)
        \pnode(99.00000,162.00000){4}
        \psdot(99.00000,162.00000)
        \pnode(171.00000,90.00000){5}
        \psdot(171.00000,90.00000)
        \pnode(171.00000,162.00000){6}
        \psdot(171.00000,162.00000)
        \pnode(135.00000,234.00000){7}
        \psdot(135.00000,234.00000)
        \ncline{1}{2}
        \ncline{2}{3}
        \ncline{2}{4}
        \ncline{4}{7}
        \ncline{5}{6}
        \ncline{6}{7}
        \psline[linecolor=red,arcangle=25](73.00000,18.00000)(73.00000,90.00000)(104.00000,157.00000)
        \psline[linecolor=yellow](109.00000,167.00000)(135.00000,219.00000)
        \psline[linecolor=magenta,arcangle=25](53.00000,90.00000)(17.00000,162.00000)
        \psline[linecolor=blue](181.00000,90.00000)(181.00000,152.00000)
        \psline[linecolor=green](181.00000,162.00000)(145.00000,234.00000)
        \psdot[linecolor=orange](63.00000,8.00000)
        \psdot[linecolor=deeppink,linewidth=3](171.00000,80.00000)
        \psdot[linecolor=darkred](17.00000,172.00000)
      \end{pspicture}}    
  \end{pspicture}
  \caption{The PCDDs of the non-crossing trees of Figure~\ref{fig:mainex}.}
  \label{fig:pcddex}
\end{figure}

By Proposition~\ref{prop:pegrec} we can encode unlabeled nc-trees and
their duality with medial ditrees endowed with a PCD.  This encoding
can be extended to labeled nc-trees by encoding one additional
piece of information: which chain of the PCD corresponds to the
vertex labeled $1$.

\begin{defn}
  \label{defn:PCDD}   
  A \emph{Perfectly Chain Decomposed Ditree} (\emph{PCDD} for short)
  is a medial ditree endowed with a PCD and a \emph{Flagged Perfectly
    Chain Decomposed Ditree} is a PCDD endowed with a distinguished
  chain called its \emph{flag}.

  We will use the same symbol (typically $d$) to denote the PCDD and
  its underlying medial ditree, and in that case the flag will be
  denoted by $f(d)$. For a chain $c\in \mathcal{C}$ we use the
  notation $\alpha(c)$ (resp. $\omega(c)$) to stand for the first
  (resp. last) vertex of $c$, and for a flagged PCDD $d$ we use the
  notation $\alpha(d)$ and $\omega(d)$ to stand for
  $\alpha \left( f(d) \right)$ and $\omega \left( f(d) \right)$
  respectively.

  The set of flagged PCDDs with $m$ vertices will be denoted by
  $\mathcal{P}_m$ and the set of (unflagged) PCDDs with $m$ vertices
  by $\widetilde{\mathcal{P}}_m$ and we let $\mathcal{P} = \bigcup_{m\ge 0} \mathcal{P}_m$,
  and $\widetilde{\mathcal{P}} = \bigcup_{m\ge 0} \widetilde{\mathcal{P}}_m$.

  The \emph{reverse} $\bar{d}$ of a PCDD $d$ is the PCDD whose underlying ditree
  is the reverse ditree, its chains are the reverses of the chains of $d$, and
  its flag is the reverse of the flag of $d$.
  
  For a flagged PCDD $d$, $d^{*}$ is also flagged and its flag $f^{*}$
  is defined as follows: $\alpha(f^{*}) = \alpha(f)$ and if $f$ is the
  only chain that starts at $\alpha(f)$ then $f^{*}$ is the only chain
  of $d^{*}$ that starts at $\alpha(f)$, otherwise the first edge of
  $f^{*}$ is the outgoing edge incident at $\alpha(f)$ that does not
  belong to $f$, if no such edge exist then $f^{*}$ is a trivial
  chain. All possible local configurations are shown in
  Figure~\ref{fig:dualflag}, the flags of the relevant PCDs are shown
  in red.

  We extend the definition of PCDD to include the following two
  degenerate\footnote{The first one may even be called pointless.}
  cases that correspond to the nc-trees with $0$ and $1$ edges:
  \begin{itemize}
  \item   The \emph{empty PCDD} $\lambda$ is the triple
    $\left( \emptyset, \left\{ \emptyset \right\}, \emptyset \right)$
    consisting of the empty ditree, the perfect chain decomposition
    consisting of the empty chain, and the empty chain as flag. The 
    functions $\alpha$ and $\omega$ are not defined for the empty
    flag, and therefore not for $\lambda$ either.
  \item  The \emph{point PCDD} $\mathfrak{p}$ is the triple
    $\left( p, \left\{ {p}, {p} \right\}, {p} \right)$, consisting of a
    ditree with one vertex and no edges, a chain decomposition
    consisting of two identical trivial chains, and the unique chain as
    a flag.
  \end{itemize}
\end{defn}

\begin{figure}[ht]
  \centering
  \psset{unit=.7, arrowsize=0.2}
  \begin{pspicture}(-3.5,-1.2)(11,4.8)
    \psdot(0,-.5)
    \psline(0,.5)(0,-.5)
    \psdot[linecolor=red](-.2,-.7)
    \psline[linecolor=blue](-.2,-.5)(-.2,.5)
    \psline{<->}(.5,0)(1.5,0)
    \uput[90](1,0){$\tiny *$}
    \rput(2,0){%
      \psdot(0,-.5)
    \psline(0,-.5)(0,.5)
    \psdot[linecolor=blue](.2,-.7)
    \psline[linecolor=red](.2,-.5)(.2,.5)}
    \rput(5.8,0){%
    \psdot(0,-.5)
     \psline(-.6,.5)(0,-.5)(.6,.5)      
     \psline[linecolor=red](-.8,.5)(-.2,-.5)
     \psline[linecolor=blue](.8,.5)(.2,-.5)}
    \psline{<->}(7,0)(8,0)
    \uput[90](7.5,0){$\tiny *$}
    \rput(9.2,0){%
    \psdot(0,-.5)
     \psline(-.6,.5)(0,-.5)(.6,.5)      
     \psline[linecolor=blue](-.8,.5)(-.2,-.5)
     \psline[linecolor=red](.8,.5)(.2,-.5)}
   \rput(-3,2.25){%
     \psdot(0,1)
     \psline(0,0)(0,1)
    \psdot[linecolor=red](-.2,1.2)
    \psline[linecolor=blue](-.2,0)(-.2,1)}
    \psline{<->}(-2.5,2.75)(-1.5,2.75)
    \uput[90](-2,2.75){$\tiny *$}
   \rput(-1,2.25){%
     \psdot(0,1)
     \psline(0,0)(0,1)
    \psdot[linecolor=red](-.2,1.2)
    \psline[linecolor=blue](-.2,0)(-.2,1)}
  \rput(1.5,2){%
  \psdot(0,1)
  \psline(0,2)(0,0)
  \psline[linecolor=blue](.2,2)(.2,0)
  \psdot[linecolor=red](-.2,1)}
  \psline{<->}(2,3)(3,3)
  \uput[90](2.5,3){$\tiny *$}
  \rput(3.5,2){%
  \psdot(0,1)
  \psline(0,2)(0,0)
  \psline[linecolor=blue](.2,.9)(.2,0)
  \psline[linecolor=red](.2,1.1)(.2,2)}
  \rput(7,2){%
    \psdot(0,1)
    \psline(-.6,2)(0,1)(0,0)      
    \psline(.6,2)(0,1)
    \psline[linearc=.25,linecolor=blue](-.2,0)(-.2,1)(-.8,2)
    \psline[linecolor=red](.8,2)(.2,1)}
  \psline{<->}(8,3)(9,3)
  \uput[90](8.5,3){$\tiny *$}
  \rput(10,2){%
    \psdot(0,1)
    \psline(-.6,2)(0,1)(0,0)      
    \psline(.6,2)(0,1)
    \psline[linecolor=red](-.8,2)(-.2,1)
    \psline[linearc=.25,linecolor=blue](.8,2)(.2,1)(.2,0)}
\end{pspicture}
  \caption{The flag of the dual of a flagged PCD.}
  \label{fig:dualflag}
\end{figure}
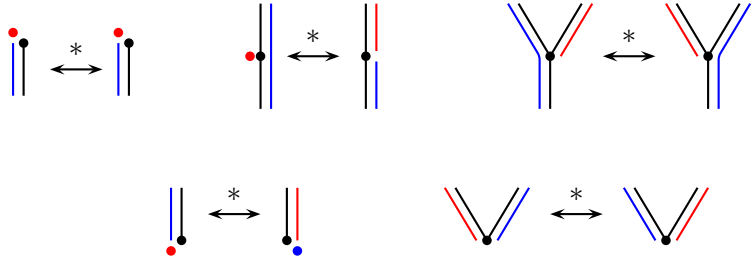

We summarize the above discussion in the following theorem, for more
details see Section~5 of~\cite{Apos2018arXivApril}.

\begin{thm}
  \label{thm:medialbij}
  The function
  $$ \mathcal{M}\co \mathcal{N}_{m} \to \mathcal{P}_m $$
  that assigns to an nc-tree $t$ its medial ditree endowed with the PCD
  induced by the leo of $t$ and having as flag the chain that corresponds
  to the leo of the vertex labeled $1$ is a duality preserving bijection.
\end{thm}

We now exhibit $\mathcal{P}$ as a free $*$-magma.  In what
follows PCDD will always mean a \emph{flagged PCDD}.

\begin{defn}
  \label{defn:fusion2} 
  Let $d_{\mathrm{l}}, d_{\mathrm{m}}, d_{\mathrm{r}}$ be PCDDs.
  Their \emph{fusion} is defined to be the PCDD
  $\Upsilon(d_{\mathrm{l}}, d_{\mathrm{m}}, d_{\mathrm{r}})$ where:
  \begin{itemize}
  \item The underlying ditree has vertices the (disjoint) union of the
    vertices of $d_{\mathrm{l}}, d_{\mathrm{m}}, d_{\mathrm{r}}$, plus
    a new vertex $v_0$.  The edges are the edges of
    $d_{\mathrm{l}}, \widebar{d}_{\mathrm{m}}, d_{\mathrm{r}}$ plus,
    provided that the corresponding flags are not empty, edges
    connecting $v_{0}$ to $\alpha(d_{\mathrm{l}})$ and
    $\alpha(d_{\mathrm{r}})$ and an edge connecting
    $\omega(\bar{d}_{\mathrm{m}})$ to $v_0$.
  \item The chains are the non-flag chains of $d_{\mathrm{l}}$,
    $\widebar{d}_{\mathrm{m}}$ and $d_{\mathrm{r}}$, and two
    additional chains:
    $f(\widebar{d}_{\mathrm{m}}) \to v_0 \to f(d_{\mathrm{r}})$, and
    $v_0 \to f(d_{\mathrm{l}})$.
  \item The flag is $v_0 \to f \left( d_{\mathrm{l}} \right)$.
  \end{itemize}
\end{defn}

Notice that with this definition
$\Upsilon(\lambda,\lambda,\lambda) = \mathfrak{p}$.  A less trivial
example of the fusion of three PCDDS is shown in
Figure~\ref{fig:fusion2}, the flag of each PCDD is indicated in red.

\begin{figure}[ht]
  \centering
\psset{unit=.6}
  \begin{pspicture}(-8.2,-4.4)(14.9,5.4)
    \rput(-3,0){%
  \begin{pspicture}(-3.5,-2.5)(6,1)
    \pnode(0,-2){m1}
    \psdot(0,-2)
    \pnode(0,-1){m2}
    \psdot(0,-1)
    \pnode(-1,0){m3}
    \psdot(-1,0)
    \pnode(1,0){m4}
    \psdot(1,0)
    \ncline{m3}{m2}
    \ncline{m4}{m2}
    \ncline{m2}{m1}
    \psline[linecolor=red,linearc=.15](-.2,-2)(-.2,-1)(-1.2,0)
    \psline[linecolor=blue](.2,-1)(1.2,0)
    \psdot[linecolor=blue](-1.2,.2)
    \psdot[linecolor=blue](1.2,.2)
    \psdot[linecolor=blue](0,-2.2)
    \uput[-90](0,-2.2){$d_{\mathrm{m}}$}
    \pnode(-2,2){l1}
    \psdot(-2,2)
    \pnode(-2,1){l2}
    \psdot(-2,1)
    \pnode(-3,3){l3}
    \psdot(-3,3)
    \pnode(-1,3){l4}
    \psdot(-1,3)
    \pnode(-1,4){l5}
    \psdot(-1,4)
    \ncline{l2}{l1}
    \ncline{l1}{l4}
    \ncline{l1}{l3}
    \ncline{l4}{l5}
    \psdot[linecolor=blue](-2,.8)
    \psdot[linecolor=blue](-3.2,3.2)
    \psdot[linecolor=blue](-1,4.2)
    \psdot[linecolor=blue](-1.2,3.2)
    \psline[linecolor=blue,linearc=.15](-2.2,1)(-2.2,2)(-3.2,3)
    \psline[linecolor=red,linearc=.15](-1.8,2)(-.8,3)(-.8,4)
    \uput[-90](-2,.8){$d_{\mathrm{l}}$}
    \pnode(2,1){r1}
    \psdot(2,1)
    \pnode(3,2){r2}
    \psdot(3,2)
    \pnode(4,1){r3}
    \psdot(4,1)
    \pnode(4,3){r4}
    \psdot(4,3)
    \pnode(2,3){r5}
    \psdot(2,3)
    \pnode(3,4){r6}
    \psdot(3,4)
    \pnode(5,4){r7}
    \psdot(5,4)
    \ncline{r1}{r2}
    \ncline{r3}{r2}
    \ncline{r2}{r4}
    \ncline{r2}{r5}
    \ncline{r4}{r6}
    \ncline{r4}{r7}
    \psdot[linecolor=blue](1.8,.8)
    \psdot[linecolor=blue](4.2,.8)
    \psdot[linecolor=blue](1.8,3.2)
    \psdot[linecolor=blue](2.8,4.2)
    \psdot[linecolor=blue](5.2,4.2)
    \psline[linecolor=blue,linearc=.15](4.2,1)(3.2,2)(5.2,4)
    \psline[linecolor=blue,linearc=.15](3.8,3)(2.8,4)
    \psline[linecolor=red,linearc=.15](1.8,1)(2.8,2)(1.8,3)
    \uput[-90](3,1){$d_{\mathrm{r}}$}
  \end{pspicture}}
   \rput(10,0){%
  \begin{pspicture}(-3.5,-2.5)(6,3)
    \pnode(0,0){m1}
    \psdot(0,0)
    \pnode(0,-1){m2}
    \psdot(0,-1)
    \pnode(-1,-2){m3}
    \psdot(-1,-2)
    \pnode(1,-2){m4}
    \psdot(1,-2)
    \ncline{m3}{m2}
    \ncline{m4}{m2}
    \ncline{m2}{m1}
    \psline[linecolor=blue](1.2,-2)(.2,-1)
    \psdot[linecolor=blue](-1.2,-2.2)
    \psdot[linecolor=blue](1.2,-2.2)
    \pnode(-2,3){l1}
    \psdot(-2,3)
    \pnode(-2,2){l2}
    \psdot(-2,2)
    \pnode(-3,4){l3}
    \psdot(-3,4)
    \pnode(-1,4){l4}
    \psdot(-1,4)
    \pnode(-1,5){l5}
    \psdot(-1,5)
    \ncline{l2}{l1}
    \ncline{l1}{l4}
    \ncline{l1}{l3}
    \ncline{l4}{l5}
    \psdot[linecolor=blue](-2,1.8)
    \psdot[linecolor=blue](-3.2,4.2)
    \psdot[linecolor=blue](-1,5.2)
    \psdot[linecolor=blue](-1.2,4.2)
    \psline[linecolor=blue,linearc=.15](-2.2,2)(-2.2,3)(-3.2,4)
    \pnode(2,2){r1}
    \psdot(2,2)
    \pnode(3,3){r2}
    \psdot(3,3)
    \pnode(4,2){r3}
    \psdot(4,2)
    \pnode(4,4){r4}
    \psdot(4,4)
    \pnode(2,4){r5}
    \psdot(2,4)
    \pnode(3,5){r6}
    \psdot(3,5)
    \pnode(5,5){r7}
    \psdot(5,5)
    \ncline{r1}{r2}
    \ncline{r3}{r2}
    \ncline{r2}{r4}
    \ncline{r2}{r5}
    \ncline{r4}{r6}
    \ncline{r4}{r7}
    \psdot[linecolor=blue](4.2,1.8)
    \psdot[linecolor=blue](1.8,4.2)
    \psdot[linecolor=blue](2.8,5.2)
    \psdot[linecolor=blue](5.2,5.2)
    \psline[linecolor=blue,linearc=.15](4.2,2)(3.2,3)(5.2,5)
    \psline[linecolor=blue,linearc=.15](3.8,4)(2.8,5)
    \pnode(0,1){v}
    \psdot[linecolor=magenta](0,1)
    \ncline{m1}{v}
    \ncline{v}{l1}
    \ncline{v}{r1}
    \psdot[linecolor=blue](1.8,2.2)
    \psdot[linecolor=blue](.3,0)
    \psline[linecolor=blue,linearc=.15](-1.2,-2)(-.2,-1)(-.2,.8)(.2,1)(2,1.8)(3,2.8)(1.8,4)
    \psline[linecolor=red,linearc=.15](0,1.2)(-1.8,3)(-.8,4)(-.8,5)
  \end{pspicture}}
   \uput[-90](9,-2.7){$\Upsilon \left( d_{\mathrm{l}}, d_{\mathrm{m}}, d_{\mathrm{r}} \right)$}
  \end{pspicture}
  \caption{An example of the fusion of PCDDs.}
  \label{fig:fusion2}
\end{figure}
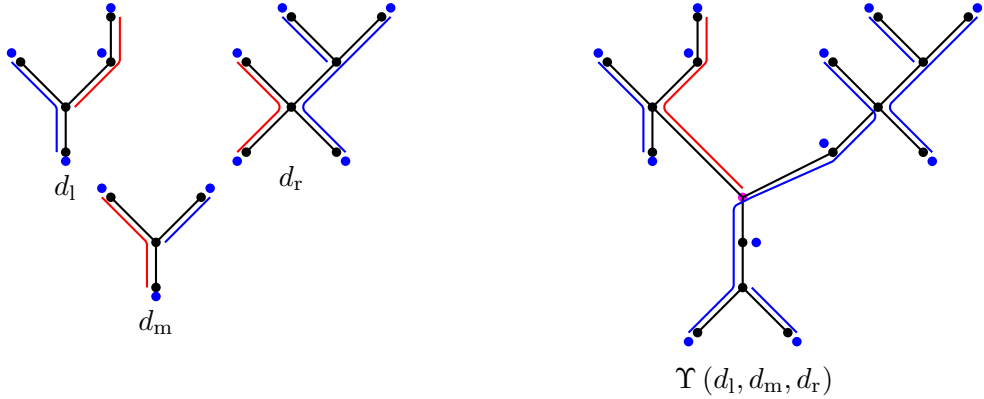

\begin{thm}
  \label{thm:PCDupsilon}
  With the above definitions $\mathcal{P}$ is a free $*$-magma.
\end{thm}

\begin{proof}
  Starting with a non-empty PCDD $d$ and removing $\alpha(d)$ we
  obtain three PCDDs: $d_{\mathrm{l}}$ induced by those vertices of
  $d$ that are above $\alpha(d)$ and were connected to $\alpha(d)$ by
  the first edge of $f$, $d_{\mathrm{m}}$ the inverse of the PCDD
  induced by the vertices of $d$ that are bellow $\alpha(d)$, and
  $d_{\mathrm{r}}$ induced by the remaining vertices.  Clearly
  $d = \Upsilon(d_{\mathrm{l}}, d_{\mathrm{m}}, d_{\mathrm{r}})$, and
  since $d$ is finite it's clear that by recursively continuing this
  process we will eventually find an expression for $d$ that consists
  of applications of $\Upsilon$ and $\lambda$, and that such
  expression is unique.  So $\mathcal{P}$ is a ternary magma freely
  generated by $\lambda$.

  To follow the proof that Equation~\ref{eq:terndual} is satisfied the
  reader may want to consult Figure~\ref{fig:fusion2du}, where the
  dual of $\Upsilon(d_{\mathrm{l}}, d_{\mathrm{m}}, d_{\mathrm{r}})$
  of Figure~\ref{fig:fusion2} is shown as the fusion of
  $d_{\mathrm{l}}^{*}$, $d_{\mathrm{m}}^{*}$, and
  $d_{\mathrm{r}}^{*}$.  We first note that the underlying ditrees of
  both sides of the equation are equal.  We need to prove that at
  every vertex the same choice of connections is made, and this is
  clear for vertices different than $v_0$, $\alpha(d_{\mathrm{l}})$,
  $\alpha(\widebar{d_{\mathrm{m}}})$, and $\alpha(d_{\mathrm{r}})$
  since switching the connections can be done either before or after
  fusing the PCDDs.  Switching the connections of
  $\Upsilon(d_{\mathrm{l}}, d_{\mathrm{m}}, d_{\mathrm{r}})$ at
  $\alpha(\widebar{d_{\mathrm{m}}})$ means that we connect
  $f(\widebar{d_{\mathrm{m}}}^{*})$ to $v_0$ and by switching at $v_0$
  the resulting chain continues by connecting $v_0$ to
  $f(d_{\mathrm{l}}^{*})$.  By definition the same choices of
  connections are made in the construction of
  $\Upsilon(d_{\mathrm{r}}^{*},
  d_{\mathrm{m}}^{*},d_{\mathrm{l}}^{*})$.  Similarly, one can easily
  see that the flags of the two sides also agree.
\end{proof}

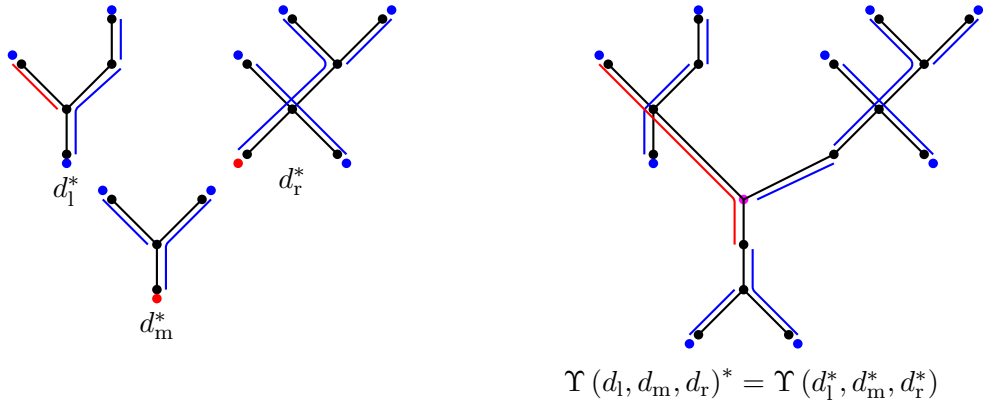
\begin{figure}[ht]
  \centering
\psset{unit=.6}
  \begin{pspicture}(-8,-4.6)(14,5.3)
    \rput(-3,0){%
  \begin{pspicture}(-3.5,-2.5)(6,1)
    \pnode(0,-2){m1}
    \psdot(0,-2)
    \pnode(0,-1){m2}
    \psdot(0,-1)
    \pnode(-1,0){m3}
    \psdot(-1,0)
    \pnode(1,0){m4}
    \psdot(1,0)
    \ncline{m3}{m2}
    \ncline{m4}{m2}
    \ncline{m2}{m1}
    \psline[linecolor=blue,linearc=.15](.2,-2)(.2,-1)(1.2,0)
    \psline[linecolor=blue](-.2,-1)(-1.2,0)
    \psdot[linecolor=blue](-1.2,.2)
    \psdot[linecolor=blue](1.2,.2)
    \psdot[linecolor=red](0,-2.2)
    \uput[-90](0,-2.2){$d_{\mathrm{m}}^{*}$}
    \pnode(-2,2){l1}
    \psdot(-2,2)
    \pnode(-2,1){l2}
    \psdot(-2,1)
    \pnode(-3,3){l3}
    \psdot(-3,3)
    \pnode(-1,3){l4}
    \psdot(-1,3)
    \pnode(-1,4){l5}
    \psdot(-1,4)
    \ncline{l2}{l1}
    \ncline{l1}{l4}
    \ncline{l1}{l3}
    \ncline{l4}{l5}
    \psdot[linecolor=blue](-2,.8)
    \psdot[linecolor=blue](-3.2,3.2)
    \psdot[linecolor=blue](-1,4.2)
    \psline[linecolor=red,linearc=.15](-2.2,2)(-3.2,3)
    \psline[linecolor=blue,linearc=.15](-1.8,1)(-1.8,2)(-.8,2.9)
    \psline[linecolor=blue](-.8,3.1)(-.8,4)
    \uput[-90](-2,.8){$d_{\mathrm{l}}^{*}$}
    \pnode(2,1){r1}
    \psdot(2,1)
    \pnode(3,2){r2}
    \psdot(3,2)
    \pnode(4,1){r3}
    \psdot(4,1)
    \pnode(4,3){r4}
    \psdot(4,3)
    \pnode(2,3){r5}
    \psdot(2,3)
    \pnode(3,4){r6}
    \psdot(3,4)
    \pnode(5,4){r7}
    \psdot(5,4)
    \ncline{r1}{r2}
    \ncline{r3}{r2}
    \ncline{r2}{r4}
    \ncline{r2}{r5}
    \ncline{r4}{r6}
    \ncline{r4}{r7}
    \psdot[linecolor=red](1.8,.8)
    \psdot[linecolor=blue](4.2,.8)
    \psdot[linecolor=blue](1.8,3.2)
    \psdot[linecolor=blue](2.8,4.2)
    \psdot[linecolor=blue](5.2,4.2)
    \psline[linecolor=blue,linearc=.15](4.2,1)(3.2,2)(2.2,3)
    \psline[linecolor=blue,linearc=.15](4.2,3)(5.2,4)
    \psline[linecolor=blue,linearc=.15](1.8,1.1)(3.8,3)(2.8,4)
    \uput[-90](3,1){$d_{\mathrm{r}}^{*}$}
  \end{pspicture}}
   \rput(10,0){%
  \begin{pspicture}(-3.5,-2.5)(6,3)
    \pnode(0,0){m1}
    \psdot(0,0)
    \pnode(0,-1){m2}
    \psdot(0,-1)
    \pnode(-1,-2){m3}
    \psdot(-1,-2)
    \pnode(1,-2){m4}
    \psdot(1,-2)
    \ncline{m3}{m2}
    \ncline{m4}{m2}
    \ncline{m2}{m1}
    \psline[linecolor=blue](-1.2,-2)(-.2,-1)
    \psdot[linecolor=blue](-1.2,-2.2)
    \psdot[linecolor=blue](1.2,-2.2)
    \pnode(-2,3){l1}
    \psdot(-2,3)
    \pnode(-2,2){l2}
    \psdot(-2,2)
    \pnode(-3,4){l3}
    \psdot(-3,4)
    \pnode(-1,4){l4}
    \psdot(-1,4)
    \pnode(-1,5){l5}
    \psdot(-1,5)
    \ncline{l2}{l1}
    \ncline{l1}{l4}
    \ncline{l1}{l3}
    \ncline{l4}{l5}
    \psdot[linecolor=blue](-2,1.8)
    \psdot[linecolor=blue](-3.2,4.2)
    \psdot[linecolor=blue](-1,5.2)
    \psline[linecolor=blue,linearc=.15](-2.2,2)(-2.2,3)(-1.2,4)
    \psline[linecolor=blue](-.8,4)(-.8,5)
    \pnode(2,2){r1}
    \psdot(2,2)
    \pnode(3,3){r2}
    \psdot(3,3)
    \pnode(4,2){r3}
    \psdot(4,2)
    \pnode(4,4){r4}
    \psdot(4,4)
    \pnode(2,4){r5}
    \psdot(2,4)
    \pnode(3,5){r6}
    \psdot(3,5)
    \pnode(5,5){r7}
    \psdot(5,5)
    \ncline{r1}{r2}
    \ncline{r3}{r2}
    \ncline{r2}{r4}
    \ncline{r2}{r5}
    \ncline{r4}{r6}
    \ncline{r4}{r7}
    \psdot[linecolor=blue](4.2,1.8)
    \psdot[linecolor=blue](1.8,4.2)
    \psdot[linecolor=blue](2.8,5.2)
    \psdot[linecolor=blue](5.2,5.2)
    \psline[linecolor=blue,linearc=.15](2,2.2)(3.8,4)(2.8,5)
    \psline[linecolor=blue,linearc=.15](4.2,2)(2.2,4)
    \psline[linecolor=blue,linearc=.15](4.2,4)(5.2,5)
    \pnode(0,1){v}
    \psdot[linecolor=magenta](0,1)
    \ncline{m1}{v}
    \ncline{v}{l1}
    \ncline{v}{r1}
    \psline[linecolor=blue,linearc=.15](1.2,-2)(.2,-1)(.2,-.1)
    \psline[linecolor=blue,linearc=.15](.3,1)(2,1.8)
    \psline[linecolor=red,linearc=.15](-.2,0)(-.2,1)(-3.2,4)
  \end{pspicture}}
   \uput[-90](9,-2.7){$\Upsilon \left( d_{\mathrm{l}}, d_{\mathrm{m}}, d_{\mathrm{r}} \right)^{*}= \Upsilon \left( d_{\mathrm{l}}^{*}, d_{\mathrm{m}}^{*},d_{\mathrm{r}} ^{*}\right)$ }
  \end{pspicture}
  \caption{The dual of Figure~\ref{fig:fusion2}.}
  \label{fig:fusion2du}
\end{figure}

\subsection{$\mathcal{N}$ as a free $*$-magma}
\label{sec:pmagma}

We use the bijection $\mathcal{M}$ of Theorem~\ref{thm:medialbij} to
endow $\mathcal{N}$ with the structure of a free ternary magma generated
by the nc-tree with one vertex $\lambda$, i.e. so that $\mathcal{M}$ is
the structural bijection.  Since $\mathcal{M}$ is duality preserving this
exhibits $\mathcal{N}$ endowed with nc-duality as a free $*$-magma.

Given an nc-tree $t$ let $1\,k$ be the rightmost edge incident to $1$.
Removing that edge gives a forest of two nc-trees the one attached to
$k$ and the one attached to $1$, $t_\mathrm{l}$ is the lateral,
$t_\mathrm{m}$ is the tree to the left of $(1,k)$ and $t_\mathrm{r}$
the one to the right, see Figure~\ref{fig:duluqbij}.

Conversely given three nc-trees $t_{\mathrm{l}}$, $t_{\mathrm{m}}$,
and $t_{\mathrm{r}}$ of orders $n_1$, $n_2$, and $n_3$ respectively,
their fusion
$\Upsilon \left( t_{\mathrm{l}}, t_{\mathrm{m}}, t_{\mathrm{r}}
\right)$ is obtained by relabeling $t_{\mathrm{r}}$ via
$i \mapsto i+1$, $t_{\mathrm{m}}$ by $i \mapsto i + n_3$,
and finally $t_{\mathrm{l}}$ by
$i \mapsto i + n_2 + n_3 -1$ except that we
keep the label of $1$.  Notice that the roots of $t_{\mathrm{m}}$ and
$t_{\mathrm{r}}$ receive the same label $n_3 + 1$ so we
identify them.  Finally we add an edge connecting $1$ and
$n_3 + 1$.

As an example, the nc-trees that correspond to the PCDDs of the example in
Figure~\ref{fig:fusion2}, and their fusion are shown in
Figure~\ref{fig:ncfusion}.

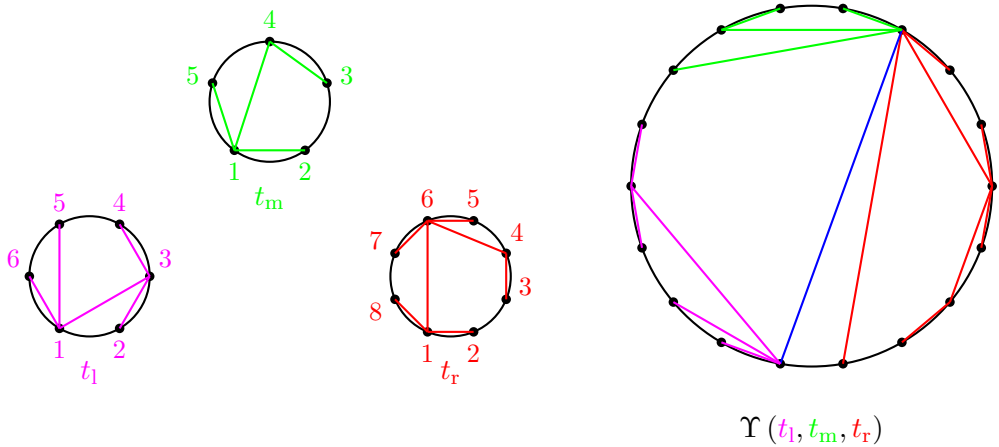
\begin{figure}[ht]
\psset{unit=.8}
  \begin{pspicture}(-2.5,-2)(14.5,6.3)
    \rput (2,4){
      \begin{pspicture}(-1.25106, -1.10902)(1.25106, 1.30000)
        \pscircle(0,0){1}
        \pnode(-0.58779,-0.80902){1}
        \psdot(-0.58779,-0.80902)
        \uput[-90](-0.58779,-0.80902){\green \small $1$}
        \pnode(0.58779,-0.80902){2}
        \psdot(0.58779,-0.80902)
        \uput[-90](0.58779,-0.80902){\green \small $2$}
        \pnode(0.95106,0.30902){3}
        \psdot(0.95106,0.30902)
        \uput[18](0.95106,0.30902){\green \small $3$}
        \pnode(0.00000,1.00000){4}
        \psdot(0.00000,1.00000)
        \uput[90](0.00000,1.00000){\green \small $4$}
        \pnode(-0.95106,0.30902){5}
        \psdot(-0.95106,0.30902)
        \uput[162](-0.95106,0.30902){\green \small $5$}
        \ncline[linecolor=green]{1}{2}
        \ncline[linecolor=green]{1}{4}
        \ncline[linecolor=green]{1}{5}
        \ncline[linecolor=green]{3}{4}
        \uput[-90](0,-1.2){\green $t_{\mathrm{m}}$}
      \end{pspicture}}
    \rput(-1,1){
      \begin{pspicture}(-1.30000, -1.16603)(1.30000, 1.16603)
        \pscircle(0,0){1}
        \pnode(-0.50000,-0.86603){1}
        \psdot(-0.50000,-0.86603)
        \uput[-90](-0.50000,-0.86603){\magenta \small $1$}
        \pnode(0.50000,-0.86603){2}
        \psdot(0.50000,-0.86603)
        \uput[-90](0.50000,-0.86603){\magenta \small $2$}
        \pnode(1.00000,0.00000){3}
        \psdot(1.00000,0.00000)
        \uput[45](1.00000,0.00000){\magenta \small $3$}
        \pnode(0.50000,0.86603){4}
        \psdot(0.50000,0.86603)
        \uput[90](0.50000,0.86603){\magenta \small $4$}
        \pnode(-0.50000,0.86603){5}
        \psdot(-0.50000,0.86603)
        \uput[90](-0.50000,0.86603){\magenta \small $5$}
        \pnode(-1.00000,-0.00000){6}
        \psdot(-1.00000,-0.00000)
        \uput[135](-1.00000,-0.00000){\magenta \small $6$}
        \ncline[linecolor=magenta]{1}{3}
        \ncline[linecolor=magenta]{1}{5}
        \ncline[linecolor=magenta]{1}{6}
        \ncline[linecolor=magenta]{2}{3}
        \ncline[linecolor=magenta]{3}{4}
        \uput[-90](0,-1.2){\magenta $t_{\mathrm{l}}$}
      \end{pspicture}}
    \rput(5,1){
      \begin{pspicture}(-1.22388, -1.22388)(1.22388, 1.22388)
        \pscircle(0,0){1}
        \pnode(-0.38268,-0.92388){1}
        \psdot(-0.38268,-0.92388)
        \uput[-90](-0.38268,-0.92388){\red \small $1$}
        \pnode(0.38268,-0.92388){2}
        \psdot(0.38268,-0.92388)
        \uput[-90](0.38268,-0.92388){\red \small $2$}
        \pnode(0.92388,-0.38268){3}
        \psdot(0.92388,-0.38268)
        \uput[30](0.92388,-0.38268){\red \small $3$}
        \pnode(0.92388,0.38268){4}
        \psdot(0.92388,0.38268)
        \uput[60](0.92388,0.38268){\red \small $4$}
        \pnode(0.38268,0.92388){5}
        \psdot(0.38268,0.92388)
        \uput[90](0.38268,0.92388){\red \small $5$}
        \pnode(-0.38268,0.92388){6}
        \psdot(-0.38268,0.92388)
        \uput[90](-0.38268,0.92388){\red \small $6$}
        \pnode(-0.92388,0.38268){7}
        \psdot(-0.92388,0.38268)
        \uput[150](-0.92388,0.38268){\red \small $7$}
        \pnode(-0.92388,-0.38268){8}
        \psdot(-0.92388,-0.38268)
        \uput[210](-0.92388,-0.38268){\red \small $8$}
        \ncline[linecolor=red]{1}{2}
        \ncline[linecolor=red]{1}{6}
        \ncline[linecolor=red]{1}{8}
        \ncline[linecolor=red]{3}{4}
        \ncline[linecolor=red]{4}{6}
        \ncline[linecolor=red]{5}{6}
        \ncline[linecolor=red]{6}{7}
        \uput[-90](0,-1.2){\red $t_{\mathrm{r}}$}
      \end{pspicture}}
    \rput(11,2.5){
      \psset{unit=3}
      \begin{pspicture}(-1.30000, -1.28481)(1.30000, 1.28481)
        \pscircle(0,0){1}
        \pnode(-0.17365,-0.98481){1}
        \psdot(-0.17365,-0.98481)
        \pnode(0.17365,-0.98481){2}
        \psdot(0.17365,-0.98481)
        \pnode(0.50000,-0.86603){3}
        \psdot(0.50000,-0.86603)
        \pnode(0.76604,-0.64279){4}
        \psdot(0.76604,-0.64279)
        \pnode(0.93969,-0.34202){5}
        \psdot(0.93969,-0.34202)
        \pnode(1.00000,-0.00000){6}
        \psdot(1.00000,-0.00000)
        \pnode(0.93969,0.34202){7}
        \psdot(0.93969,0.34202)
        \pnode(0.76604,0.64279){8}
        \psdot(0.76604,0.64279)
        \pnode(0.50000,0.86603){9}
        \psdot(0.50000,0.86603)
        \pnode(0.17365,0.98481){10}
        \psdot(0.17365,0.98481)
        \pnode(-0.17365,0.98481){11}
        \psdot(-0.17365,0.98481)
        \pnode(-0.50000,0.86603){12}
        \psdot(-0.50000,0.86603)
        \pnode(-0.76604,0.64279){13}
        \psdot(-0.76604,0.64279)
        \pnode(-0.93969,0.34202){14}
        \psdot(-0.93969,0.34202)
        \pnode(-1.00000,0.00000){15}
        \psdot(-1.00000,0.00000)
        \pnode(-0.93969,-0.34202){16}
        \psdot(-0.93969,-0.34202)
        \pnode(-0.76604,-0.64279){17}
        \psdot(-0.76604,-0.64279)
        \pnode(-0.50000,-0.86603){18}
        \psdot(-0.50000,-0.86603)
        \ncline[linecolor=red]{9}{2}
        \ncline[linecolor=red]{9}{6}
        \ncline[linecolor=red]{9}{8}
        \ncline[linecolor=blue]{1}{9}
        \ncline[linecolor=magenta]{1}{15}
        \ncline[linecolor=magenta]{1}{17}
        \ncline[linecolor=magenta]{1}{18}
        \ncline[linecolor=red]{3}{4}
        \ncline[linecolor=red]{4}{6}
        \ncline[linecolor=red]{5}{6}
        \ncline[linecolor=red]{6}{7}
        \ncline[linecolor=green]{9}{10}
        \ncline[linecolor=green]{9}{12}
        \ncline[linecolor=green]{9}{13}
        \ncline[linecolor=green]{11}{12}
        \ncline[linecolor=magenta]{14}{15}
        \ncline[linecolor=magenta]{15}{16}
        \uput[-90](0,-1.2){$\Upsilon \left( {\magenta t_{\mathrm{l}}}, {\green t_{\mathrm{m}}}, {\red t_{\mathrm{r}}} \right)$}
      \end{pspicture}}
  \end{pspicture}
  \caption{The nc-trees corresponding to the PCDDs of Figure~\ref{fig:fusion2}.}
  \label{fig:ncfusion}
\end{figure}

\section{The structural bijections}
\label{sec:strbij}

In this section we give combinatorial/topological interpretations of
the structural bijections in Diagram~\eqref{eq:commdiag}.

\subsection{The structural bijection  $\psi \co \mathcal{Q} \to \mathcal{T}$}
\label{sec:psi}

A nice topological/combinatorial description of
${\psi \co \mathcal{Q}_{m} \to \mathcal{T}_m}$ has been given
in~\cite{HiltonPedersen1991}.  For a quadrangular dissection $q$,
$\psi(q)$ is a sort of dual of $q$ viewed as a graph embedded in the
disk with all its vertices mapped on the boundary circle: the disk is
divided into $n-1$ quadrangular cells (the cells of the dissection $q$) and
$2n$ bigons formed by the edges of the polygons and the arcs of the
boundary circle.  Let $T$ be the $4$-valent plane tree that has a
vertex for each of these regions, and an edge between two vertices if
the corresponding regions share an edge.  See
Figure~\ref{fig:dualbegex}, where, in the middle, a vertex that
correspond to a cell is drawn in the interior of that cell, and a
vertex that corresponds to a bigon is drawn in the boundary arc of
that bigon.  Clearly bigons give leaves of $T$ and cells give internal
vertices.  The ternary tree $\psi(t)$ is obtained from $T$ by removing
the leaf that comes from the bigon that contains the root edge $1\,2$,
declaring the vertex it was attached to be the root of the remaining
tree, and using the orientation of the disk to order the children of
any internal vertex.  See Figure~\ref{fig:dualbegex} for an example of
this construction.

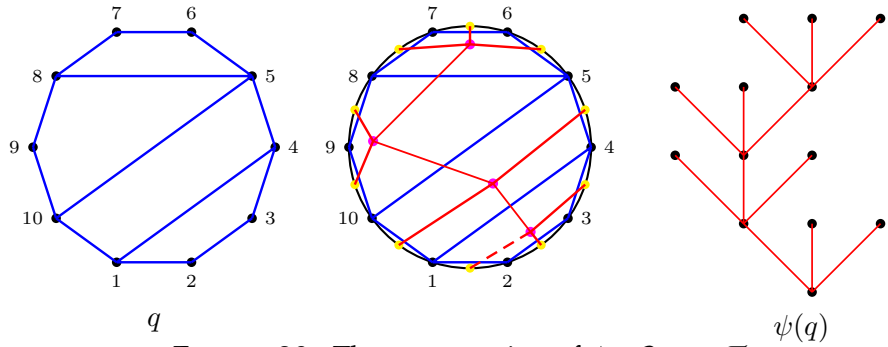
\begin{figure}[ht]
\psset{unit=.7}
\begin{pspicture}(-6,-2.4)(12,3.6)
\rput(-3,.5){%
\psset{unit=2.3}
  \begin{pspicture}(-1,-1)(1,1)
\pnode(-0.309016994374947, 0.951056516295154){7}
\psdot(-0.309016994374947, 0.951056516295154)
\uput[90](-0.309016994374947, 0.951056516295154){\tiny $7$}
\pnode(-0.809016994374947, 0.587785252292473){8}
\psdot(-0.809016994374947, 0.587785252292473)
\uput[180](-0.809016994374947, 0.587785252292473){\tiny $8$}
\pnode(-1.00000000000000, 0){9}
\psdot(-1.00000000000000, 0)
\uput[180](-1,0){\tiny $9$}
\pnode(-0.809016994374947, -0.587785252292473){10}
\psdot(-0.809016994374947, -0.587785252292473)
\uput[180](-0.809016994374947, -0.587785252292473){\tiny $10$}
\pnode(-0.309016994374948, -0.951056516295154){1}
\psdot(-0.309016994374948, -0.951056516295154)
\uput[-90](-0.309016994374948, -0.951056516295154){\tiny $1$}
\pnode(0.309016994374947, -0.951056516295154){2}
\psdot(0.309016994374947, -0.951056516295154)
\uput[-90](0.309016994374947, -0.951056516295154){\tiny $2$}
\pnode(0.809016994374947, -0.587785252292473){3}
\psdot(0.809016994374947, -0.587785252292473)
\uput[0](0.809016994374947, -0.587785252292473){\tiny $3$}
\pnode(1.00000000000000, 0){4}
\psdot(1.00000000000000, 0)
\uput[0](1,0){\tiny $4$}
\pnode(0.809016994374947, 0.587785252292473){5}
\psdot(0.809016994374947, 0.587785252292473)
\uput[0](0.809016994374947, 0.587785252292473){\tiny $5$}
\pnode(0.309016994374948, 0.951056516295153){6}
\psdot(0.309016994374948, 0.951056516295153)
\uput[90](0.309016994374948, 0.951056516295153){\tiny $6$}
\ncline[linecolor=blue,linewidth=.02]{1}{2}
\ncline[linecolor=blue,linewidth=.02]{2}{3}
\ncline[linecolor=blue,linewidth=.02]{3}{4}
\ncline[linecolor=blue,linewidth=.02]{4}{5}
\ncline[linecolor=blue,linewidth=.02]{5}{6}
\ncline[linecolor=blue,linewidth=.02]{6}{7}
\ncline[linecolor=blue,linewidth=.02]{7}{8}
\ncline[linecolor=blue,linewidth=.02]{8}{9}
\ncline[linecolor=blue,linewidth=.02]{9}{10}
\ncline[linecolor=blue,linewidth=.02]{10}{1}
\ncline[linecolor=blue,linewidth=.02]{1}{4}
\ncline[linecolor=blue,linewidth=.02]{5}{8}
\ncline[linecolor=blue,linewidth=.02]{5}{10}
  \end{pspicture}}
\uput[-90](-3,-2.4){$q$}
\rput(3,.5){%
\psset{unit=2.3}
  \begin{pspicture}(-1,-1)(1,1)
\pscircle(0,0){1}
\pnode(-0.309016994374947, 0.951056516295154){7}
\psdot(-0.309016994374947, 0.951056516295154)
\uput[90](-0.309016994374947, 0.951056516295154){\tiny $7$}
\pnode(-0.809016994374947, 0.587785252292473){8}
\psdot(-0.809016994374947, 0.587785252292473)
\uput[180](-0.809016994374947, 0.587785252292473){\tiny $8$}
\pnode(-1.00000000000000, 0){9}
\psdot(-1.00000000000000, 0)
\uput[180](-1,0){\tiny $9$}
\pnode(-0.809016994374947, -0.587785252292473){10}
\psdot(-0.809016994374947, -0.587785252292473)
\uput[180](-0.809016994374947, -0.587785252292473){\tiny $10$}
\pnode(-0.309016994374948, -0.951056516295154){1}
\psdot(-0.309016994374948, -0.951056516295154)
\uput[-90](-0.309016994374948, -0.951056516295154){\tiny $1$}
\pnode(0.309016994374947, -0.951056516295154){2}
\psdot(0.309016994374947, -0.951056516295154)
\uput[-90](0.309016994374947, -0.951056516295154){\tiny $2$}
\pnode(0.809016994374947, -0.587785252292473){3}
\psdot(0.809016994374947, -0.587785252292473)
\uput[0](0.809016994374947, -0.587785252292473){\tiny $3$}
\pnode(1.00000000000000, 0){4}
\psdot(1.00000000000000, 0)
\uput[0](1,0){\tiny $4$}
\pnode(0.809016994374947, 0.587785252292473){5}
\psdot(0.809016994374947, 0.587785252292473)
\uput[0](0.809016994374947, 0.587785252292473){\tiny $5$}
\pnode(0.309016994374948, 0.951056516295153){6}
\psdot(0.309016994374948, 0.951056516295153)
\uput[90](0.309016994374948, 0.951056516295153){\tiny $6$}
\ncline[linecolor=blue,linewidth=.02]{1}{2}
\ncline[linecolor=blue,linewidth=.02]{2}{3}
\ncline[linecolor=blue,linewidth=.02]{3}{4}
\ncline[linecolor=blue,linewidth=.02]{4}{5}
\ncline[linecolor=blue,linewidth=.02]{5}{6}
\ncline[linecolor=blue,linewidth=.02]{6}{7}
\ncline[linecolor=blue,linewidth=.02]{7}{8}
\ncline[linecolor=blue,linewidth=.02]{8}{9}
\ncline[linecolor=blue,linewidth=.02]{9}{10}
\ncline[linecolor=blue,linewidth=.02]{10}{1}
\ncline[linecolor=blue,linewidth=.02]{1}{4}
\ncline[linecolor=blue,linewidth=.02]{5}{8}
\ncline[linecolor=blue,linewidth=.02]{5}{10}
\psdot[linecolor=magenta,linewidth=.02](.5,-.7)
\psdot[linecolor=magenta,linewidth=.02](.19,-.3)
\psdot[linecolor=magenta,linewidth=.02](-.8,.05)
\psdot[linecolor=magenta,linewidth=.02](0,.85)    
\pnode(.5,-.7){d1}   
\pnode(.19,-.3){d2}  
\pnode(-.8,.05){d3}  
\pnode(0,.85){d4}
\ncline[linecolor=red,linewidth=.015]{d1}{d2}
\ncline[linecolor=red,linewidth=.015]{d2}{d3}
\ncline[linecolor=red,linewidth=.015]{d3}{d4}
\pnode(0,1.0){b5}					
\pnode(-0.587785252292473, 0.8090169943749475){b6}
\pnode(-0.9510565162951535, 0.3090169943749475){b7}	
\pnode(-0.9510565162951536, -0.3090169943749473){b8}	
\pnode(-0.5877852522924732, -0.8090169943749473){b9}	
\pnode(0, -1.0){b0}
\pnode(0.5877852522924729, -0.8090169943749476){b1}
\pnode(0.9510565162951535, -0.3090169943749476){b2}	
\pnode(0.9510565162951536, 0.3090169943749472){b3}
\pnode(0.5877852522924734, 0.8090169943749472){b4}   
\psdot[linecolor=yellow,linewidth=.015](0, 1.0)					
\psdot[linecolor=yellow,linewidth=.015](-0.587785252292473, 0.8090169943749475)	
\psdot[linecolor=yellow,linewidth=.015](-0.9510565162951535, 0.3090169943749475)	
\psdot[linecolor=yellow,linewidth=.015](-0.9510565162951536, -0.3090169943749473)	
\psdot[linecolor=yellow,linewidth=.015](-0.5877852522924732, -0.8090169943749473)	
\psdot[linecolor=yellow,linewidth=.015](0, -1.0)					
\psdot[linecolor=yellow,linewidth=.015](0.5877852522924729, -0.8090169943749476)	
\psdot[linecolor=yellow,linewidth=.015](0.9510565162951535, -0.3090169943749476)	
\psdot[linecolor=yellow,linewidth=.015](0.9510565162951536, 0.3090169943749472)	
\psdot[linecolor=yellow,linewidth=.015](0.5877852522924734, 0.8090169943749472)   
\ncline[linecolor=red,linewidth=.02,linestyle=dashed]{d1}{b0}
\ncline[linecolor=red,linewidth=.02]{d1}{b1}
\ncline[linecolor=red,linewidth=.02]{d1}{b2}
\ncline[linecolor=red,linewidth=.02]{d2}{b9}
\ncline[linecolor=red,linewidth=.02]{d2}{b3}
\ncline[linecolor=red,linewidth=.02]{d3}{b8}
\ncline[linecolor=red,linewidth=.02]{d3}{b7}
\ncline[linecolor=red,linewidth=.02]{d4}{b6}
\ncline[linecolor=red,linewidth=.02]{d4}{b5}
\ncline[linecolor=red,linewidth=.02]{d4}{b4}
  \end{pspicture}}
\rput(8.2,1){%
\psset{unit=1.3}
\begin{pspicture}(3,7)
  \pnode(2.5,1){d1}
  \rput(2.5,1){\psdot(0,0)}
  \pnode(1.5,2){d2}
  \rput(1.5,2){\psdot(0,0)}
  \pnode(2.5,2){b3}
  \rput(2.5,2){\psdot(0,0)}
  \pnode(3.5,2){b2}
  \rput(3.5,2){\psdot(0,0)}
  \pnode(.5,3){b10}
  \rput(.5,3){\psdot(0,0)}
  \pnode(1.5,3){d3}
  \rput(1.5,3){\psdot(0,0)}
  \pnode(2.5,3){b4}
  \rput(2.5,3){\psdot(0,0)}
  \pnode(.5,4){b9}
  \rput(.5,4){\psdot(0,0)}
  \pnode(1.5,4){b8}
  \rput(1.5,4){\psdot(0,0)}
  \pnode(2.5,4){d4}
  \rput(2.5,4){\psdot(0,0)}
  \pnode(1.5,5){b7}
  \rput(1.5,5){\psdot(0,0)}
  \pnode(2.5,5){b6}
  \rput(2.5,5){\psdot(0,0)}
  \pnode(3.5,5){b5}
  \rput(3.5,5){\psdot(0,0)}
\ncline[linecolor=red,linewidth=.025]{d1}{b2}
\ncline[linecolor=red,linewidth=.025]{d1}{b3}
\ncline[linecolor=red,linewidth=.025]{d1}{d2}
\ncline[linecolor=red,linewidth=.025]{d2}{b10}
\ncline[linecolor=red,linewidth=.025]{d2}{d3}
\ncline[linecolor=red,linewidth=.025]{d2}{b4}
\ncline[linecolor=red,linewidth=.025]{d3}{b9}
\ncline[linecolor=red,linewidth=.025]{d3}{b8}
\ncline[linecolor=red,linewidth=.025]{d3}{d4}
\ncline[linecolor=red,linewidth=.025]{d4}{b7}
\ncline[linecolor=red,linewidth=.025]{d4}{b6}
\ncline[linecolor=red,linewidth=.025]{d4}{b5}
  \end{pspicture}}
\uput[-90](9.3,-2.4){$\psi(q)$}
  \end{pspicture}
  \caption{The construction of $\psi \co \mathcal{Q}_m \to \mathcal{T}_m$. }
  \label{fig:dualbegex}
\end{figure}

Clearly this process of obtaining $\psi(q)$ can be reversed: starting
with a ternary tree $t$ with $m$ internal vertices construct an
$4$-valent plane tree $T$ with $n := m+1$ vertices by attaching a new
leaf labeled $1\,\,2$ bellow the root. Then list the leaves of $T$ in
the order induced by the counterclockwise orientation (starting at
$1\,\,2$) and label them by the edges of the $2n$-gon in the order
$1\,\,2, 2\,\,3, \ldots, 2n\,\,1$, and label the corresponding pendant
edges by the same label.  Since $t$ has $2n+1$ leaves and only $n-1$
internal vertices there is at least one internal vertex with all its
children being leaves; if such a vertex has children labeled (from
right to left) $i\,\,i+1, i+1\,\, i+2, i+2\, i+3$, label it
$i\,\, i+1\,\,i+2\,\,i+3$ and the the edge connecting it to its parent
$i\,\,i+3$. Proceeding recursively we can label all internal vertices
of $T$ with the vertices of an quadrangular cell, and all non-pendant
edges of $T$ with a diagonal of the $2n$-gon. From this decorated tree
we can reconstruct the $n$-cluster that corresponds to the polygonal
dissection, for an example see Figure~\ref{fig:psiinverse4}, where we
show $\psi^{-1}(t)$ for the ternary tree at the bottom right of
Figure~\ref{fig:dualbegex}.

\begin{figure}[ht]
  \begin{pspicture}(-7,-5.3)(11,5.5)
\rput(-5,0){%
\begin{pspicture}(3,7)
  \pnode(2.5,1){d1}
  \rput(2.5,1){\psdot(0,0)}
  \pnode(1.5,2){d2}
  \rput(1.5,2){\psdot(0,0)}
  \pnode(2.5,2){b3}
  \rput(2.5,2){\psdot(0,0)}
  \pnode(3.5,2){b2}
  \rput(3.5,2){\psdot(0,0)}
  \pnode(.5,3){b10}
  \rput(.5,3){\psdot(0,0)}
  \pnode(1.5,3){d3}
  \rput(1.5,3){\psdot(0,0)}
  \pnode(2.5,3){b4}
  \rput(2.5,3){\psdot(0,0)}
  \pnode(.5,4){b9}
  \rput(.5,4){\psdot(0,0)}
  \pnode(1.5,4){b8}
  \rput(1.5,4){\psdot(0,0)}
  \pnode(2.5,4){d4}
  \rput(2.5,4){\psdot(0,0)}
  \pnode(1.5,5){b7}
  \rput(1.5,5){\psdot(0,0)}
  \pnode(2.5,5){b6}
  \rput(2.5,5){\psdot(0,0)}
  \pnode(3.5,5){b5}
  \rput(3.5,5){\psdot(0,0)}
\ncline[linewidth=.025]{d1}{b2}
\ncline[linewidth=.025]{d1}{b3}
\ncline[linewidth=.025]{d1}{d2}
\ncline[linewidth=.025]{d2}{b10}
\ncline[linewidth=.025]{d2}{d3}
\ncline[linewidth=.025]{d2}{b4}
\ncline[linewidth=.025]{d3}{b9}
\ncline[linewidth=.025]{d3}{b8}
\ncline[linewidth=.025]{d3}{d4}
\ncline[linewidth=.025]{d4}{b7}
\ncline[linewidth=.025]{d4}{b6}
\ncline[linewidth=.025]{d4}{b5}
  \end{pspicture}}
\uput[-90](-3.7,-3.5){$t$}
\rput(1,0){%
\psset{unit=1.8,nodesep=.15}
\begin{pspicture}(3,5)
  \pnode(2.5,0){b0}
  \rput(2.5,0){\footnotesize {\blue $1\,\,2$}}
  \pnode(2.5,1){d1}
  \rput(2.5,1){\footnotesize {\blue $1\,\,2\,\,3\,\,4$}}
  \pnode(1.5,2){d2}
  \rput(1.5,2){\footnotesize {\blue $1\,\,4\,\,5\,\,10$}}
  \pnode(2.5,2){b3}
  \rput(2.5,2){\footnotesize {\blue $3\,\,4$}}
  \pnode(3.5,2){b2}
  \rput(3.5,2){\footnotesize {\blue $2\,\,3$}}
  \pnode(.5,3){b10}
  \rput(.5,3){\footnotesize {\blue $10\,\,1$}}
  \pnode(1.5,3){d3}
  \rput(1.5,3){\footnotesize {\blue $5\,\,8\,\,9\,\,10$}}
  \pnode(2.5,3){b4}
  \rput(2.5,3){\footnotesize {\blue $4\,\,5$}}
  \pnode(.5,4){b9}
  \rput(.5,4){\footnotesize {\blue $9\,\,10$}}
  \pnode(1.5,4){b8}
  \rput(1.5,4){\footnotesize {\blue $8\,\,9$}}
  \pnode(2.5,4){d4}
  \rput(2.5,4){\footnotesize {\blue $5\,\,6\,\,7\,\,8$}}
  \pnode(1.5,5){b7}
  \rput(1.5,5){\footnotesize {\blue $7\,\,8$}}
  \pnode(2.5,5){b6}
  \rput(2.5,5){\footnotesize {\blue $6\,\,7$}}
  \pnode(3.5,5){b5}
  \rput(3.5,5){\footnotesize {\blue $5\,\,6$}}
\ncline[linestyle=dashed]{b0}{d1}
\ncput*{\footnotesize {\red $1\,\,2$}}
\ncline{d1}{b2}
\ncput*{\footnotesize {\red $2\,\,3$}}
\ncline{d1}{b3}
\ncput*{\footnotesize {\red $3\,\,4$}}
\ncline{d1}{d2}
\ncput*{\footnotesize {\red $1\,\,4$}}
\ncline{d2}{b10}
\ncput*{\footnotesize {\red $1\,\,10$}}
\ncline{d2}{d3}
\ncput*{\footnotesize {\red $5\,\,10$}}
\ncline{d2}{b4}
\ncput*{\footnotesize {\red $4\,\,5$}}
\ncline{d3}{b9}
\ncput*{\footnotesize {\red $9\,\,10$}}
\ncline{d3}{b8}
\ncput*{\footnotesize {\red $8\,\,9$}}
\ncline{d3}{d4}
\ncput*{\footnotesize {\red $5\,\,8$}}
\ncline{d4}{b7}
\ncput*{\footnotesize {\red $7\,\,8$}}
\ncline{d4}{b6}
\ncput*{\footnotesize {\red $6\,\,7$}}
\ncline{d4}{b5}
\ncput*{\footnotesize {\red $5\,\,6$}}
\end{pspicture}}
\rput(7,0){%
  \begin{pspicture}(-3,2)
    \psdots(0,0)(1,0)(1,1)(0,1)(-1,0)(-1,1)(-2,0)(-2,1)(-2,2)(-1,2)
    \psline[linecolor=blue](0,0)(1,0)(1,1)(0,1)(0,0)(-1,0)(-1,1)(0,1)
    \psline[linecolor=blue](-1,0)(-2,0)(-2,1)(-1,1)
    \psline[linecolor=blue](-2,1)(-2,2)(-1,2)(-1,1)
    \uput[-90](0,0){\tiny $1$}
    \uput[-90](1,0){\tiny $2$}
    \uput[90](1,1){\tiny $3$}
    \uput[90](0,1){\tiny $4$}
    \uput[-90](-1,0){\tiny $10$}
    \uput[45](-1,1){\tiny $5$}
    \uput[-90](-2,0){\tiny $9$}
    \uput[180](-2,1){\tiny $8$}
    \uput[90](-2,2){\tiny $7$}
    \uput[90](-1,2){\tiny $6$}
  \end{pspicture}}
\uput[-90](8,-3){$\psi^{-1}(t)$}
  \end{pspicture}
  \caption{The construction of $\psi^{-1} \co \mathcal{T}_m \to \mathcal{Q}_m$. }
  \label{fig:psiinverse4}
\end{figure}
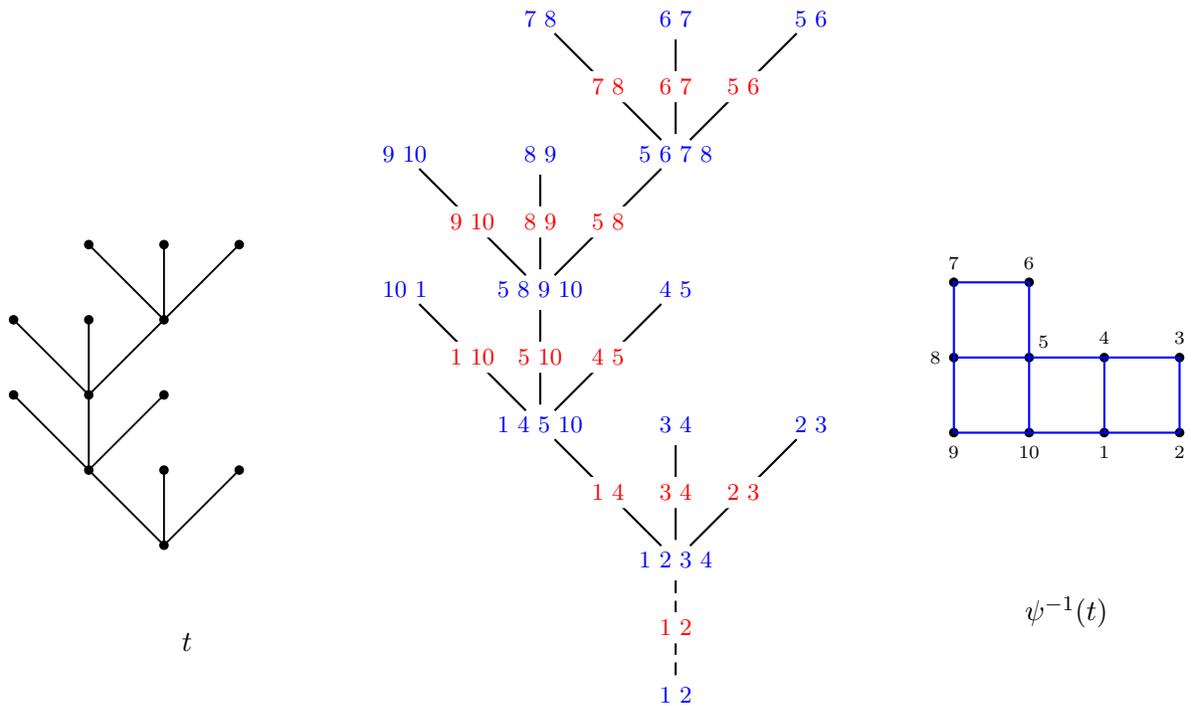

Note that the above description of $\psi^{-1}(t)$ can be expressed in terms
of the operation that that $t$ induces on ternary magmas described
in Remark~\ref{rem:ternop}. The label of an internal vertex $v$ of the
intermediate tree $T$ is obtained by applying the ternary operator
$\mathcal{S}_n^3 \to \mathcal{S}_n\co (a,b,c) \mapsto abc$ 
to the labels of the outgoing edges of $v$ viewed as transpositions,
while the label of an edge to an internal vertex is obtained by
applying the ternary operator
$\mathcal{S}_n^3 \to \mathcal{S}_n\co (a,b,c) \mapsto c^{ba}$.

\subsection{The structural bijection  $\sigma\co \mathcal{N} \to \mathcal{T}$}
\label{sec:sigma}

The structural bijection $\sigma\co \mathcal{N} \to \mathcal{T}$ is,
modulo some choices, the bijection defined in Lemme 3.11
of~\cite{DulPen1993}.  Indeed the authors there define a bijection
recursively by making an arbitrary choice of one of the six bijections
$\mathcal{N}_{2} \to \mathcal{T}_{2}$ and then for
$t\in \mathcal{N}_m$ with $m>2$ they recursively define the image of
$t$ to be
$\Upsilon \left( t_{\mathrm{l}}, t_{\mathrm{m}}, t_{\mathrm{r}}
\right)$, where $t_{\mathrm{l}}$, $t_{\mathrm{m}}$, and
$t_{\mathrm{r}}$ are defined, taking into account the difference in
conventions, as in the second paragraph of Section~\ref{sec:pmagma},
see Figure~\ref{fig:duluqbij}.  It follows that if we chose the structural
bijection when $m=2$ their bijection is exactly $\sigma$.

\begin{figure}[ht]
  \centering
  \psset{unit=1.3}
    \begin{pspicture}(-1.4,-1.4)(1.4,1.4)
    \pscircle(0,0){1}
    \psdots(0,-1)(0.6427876097, 0.7660444431)
    \uput[-90](0, -1){\small $1$}
    \uput[90](0.6427876097, 0.7660444431){\small $k$}
    \psline(0, -1)(0.6427876097, 0.7660444431)
    \pscurve(0.6427876097, 0.7660444431)(.7,0)(0.8660254038, -0.5000000000)
    \uput[45](1,.1){$t_{\mathrm{r}}$}
    \pscurve(0.6427876097, 0.7660444431)(-.1,.6)(-0.8660254038, 0.5000000000) 
    \uput[90](-.4,1){$t_{\mathrm{m}}$}
    \pscurve(0,-1)(-.4,-.7)(-0.8660254038, -0.5000000000)
    \uput[-90](-.6,-.8){$t_{\mathrm{l}}$}
  \end{pspicture}
  \caption{Expressing an nc-tree as
    $\Upsilon \left( t_{\mathrm{l}}, t_{\mathrm{m}}, t_{\mathrm{r}}
    \right)$.}
  \label{fig:duluqbij}
\end{figure}
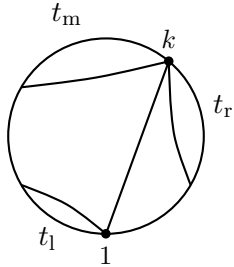

Since $\sigma = \psi\circ \phi^{-1}$ this work provides a combinatorial/topological
interpretation of their bijection.

\subsection{The structural bijection  $\phi\co \mathcal{Q} \to \mathcal{N}$}
\label{sec:phi}
Let $q$ be a quadrangular dissection with $m$ cells, then the polygon
has $2n$ vertices where $n = m+1$ and there are $m-1$ diagonals. Since there
are $2n$ vertices and $n-1$ cells, there is at least one cell with
boundary containing three edges of the polygon.  By inductively removing
such extremal cells one can see that each dissecting diagonal connects
two vertices of opposite parity, and so each cell has a diagonal that
connects two odd vertices and a diagonal that connects two even
vertices.  The non-crossing tree $\phi(q)$ is the tree obtained by
taking the ``odd'' diagonals of the cells, deleting the even vertices,
and relabeling the odd vertices via $2i-1 \mapsto i$.  Since each edge
of $\phi(q)$ is contained in a cell of the quadrangulation this is indeed
an nc-tree. 

To obtain $\phi^{-1}(t)$, for a non-crossing tree $t$, start by
pegging $t$ on the disk with vertices labeled $1,3, \ldots, 2n-1$, and
construct $\kappa(t)$ with vertices labeled $2,4,\ldots, 2n$.  An edge
$e$ of $t$ intersects only its dual edge $e^{*}$ in $\kappa(t)$ and so
we can construct a quadrangular cell by connecting their endpoints, if
$e = i\,j$ with $i<j$ and $e^{*} = k,l$ with $k<l$ we get the
quadrangular cell $i\,k\,j\,l$ of $\phi^{-1}(t)$.  See
Figure~\ref{fig:phi}, for an example of this construction.

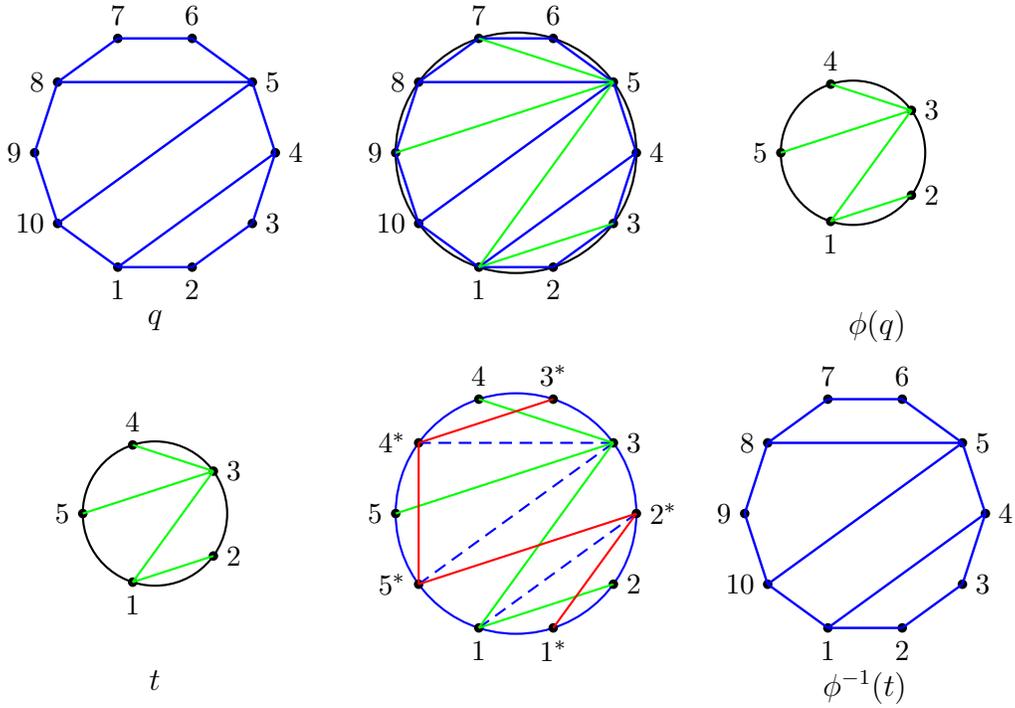
\begin{figure}[ht]
  \centering
    \psset{unit=1.6}
    \begin{pspicture}(-3.5,-4.8)(5.4,1.7)
      \rput(-2,0){%
  \begin{pspicture}(-1,-1)(1,1)
\pnode(-0.309016994374947, 0.951056516295154){7}
\psdot(-0.309016994374947, 0.951056516295154)
\uput[90](-0.309016994374947, 0.951056516295154){$\tiny 7$}
\pnode(-0.809016994374947, 0.587785252292473){8}
\psdot(-0.809016994374947, 0.587785252292473)
\uput[180](-0.809016994374947, 0.587785252292473){$\tiny 8$}
\pnode(-1.00000000000000, 0){9}
\psdot(-1.00000000000000, 0)
\uput[180](-1,0){$\tiny 9$}
\pnode(-0.809016994374947, -0.587785252292473){10}
\psdot(-0.809016994374947, -0.587785252292473)
\uput[180](-0.809016994374947, -0.587785252292473){$\tiny 10$}
\pnode(-0.309016994374948, -0.951056516295154){1}
\psdot(-0.309016994374948, -0.951056516295154)
\uput[-90](-0.309016994374948, -0.951056516295154){$\tiny 1$}
\pnode(0.309016994374947, -0.951056516295154){2}
\psdot(0.309016994374947, -0.951056516295154)
\uput[-90](0.309016994374947, -0.951056516295154){$\tiny 2$}
\pnode(0.809016994374947, -0.587785252292473){3}
\psdot(0.809016994374947, -0.587785252292473)
\uput[0](0.809016994374947, -0.587785252292473){$\tiny 3$}
\pnode(1.00000000000000, 0){4}
\psdot(1.00000000000000, 0)
\uput[0](1,0){$\tiny 4$}
\pnode(0.809016994374947, 0.587785252292473){5}
\psdot(0.809016994374947, 0.587785252292473)
\uput[0](0.809016994374947, 0.587785252292473){$\tiny 5$}
\pnode(0.309016994374948, 0.951056516295153){6}
\psdot(0.309016994374948, 0.951056516295153)
\uput[90](0.309016994374948, 0.951056516295153){$\tiny 6$}
\ncline[linecolor=blue,linewidth=.02]{1}{2}
\ncline[linecolor=blue,linewidth=.02]{2}{3}
\ncline[linecolor=blue,linewidth=.02]{3}{4}
\ncline[linecolor=blue,linewidth=.02]{4}{5}
\ncline[linecolor=blue,linewidth=.02]{5}{6}
\ncline[linecolor=blue,linewidth=.02]{6}{7}
\ncline[linecolor=blue,linewidth=.02]{7}{8}
\ncline[linecolor=blue,linewidth=.02]{8}{9}
\ncline[linecolor=blue,linewidth=.02]{9}{10}
\ncline[linecolor=blue,linewidth=.02]{10}{1}
\ncline[linecolor=blue,linewidth=.02]{1}{4}
\ncline[linecolor=blue,linewidth=.02]{5}{8}
\ncline[linecolor=blue,linewidth=.02]{5}{10}
  \end{pspicture}}
\uput[-90](-2,-1.2){\large $q$}
      \rput(1,0){%
  \begin{pspicture}(-1,-1)(1,1)
\pscircle(0,0){1}
\pnode(-0.309016994374947, 0.951056516295154){7}
\psdot(-0.309016994374947, 0.951056516295154)
\uput[90](-0.309016994374947, 0.951056516295154){$\tiny 7$}
\pnode(-0.809016994374947, 0.587785252292473){8}
\psdot(-0.809016994374947, 0.587785252292473)
\uput[180](-0.809016994374947, 0.587785252292473){$\tiny 8$}
\pnode(-1.00000000000000, 0){9}
\psdot(-1.00000000000000, 0)
\uput[180](-1,0){$\tiny 9$}
\pnode(-0.809016994374947, -0.587785252292473){10}
\psdot(-0.809016994374947, -0.587785252292473)
\uput[180](-0.809016994374947, -0.587785252292473){$\tiny 10$}
\pnode(-0.309016994374948, -0.951056516295154){1}
\psdot(-0.309016994374948, -0.951056516295154)
\uput[-90](-0.309016994374948, -0.951056516295154){$\tiny 1$}
\pnode(0.309016994374947, -0.951056516295154){2}
\psdot(0.309016994374947, -0.951056516295154)
\uput[-90](0.309016994374947, -0.951056516295154){$\tiny 2$}
\pnode(0.809016994374947, -0.587785252292473){3}
\psdot(0.809016994374947, -0.587785252292473)
\uput[0](0.809016994374947, -0.587785252292473){$\tiny 3$}
\pnode(1.00000000000000, 0){4}
\psdot(1.00000000000000, 0)
\uput[0](1,0){$\tiny 4$}
\pnode(0.809016994374947, 0.587785252292473){5}
\psdot(0.809016994374947, 0.587785252292473)
\uput[0](0.809016994374947, 0.587785252292473){$\tiny 5$}
\pnode(0.309016994374948, 0.951056516295153){6}
\psdot(0.309016994374948, 0.951056516295153)
\uput[90](0.309016994374948, 0.951056516295153){$\tiny 6$}
\ncline[linecolor=blue,linewidth=.02]{1}{2}
\ncline[linecolor=blue,linewidth=.02]{2}{3}
\ncline[linecolor=blue,linewidth=.02]{3}{4}
\ncline[linecolor=blue,linewidth=.02]{4}{5}
\ncline[linecolor=blue,linewidth=.02]{5}{6}
\ncline[linecolor=blue,linewidth=.02]{6}{7}
\ncline[linecolor=blue,linewidth=.02]{7}{8}
\ncline[linecolor=blue,linewidth=.02]{8}{9}
\ncline[linecolor=blue,linewidth=.02]{9}{10}
\ncline[linecolor=blue,linewidth=.02]{10}{1}
\ncline[linecolor=blue,linewidth=.02]{1}{4}
\ncline[linecolor=blue,linewidth=.02]{5}{8}
\ncline[linecolor=blue,linewidth=.02]{5}{10}
\ncline[linecolor=green]{1}{3}
\ncline[linecolor=green]{1}{5}
\ncline[linecolor=green]{9}{5}
\ncline[linecolor=green]{7}{5}
  \end{pspicture}}
     \rput(3.8,0){%
\psset{unit=.6}
  \begin{pspicture}(-1,-1)(1,1)
\pscircle(0,0){1}
\pnode(-0.309016994374947, 0.951056516295154){7}
\psdot(-0.309016994374947, 0.951056516295154)
\uput[90](-0.309016994374947, 0.951056516295154){$\tiny 4$}
\pnode(-1.00000000000000, 0){9}
\psdot(-1.00000000000000, 0)
\uput[180](-1,0){$\tiny 5$}
\pnode(-0.309016994374948, -0.951056516295154){1}
\psdot(-0.309016994374948, -0.951056516295154)
\uput[-90](-0.309016994374948, -0.951056516295154){$\tiny 1$}
\pnode(0.809016994374947, -0.587785252292473){3}
\psdot(0.809016994374947, -0.587785252292473)
\uput[0](0.809016994374947, -0.587785252292473){$\tiny 2$}
\pnode(0.809016994374947, 0.587785252292473){5}
\psdot(0.809016994374947, 0.587785252292473)
\uput[0](0.809016994374947, 0.587785252292473){$\tiny 3$}
\ncline[linecolor=green]{1}{3}
\ncline[linecolor=green]{1}{5}
\ncline[linecolor=green]{9}{5}
\ncline[linecolor=green]{7}{5}
  \end{pspicture}}
\uput[-90](4,-1.2){\large $\phi(q)$}
     \rput(-2,-3){%
\psset{unit=.6}
  \begin{pspicture}(-1,-1)(1,1)
\pscircle(0,0){1}
\pnode(-0.309016994374947, 0.951056516295154){7}
\psdot(-0.309016994374947, 0.951056516295154)
\uput[90](-0.309016994374947, 0.951056516295154){$\tiny 4$}
\pnode(-1.00000000000000, 0){9}
\psdot(-1.00000000000000, 0)
\uput[180](-1,0){$\tiny 5$}
\pnode(-0.309016994374948, -0.951056516295154){1}
\psdot(-0.309016994374948, -0.951056516295154)
\uput[-90](-0.309016994374948, -0.951056516295154){$\tiny 1$}
\pnode(0.809016994374947, -0.587785252292473){3}
\psdot(0.809016994374947, -0.587785252292473)
\uput[0](0.809016994374947, -0.587785252292473){$\tiny 2$}
\pnode(0.809016994374947, 0.587785252292473){5}
\psdot(0.809016994374947, 0.587785252292473)
\uput[0](0.809016994374947, 0.587785252292473){$\tiny 3$}
\ncline[linecolor=green]{1}{3}
\ncline[linecolor=green]{1}{5}
\ncline[linecolor=green]{9}{5}
\ncline[linecolor=green]{7}{5}
  \end{pspicture}}
      \rput(1,-3){%
  \begin{pspicture}(-1,-1)(1,1)
\pscircle[linecolor=blue](0,0){1}
\pnode(-0.309016994374947, 0.951056516295154){7}
\psdot(-0.309016994374947, 0.951056516295154)
\uput[90](-0.309016994374947, 0.951056516295154){$\tiny 4$}
\pnode(-0.809016994374947, 0.587785252292473){8}
\psdot(-0.809016994374947, 0.587785252292473)
\uput[180](-0.809016994374947, 0.587785252292473){$\tiny 4^{*}$}
\pnode(-1.00000000000000, 0){9}
\psdot(-1.00000000000000, 0)
\uput[180](-1,0){$\tiny 5$}
\pnode(-0.809016994374947, -0.587785252292473){10}
\psdot(-0.809016994374947, -0.587785252292473)
\uput[180](-0.809016994374947, -0.587785252292473){$\tiny 5^{*}$}
\pnode(-0.309016994374948, -0.951056516295154){1}
\psdot(-0.309016994374948, -0.951056516295154)
\uput[-90](-0.309016994374948, -0.951056516295154){$\tiny 1$}
\pnode(0.309016994374947, -0.951056516295154){2}
\psdot(0.309016994374947, -0.951056516295154)
\uput[-90](0.309016994374947, -0.951056516295154){$\tiny 1^{*}$}
\pnode(0.809016994374947, -0.587785252292473){3}
\psdot(0.809016994374947, -0.587785252292473)
\uput[0](0.809016994374947, -0.587785252292473){$\tiny 2$}
\pnode(1.00000000000000, 0){4}
\psdot(1.00000000000000, 0)
\uput[0](1,0){$\tiny 2^{*}$}
\pnode(0.809016994374947, 0.587785252292473){5}
\psdot(0.809016994374947, 0.587785252292473)
\uput[0](0.809016994374947, 0.587785252292473){$\tiny 3$}
\pnode(0.309016994374948, 0.951056516295153){6}
\psdot(0.309016994374948, 0.951056516295153)
\uput[90](0.309016994374948, 0.951056516295153){$\tiny 3^{*}$}
\ncline[linecolor=blue,linestyle=dashed]{1}{4}
\ncline[linecolor=blue,linestyle=dashed]{5}{8}
\ncline[linecolor=blue,linestyle=dashed]{5}{10}
\ncline[linecolor=green]{1}{3}
\ncline[linecolor=green]{1}{5}
\ncline[linecolor=green]{9}{5}
\ncline[linecolor=green]{7}{5}
\ncline[linecolor=red]{2}{4}
\ncline[linecolor=red]{10}{8}
\ncline[linecolor=red]{6}{8}
\ncline[linecolor=red]{4}{10}
  \end{pspicture}}
      \rput(3.9,-3){%
  \begin{pspicture}(-1,-1)(1,1)
\pnode(-0.309016994374947, 0.951056516295154){7}
\psdot(-0.309016994374947, 0.951056516295154)
\uput[90](-0.309016994374947, 0.951056516295154){$\tiny 7$}
\pnode(-0.809016994374947, 0.587785252292473){8}
\psdot(-0.809016994374947, 0.587785252292473)
\uput[180](-0.809016994374947, 0.587785252292473){$\tiny 8$}
\pnode(-1.00000000000000, 0){9}
\psdot(-1.00000000000000, 0)
\uput[180](-1,0){$\tiny 9$}
\pnode(-0.809016994374947, -0.587785252292473){10}
\psdot(-0.809016994374947, -0.587785252292473)
\uput[180](-0.809016994374947, -0.587785252292473){$\tiny 10$}
\pnode(-0.309016994374948, -0.951056516295154){1}
\psdot(-0.309016994374948, -0.951056516295154)
\uput[-90](-0.309016994374948, -0.951056516295154){$\tiny 1$}
\pnode(0.309016994374947, -0.951056516295154){2}
\psdot(0.309016994374947, -0.951056516295154)
\uput[-90](0.309016994374947, -0.951056516295154){$\tiny 2$}
\pnode(0.809016994374947, -0.587785252292473){3}
\psdot(0.809016994374947, -0.587785252292473)
\uput[0](0.809016994374947, -0.587785252292473){$\tiny 3$}
\pnode(1.00000000000000, 0){4}
\psdot(1.00000000000000, 0)
\uput[0](1,0){$\tiny 4$}
\pnode(0.809016994374947, 0.587785252292473){5}
\psdot(0.809016994374947, 0.587785252292473)
\uput[0](0.809016994374947, 0.587785252292473){$\tiny 5$}
\pnode(0.309016994374948, 0.951056516295153){6}
\psdot(0.309016994374948, 0.951056516295153)
\uput[90](0.309016994374948, 0.951056516295153){$\tiny 6$}
\ncline[linecolor=blue,linewidth=.02]{1}{2}
\ncline[linecolor=blue,linewidth=.02]{2}{3}
\ncline[linecolor=blue,linewidth=.02]{3}{4}
\ncline[linecolor=blue,linewidth=.02]{4}{5}
\ncline[linecolor=blue,linewidth=.02]{5}{6}
\ncline[linecolor=blue,linewidth=.02]{6}{7}
\ncline[linecolor=blue,linewidth=.02]{7}{8}
\ncline[linecolor=blue,linewidth=.02]{8}{9}
\ncline[linecolor=blue,linewidth=.02]{9}{10}
\ncline[linecolor=blue,linewidth=.02]{10}{1}
\ncline[linecolor=blue,linewidth=.02]{1}{4}
\ncline[linecolor=blue,linewidth=.02]{5}{8}
\ncline[linecolor=blue,linewidth=.02]{5}{10}
  \end{pspicture}}
\uput[-90](-2,-4.2){\large $t$}
\uput[-90](3.9,-4.2){\large $\phi^{-1}(t)$}
    \end{pspicture}
  \caption{The construction of $\phi\co \mathcal{Q}_m \to \mathcal{N}_m$ (top) and its inverse (bottom).}
  \label{fig:phi}
\end{figure}

To see that the above construction does indeed give the structural
bijection $\mathcal{Q} \to \mathcal{N}$, 
notice that this is obviously true for
$m=0,1$ and, as shown in Figure~\ref{fig:duluqbij}, the ternary
operations agree.

It turns out that
$\phi$ is not only duality preserving but also equivariant with
respect to the respective dihedral group actions (see
Proposition~\ref{prop:kapparef} for the action of the dihedral group
$\mathrm{D}_{2n}$ on $\mathcal{N}_m$).

Indeed, notice that the analogous construction using even diagonals
will give $\kappa(\phi(q))$, thus showing that $\kappa$ is the
push-forward of rotation by $\pi/n$.  Notice also that $r_{1\,2}$, the
reflection across the perpendicular bisector of the root edge $1\,2$,
interchanges ``even'' and ``odd'' diagonals, and maps the vertex
labeled $i$ to the vertex labeled $2n+3-i \pmod{2n}$, so that $2i-1$
(the label of the $i$-th vertex of $t$) is mapped to $2(n+2-1)$ (the
label of the $i^{*}$-th label of $t^{*}$).  Thus $r$ is the push
forward of $r_{1\,2}$.  So we have:
\begin{thm}
  \label{thm:d2nequivariance} The bijection $\phi$ is $\mathcal{D}_{2n}$-equivariant.
\end{thm}

\subsubsection{Relation of $\phi$ to Schaeffer's bijection}
\label{sec:schaeffer}

The bijection $\phi$ is closely related to a bijection between rooted
quadrangulations of the plane with $m$ faces and well-labeled trees
with $m$ edges defined in~\cite{Schaeffer} (see
also~\cite{ChassaingSchaeffer2004}).  A \emph{rooted quadrangulation}
is a map of the sphere $Q$ where every face has degree $4$, together
with a distinguished \emph{oriented} edge on the boundary of the
unbound face called the \emph{root edge}. The starting vertex of the
the root edge of $Q$ is called the \emph{root}.  A \emph{well-labeled
  tree} is an ordered tree with its vertices labeled by positive
integers in such a way that the labels of two adjacent vertices differ
by one and the root is labeled $1$.

The Schaefeer bijection $S$ is defined as follows: let $Q$ be a rooted
quadrangulation.  Start by labeling the vertices of $Q$ with their
distance form the root vertex $v_0$.  Around every face of $Q$ at
least one pair of opposite vertices have the same label.  Call a face
\emph{simple} if only one pair of opposite vertices has the same
labels, and \emph{confluent} otherwise.  The image of $Q$ is obtained
by taking the diagonal connecting the two vertices with the maximum
degree for confluent faces, while for a simple face $f$ we select the
edge incident to the vertex with maximal label that is leaving $f$ on
its left.  The root of $S(Q)$ is the first edge incident to the
endpoint of the root of $Q$, counterclockwise starting from the root
of $Q$.


To express $\phi$ in terms of
$S$ we construct a rooted quadrangulation of the plane associated with
a rooted quadrangular dissection $q$ of a
$2n$-gon by adding an extra vertex at a point in the exterior of the
polygon and connecting it by an edge to all the even vertices.  When
we compute the distances from the new vertex, the even vertices are at
distance $1$ and the odd vertices at distance
$2$.  So all the cells of
$q$ are confluent faces, and all the new faces simple.  Therefore each
cell of
$q$ contributes its odd diagonal to the resulting well labeled tree,
while each of the new faces contributes the leftmost of the sides of
the polygon in its boundary.  The result is a well labeled tree where
all internal vertices have label $2$, all leaves have label
$1$ and each internal vertex is adjacent to exactly one leaf. Such
trees are in bijection with bipartisan trees, just delete all leaves
and use their position as a marker where the right children of every
non-root internal vertex end, and where the left children begin.  In
Figure~\ref{fig:schaefbij} we carry this construction for the
quadrangular dissection of Figure~\ref{fig:q12ndu}, the odd nc-tree is
shown in red, and the contributions of the new faces in magenta.
Clearly the nc-tree that corresponds to the bipartisan tree obtained
this way is $\phi(q)$.

\begin{figure}[ht]
  \centering
  \begin{pspicture}(-4,-9)(10,3)
    \rput(6.5,0){
      \psset{unit=1.7,arrowsize=.1}
      \begin{pspicture}(-1.26593, -1.26593)(3.26593, 1.26593)
        \pnode(-0.25882,-0.96593){1}
        \psdot(-0.25882,-0.96593)
        \pnode(0.25882,-0.96593){2}
        \psdot(0.25882,-0.96593)
        \pnode(0.70711,-0.70711){3}
        \psdot(0.70711,-0.70711)
        \pnode(0.96593,-0.25882){4}
        \psdot(0.96593,-0.25882)
        \pnode(0.96593,0.25882){5}
        \psdot(0.96593,0.25882)
        \pnode(0.70711,0.70711){6}
        \psdot(0.70711,0.70711)
        \pnode(0.25882,0.96593){7}
        \psdot(0.25882,0.96593)
        \pnode(-0.25882,0.96593){8}
        \psdot(-0.25882,0.96593)
        \pnode(-0.70711,0.70711){9}
        \psdot(-0.70711,0.70711)
        \pnode(-0.96593,0.25882){10}
        \psdot(-0.96593,0.25882)
        \pnode(-0.96593,-0.25882){11}
        \psdot(-0.96593,-0.25882)
        \pnode(-0.70711,-0.70711){12}
        \psdot(-0.70711,-0.70711)
        \uput[-90.000](0.25881905,-0.96592583){\small  $1$}
        \uput[-54.000](0.70710678,-0.70710678){\small $2$}
        \uput[-18.000](0.96592583,-0.25881905){\small $1$}
        \uput[18.000](0.96592583,0.25881905){\small $2$}
        \uput[54.000](0.70710678,0.70710678){\small $1$}
        \uput[90.000](0.25881905,0.96592583){\small $2$}
        \uput[-90.000](-0.25881905,-0.96592583){\small $2$}
        \uput[90.000](-0.25881905,0.96592583){\small $1$}
        \uput[126.000](-0.70710678,0.70710678){\small $2$}
        \uput[162.000](-0.96592583,0.25881905){\small $1$}
        \uput[198.000](-0.96592583,-0.25881905){\small $2$}
        \uput[234.000](-0.70710678,-0.70710678){\small $1$}
        \ncline[ArrowInside=->,ArrowInsidePos=.7,linecolor=blue]{1}{2}
        \ncline{1}{12}
        \ncline{2}{3}
        \ncline{2}{5}
        \ncline{3}{4}
        \ncline{4}{5}
        \ncline{5}{6}
        \ncline{5}{12}
        \ncline{6}{7}
        \ncline{7}{8}
        \ncline{7}{10}
        \ncline{7}{12}
        \ncline{8}{9}
        \ncline{9}{10}
        \ncline{10}{11}
        \ncline{11}{12}
        \ncline[linecolor=red,linestyle=dashed]{1}{5}
        \ncline[linecolor=red,linestyle=dashed]{3}{5}
        \ncline[linecolor=red,linestyle=dashed]{5}{7}
        \ncline[linecolor=red,linestyle=dashed]{7}{11}
        \ncline[linecolor=red,linestyle=dashed]{7}{9}
        \pnode(2.5,0.25882){0}
        \psdot(2.5,.25882)
        \uput[0](2.5,0.25882){$0$}
        \ncline[linecolor = green]{0}{4}
        \ncline[linecolor = green]{0}{6}
        \nccurve[angleA=135,angleB=60,ncurv=.7,linecolor=green]{0}{8}
        \nccurve[angleA=-135,angleB=-40,ncurv=.7,linecolor=green]{0}{2}
        \nccurve[angleA=105,angleB=105,ncurv=1,linecolor=green]{0}{10}
        \nccurve[angleA=-105,angleB=-80,ncurv=1.1,linecolor=green,ArrowInside=->,ArrowInsidePos=.7]{0}{12}
      \end{pspicture}}
        \rput(0,0){
      \psset{unit=1.7,arrowsize=.1}
      \begin{pspicture}(-1.26593, -1.26593)(3.26593, 1.26593)
        \pnode(-0.25882,-0.96593){1}
        \psdot(-0.25882,-0.96593)
        \pnode(0.25882,-0.96593){2}
        \psdot(0.25882,-0.96593)
        \pnode(0.70711,-0.70711){3}
        \psdot(0.70711,-0.70711)
        \pnode(0.96593,-0.25882){4}
        \psdot(0.96593,-0.25882)
        \pnode(0.96593,0.25882){5}
        \psdot(0.96593,0.25882)
        \pnode(0.70711,0.70711){6}
        \psdot(0.70711,0.70711)
        \pnode(0.25882,0.96593){7}
        \psdot(0.25882,0.96593)
        \pnode(-0.25882,0.96593){8}
        \psdot(-0.25882,0.96593)
        \pnode(-0.70711,0.70711){9}
        \psdot(-0.70711,0.70711)
        \pnode(-0.96593,0.25882){10}
        \psdot(-0.96593,0.25882)
        \pnode(-0.96593,-0.25882){11}
        \psdot(-0.96593,-0.25882)
        \pnode(-0.70711,-0.70711){12}
        \psdot(-0.70711,-0.70711)
        \uput[-90.000](0.25881905,-0.96592583){\small $2$}
        \uput[-54.000](0.70710678,-0.70710678){\small $3$}
        \uput[-18.000](0.96592583,-0.25881905){\small $4$}
        \uput[18.000](0.96592583,0.25881905){\small $5$}
        \uput[54.000](0.70710678,0.70710678){\small $6$}
        \uput[90.000](0.25881905,0.96592583){\small $7$}
        \uput[-90.000](-0.25881905,-0.96592583){\small $1$}
        \uput[90.000](-0.25881905,0.96592583){\small $8$}
        \uput[126.000](-0.70710678,0.70710678){\small $9$}
        \uput[162.000](-0.96592583,0.25881905){\small $10$}
        \uput[198.000](-0.96592583,-0.25881905){\small $11$}
        \uput[234.000](-0.70710678,-0.70710678){\small $12$}
        \ncline[ArrowInside=->,ArrowInsidePos=.7,linecolor=blue]{1}{2}
        \ncline{1}{12}
        \ncline{2}{3}
        \ncline{2}{5}
        \ncline{3}{4}
        \ncline{4}{5}
        \ncline{5}{6}
        \ncline{5}{12}
        \ncline{6}{7}
        \ncline{7}{8}
        \ncline{7}{10}
        \ncline{7}{12}
        \ncline{8}{9}
        \ncline{9}{10}
        \ncline{10}{11}
        \ncline{11}{12}
      \end{pspicture}}
    \rput(0,-6){
      \psset{unit=1.7,arrowsize=.1}
      \begin{pspicture}(-1.26593, -1.26593)(3.26593, 1.26593)
        \pnode(-0.25882,-0.96593){1}
        \psdot(-0.25882,-0.96593)
        \pnode(0.25882,-0.96593){2}
        \psdot(0.25882,-0.96593)
        \pnode(0.70711,-0.70711){3}
        \psdot(0.70711,-0.70711)
        \pnode(0.96593,-0.25882){4}
        \psdot(0.96593,-0.25882)
        \pnode(0.96593,0.25882){5}
        \psdot(0.96593,0.25882)
        \pnode(0.70711,0.70711){6}
        \psdot(0.70711,0.70711)
        \pnode(0.25882,0.96593){7}
        \psdot(0.25882,0.96593)
        \pnode(-0.25882,0.96593){8}
        \psdot(-0.25882,0.96593)
        \pnode(-0.70711,0.70711){9}
        \psdot(-0.70711,0.70711)
        \pnode(-0.96593,0.25882){10}
        \psdot(-0.96593,0.25882)
        \pnode(-0.96593,-0.25882){11}
        \psdot(-0.96593,-0.25882)
        \pnode(-0.70711,-0.70711){12}
        \psdot(-0.70711,-0.70711)
        \uput[-90.000](0.25881905,-0.96592583){\small  $1$}
        \uput[-54.000](0.70710678,-0.70710678){\small $2$}
        \uput[-18.000](0.96592583,-0.25881905){\small $1$}
        \uput[18.000](0.96592583,0.25881905){\small $2$}
        \uput[54.000](0.70710678,0.70710678){\small $1$}
        \uput[90.000](0.25881905,0.96592583){\small $2$}
        \uput[-90.000](-0.25881905,-0.96592583){\small $2$}
        \uput[90.000](-0.25881905,0.96592583){\small $1$}
        \uput[126.000](-0.70710678,0.70710678){\small $2$}
        \uput[162.000](-0.96592583,0.25881905){\small $1$}
        \uput[198.000](-0.96592583,-0.25881905){\small $2$}
        \uput[234.000](-0.70710678,-0.70710678){\small $1$}
        \ncline[linecolor=magenta,linewidth=.03]{1}{12}
        \ncline[linecolor=magenta]{2}{3}
        \ncline[linecolor=magenta]{4}{5}
        \ncline[linecolor=magenta]{6}{7}
        \ncline[linecolor=magenta]{8}{9}
        \ncline[linecolor=magenta]{10}{11}
        \ncline[linecolor=red]{1}{5}
        \ncline[linecolor=red]{3}{5}
        \ncline[linecolor=red]{5}{7}
        \ncline[linecolor=red]{7}{11}
        \ncline[linecolor=red]{7}{9}
      \end{pspicture}}
    \rput(3.2,-6.5){
      \psset{unit=.8}
      \begin{pspicture}(-0.30000, -0.30000)(3.30000, 5.30000)
        \pnode(2.00000,1.00000){1}
        \psdot(2.00000,1.00000)
        \pnode(3.00000,4.00000){2}
        \psdot(3.00000,4.00000)
        \pnode(3.00000,3.00000){3}
        \psdot(3.00000,3.00000)
        \pnode(2.00000,3.00000){4}
        \psdot(2.00000,3.00000)
        \pnode(2.00000,2.00000){5}
        \psdot(2.00000,2.00000)
        \pnode(2.00000,4.00000){6}
        \psdot(2.00000,4.00000)
        \pnode(1.00000,3.00000){7}
        \psdot(1.00000,3.00000)
        \pnode(1.00000,5.00000){8}
        \psdot(1.00000,5.00000)
        \pnode(1.00000,4.00000){9}
        \psdot(1.00000,4.00000)
        \pnode(0.00000,5.00000){10}
        \psdot(0.00000,5.00000)
        \pnode(0.00000,4.00000){11}
        \psdot(0.00000,4.00000)
        \pnode(2.00000,0.00000){12}
        \psdot(2.00000,0.00000)
        \ncline[linecolor=red]{1}{5}
        \ncline[linecolor=magenta]{1}{12}
        \ncline[linecolor=magenta]{2}{3}
        \ncline[linecolor=red]{3}{5}
        \ncline[linecolor=magenta]{4}{5}
        \ncline[linecolor=red]{5}{7}
        \ncline[linecolor=magenta]{6}{7}
        \ncline[linecolor=red]{7}{9}
        \ncline[linecolor=red]{7}{11}
        \ncline[linecolor=magenta]{8}{9}
        \ncline[linecolor=magenta]{10}{11}
      \end{pspicture}}
        \rput(8,-6.5){
      \psset{unit=.8}
      \begin{pspicture}(-0.30000, -0.30000)(3.30000, 5.30000)
        \pnode(2.00000,1.00000){1}
        \psdot(2.00000,1.00000)
        \pnode(3.00000,3.00000){3}
        \psdot(3.00000,3.00000)
        \pnode(2.00000,2.00000){5}
        \psdot(2.00000,2.00000)
        \pnode(1.00000,3.00000){7}
        \psdot(1.00000,3.00000)
        \pnode(.50000,4.00000){9}
        \psdot(.50000,4.00000)
        \pnode(-.50000,4.00000){11}
        \psdot(-.50000,4.00000)
        \ncline[linecolor=red]{1}{5}
        \ncline[linecolor=red]{3}{5}
        \ncline[linecolor=red]{5}{7}
        \ncline[linecolor=red]{7}{9}
        \ncline[linecolor=red]{7}{11}
      \end{pspicture}}
  \end{pspicture}
  \caption{$\phi$ in terms of Scaeffer's bijection.}
  \label{fig:schaefbij}
\end{figure}
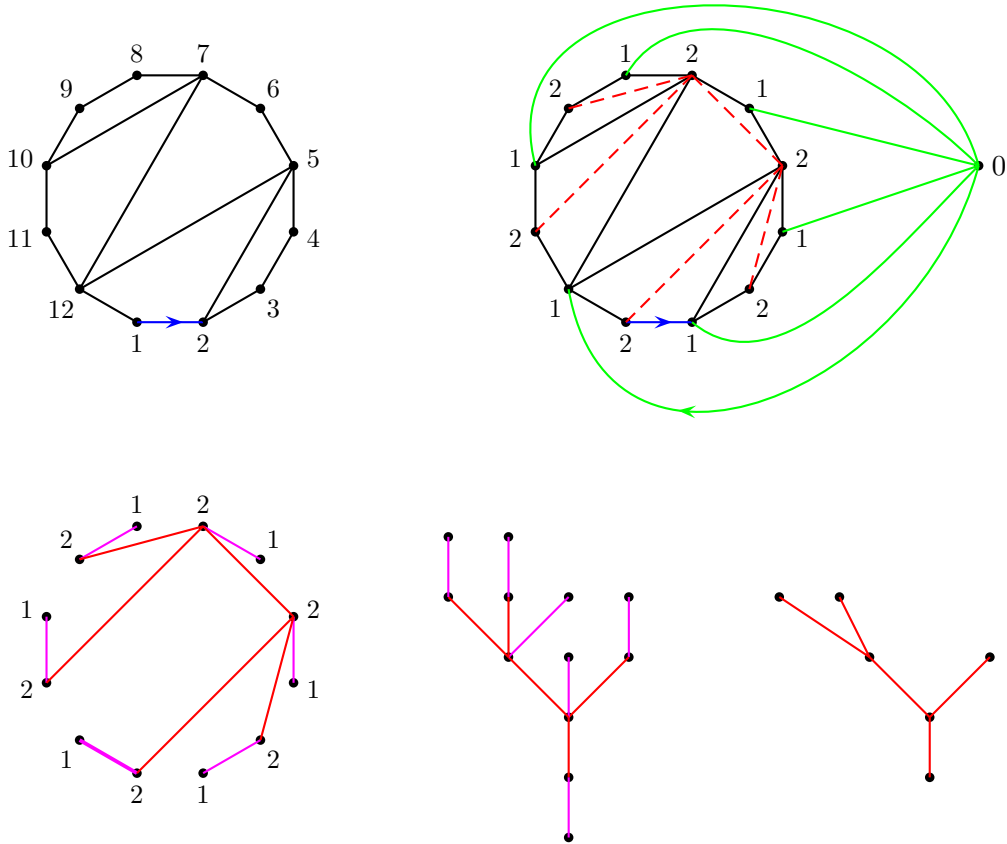

I would like to thank the anonymous referee of a previous version of
this paper for bringing this connection to my attention and providing
the construction from quadrangular dissections of a polygon
to quadrangulations of the plane.

\section{Enumerations}
\label{sec:enum}

Next we take a closer look at the action of the dihedral group
$\mathrm{D}_{2n}$ on $\mathcal{Q}_m$, where as usual $n = m+1$. If
$\kappa$ stands for the rotation by $\frac{\pi}{n}$ radians, and $r$
for the reflection across the bisector of the root edge $1\,2$ then
$$
D_{2n} = \left\{ \kappa^i\,r^j : i = 0, \ldots, 2n-1,\,\,\, j = 0,1
\right\}
$$
and the elements with $i=1$ are rotations, those with $i=0$ and
$j\neq 0$ are reflections, while, of course, $i=j=0$ gives the identity.
In particular $s = \kappa\,r$ is a reflection with axis that passes
through the vertex labeled $1$. Since $2n$ is even there are two
conjugacy classes of reflections those whose axis passes through two
diametrically opposite vertices, and those whose axis passes through
the midpoints of two diametrically opposite edges. The first class is
represented by $r$ and the second by $s$.

Notice that the subgroup $\left\langle \kappa^2, s \right\rangle$ is
isomorphic to $\mathrm{D}_n$, and the restriction of the
$\mathrm{D}_{2n}$-action on $\mathcal{Q}_m$ on that subgroup is
carried by the structural bijection $\phi$ to the standard action of
$\mathrm{D}_n$ on $\mathcal{N}_m$, where $\kappa^2$ is rotation by
$\frac{2\pi}{n}$ radians and $s$ is the reflection across the diameter
that passes through $1$, see Proposition~\ref{prop:kapparef} and
Theorem~\ref{thm:d2nequivariance}.

\begin{thm}
  \label{thm:fp} Every reflection in $\mathrm{D}_{2n}$ fixes $s_m$ elements of
  $\mathcal{Q}_m$.  Rotation by $\pi$ radians has $(m+1) s_m$ fixed points
  if $m$ is even, and $\frac{(m+1) s_m}{2}$ if $m$ is odd.  When
  $m \equiv 1 \pmod{4}$, rotations by $\pm \frac{\pi}{2}$ have
  $\frac{m+1}{2} s_{\frac{m+1}{2}}$ fixed points.  No other rotation has
  fixed points.
\end{thm}

\begin{proof}
  The basic observation is that the center of the polygon is fixed by
  all rotations and reflections, and for a quadrangular dissection $q$
  of a $2n$-gon fixed by an element of $\mathrm{D}_{2n}$ we have two
  cases: the center is in the interior of a cell or it's the midpoint
  of a dissecting diagonal (which has then to be a diameter of the
  circumscribed circle) of $q$, and that cell or dissecting diagonal
  has then to be invariant.

  We first examine rotations.  If the center is on a dissecting
  diameter, then since all dissecting diagonals connect vertices of
  opposite parity, this can happen if and only if $m$ is even.  This
  diameter has to be invariant under the rotation and it follows that
  the rotation is by $\pi$ radians.  Then $q$ consists of two
  dissections (one a rotation by $\pi$ of the other) of the
  $(n+1)$-gon, glued together along an edge.  See
  Figure~\ref{fig:octadecarotinv} for an example of a rotation
  invariant dissection of an octadecagon: the diameter $1\,\,10$ is a
  dissecting diagonal, and $q$ consists of a dissection of a decagon,
  glued along an edge to its rotation.

  There are $n = m+1$ diameters that could be dissecting diagonals, and
  there are $\nu_{m+2} = s_{2(m+1)}$ dissections of the $(n+1)$-gon.  It
  follows that the central rotation by $\pi$ has $(m+1) s_m$ fixed points,
  and no other rotation has fixed points.

  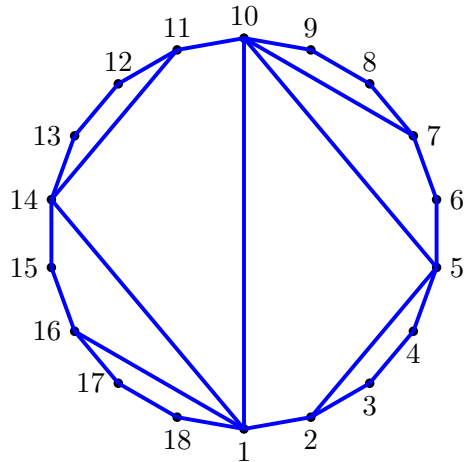
\begin{figure}[ht]
    \centering
\psset{unit=2.6}
    \begin{pspicture}(-1.5,-1.2)(1.5,1.2)
\pnode(0.0000000000, 1.0000000000){10}
\psdot(0.0000000000, 1.0000000000)
\uput[90](0.0000000000, 1.0000000000){$\small 10$}
\pnode(-0.3420201433, 0.9396926208){11}
\psdot(-0.3420201433, 0.9396926208)
\uput[90](-0.3420201433, 0.9396926208){$\small 11$}
\pnode(-0.6427876097, 0.7660444431){12}
\psdot(-0.6427876097, 0.7660444431)
\uput[90](-0.6427876097, 0.7660444431){$\small 12$}
\pnode(-0.8660254038, 0.5000000000){13}
\psdot(-0.8660254038, 0.5000000000)
\uput[180](-0.8660254038, 0.5000000000){$\small 13$}
\pnode(-0.9848077530, 0.1736481777){14}
\psdot(-0.9848077530, 0.1736481777)
\uput[180](-0.9848077530, 0.1736481777){$\small 14$}
\pnode(-0.9848077530, -0.1736481777){15}
\psdot(-0.9848077530, -0.1736481777)
\uput[-180](-0.9848077530, -0.1736481777){$\small 15$}
\pnode(-0.8660254038, -0.5000000000){16}
\psdot(-0.8660254038, -0.5000000000)
\uput[180](-0.8660254038, -0.5000000000){$\small 16$}
\pnode(-0.6427876097, -0.7660444431){17}
\psdot(-0.6427876097, -0.7660444431)
\uput[180](-0.6427876097, -0.7660444431){$\small 17$}
\pnode(-0.3420201433, -0.9396926208){18}
\psdot(-0.3420201433, -0.9396926208)
\uput[-90](-0.3420201433, -0.9396926208){$\small 18$}
\pnode(-0.0000000000, -1.0000000000){1}
\psdot(-0.0000000000, -1.0000000000)
\uput[-90](-0.0000000000, -1.0000000000){$\small 1$}
\pnode(0.3420201433, -0.9396926208){2}
\psdot(0.3420201433, -0.9396926208)
\uput[-90](0.3420201433, -0.9396926208){$\small 2$}
\pnode(0.6427876097, -0.7660444431){3}
\psdot(0.6427876097, -0.7660444431)
\uput[-90](0.6427876097, -0.7660444431){$\small 3$}
\pnode(0.8660254038, -0.5000000000){4}
\psdot(0.8660254038, -0.5000000000)
\uput[-90](0.8660254038, -0.5000000000){$\small 4$}
\pnode(0.9848077530, -0.1736481777){5}
\psdot(0.9848077530, -0.1736481777)
\uput[0](0.9848077530, -0.1736481777){$\small 5$}
\pnode(0.9848077530, 0.1736481777){6}
\psdot(0.9848077530, 0.1736481777)
\uput[0](0.9848077530, 0.1736481777){$\small 6$}
\pnode(0.8660254038, 0.5000000000){7}
\psdot(0.8660254038, 0.5000000000)
\uput[0](0.8660254038, 0.5000000000){$\small 7$}
\pnode(0.6427876097, 0.7660444431){8}
\psdot(0.6427876097, 0.7660444431)
\uput[90](0.6427876097, 0.7660444431){$\small 8$}
\pnode(0.3420201433, 0.9396926208){9}
\psdot(0.3420201433, 0.9396926208)
\uput[90](0.3420201433, 0.9396926208){$\small 9$}
\ncline[linecolor=blue,linewidth=.02]{1}{2}
\ncline[linecolor=blue,linewidth=.02]{2}{3}
\ncline[linecolor=blue,linewidth=.02]{3}{4}
\ncline[linecolor=blue,linewidth=.02]{4}{5}
\ncline[linecolor=blue,linewidth=.02]{5}{6}
\ncline[linecolor=blue,linewidth=.02]{6}{7}
\ncline[linecolor=blue,linewidth=.02]{7}{8}
\ncline[linecolor=blue,linewidth=.02]{8}{9}
\ncline[linecolor=blue,linewidth=.02]{9}{10}
\ncline[linecolor=blue,linewidth=.02]{10}{11}
\ncline[linecolor=blue,linewidth=.02]{11}{12}
\ncline[linecolor=blue,linewidth=.02]{12}{13}
\ncline[linecolor=blue,linewidth=.02]{13}{14}
\ncline[linecolor=blue,linewidth=.02]{14}{15}
\ncline[linecolor=blue,linewidth=.02]{15}{16}
\ncline[linecolor=blue,linewidth=.02]{16}{17}
\ncline[linecolor=blue,linewidth=.02]{17}{18}
\ncline[linecolor=blue,linewidth=.02]{18}{1}
\ncline[linecolor=blue,linewidth=.02]{1}{10}
\ncline[linecolor=blue,linewidth=.02]{5}{10}
\ncline[linecolor=blue,linewidth=.02]{5}{2}
\ncline[linecolor=blue,linewidth=.02]{7}{10}
\ncline[linecolor=blue,linewidth=.02]{1}{16}
\ncline[linecolor=blue,linewidth=.02]{1}{14}
\ncline[linecolor=blue,linewidth=.02]{11}{14}
\end{pspicture}
    \caption{A rotation invariant quadrangular dissection of the octadecagon.}
    \label{fig:octadecarotinv}
  \end{figure}

  If on the other hand the center belongs to an invariant cell, then
  the two diagonal of the cell are diameters and the rotation either
  fixes them or rotates one into the other.  In the first case we have
  rotation by $\pi$ and in the second by $\frac{\pi}{2}$.  The number
  of cells that are to the south or east of the invariant cell equals
  the number of cells to the west or north, and thus there is is an
  odd number of total cells.  It follows that this case occurs only
  when $m$ is odd. For a dissection invariant under rotation by $\pi$
  radians one can see that it consists of a pair of smaller
  dissections (not necessarily both non-empty), one to the south which
  rotates to the one in the north, and one to the east that rotates to
  the one on the west.  See Figure~\ref{fig:decahexarotinv} for two
  examples in the case $m=7$.

  It follows that for a given invariant cell, there are as many 
  invariant dissections as pairs of dissections with
  total number of cells equal to $2m$, which is counted by
  $s_{2m}$.  Now an invariant cell is determined by a pair
  of invariant diagonals (the two dual edges of the pair of
  dual non-crossing trees) and there are $\frac{m+1}{2}$ such
  pairs of dual edges.

  Thus rotation by $\pi$ has $\frac{(m+1)s_{m+1}}{2}$ fixed points.

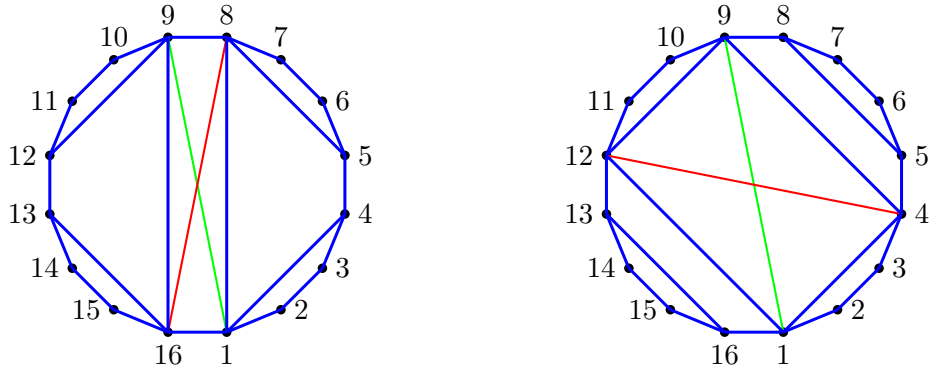
\begin{figure}[ht]
  \centering
  \psset{unit=2}
    \begin{pspicture}(-3.2,-1.6)(3.6,1.6)
    \rput(-1.6,0){%
  \begin{pspicture}(-1,-1)(1,1)
\pnode(0.195090322016128, -0.980785280403230){1}
\psdot(0.195090322016128, -0.980785280403230)
\uput[-90](0.195090322016128, -0.980785280403230){$\small 1$}
\pnode(0.555570233019602, -0.831469612302545){2}
\psdot(0.555570233019602, -0.831469612302545)
\uput[0](0.555570233019602, -0.831469612302545){$\small 2$}
\pnode(0.831469612302545, -0.555570233019602){3}
\psdot(0.831469612302545, -0.555570233019602)
\uput[0](0.831469612302545, -0.555570233019602){$\small 3$}
\pnode(0.980785280403230, -0.195090322016129){4}
\psdot(0.980785280403230, -0.195090322016129)
\uput[0](0.980785280403230, -0.195090322016129){$\small 4$}
\pnode(0.980785280403230, 0.195090322016128){5}
\psdot(0.980785280403230, 0.195090322016128)
\uput[0](0.980785280403230, 0.195090322016128){$\small 5$}
\pnode(0.831469612302545, 0.555570233019602){6}
\psdot(0.831469612302545, 0.555570233019602)
\uput[0](0.831469612302545, 0.555570233019602){$\small 6$}
\pnode(0.555570233019602, 0.831469612302545){7}
\psdot(0.555570233019602, 0.831469612302545)
\uput[90](0.555570233019602, 0.831469612302545){$\small 7$}
\pnode(0.195090322016129, 0.980785280403230){8}
\psdot(0.195090322016129, 0.980785280403230)
\uput[90](0.195090322016129, 0.980785280403230){$\small 8$}
\pnode(-0.195090322016128, 0.980785280403230){9}
\psdot(-0.195090322016128, 0.980785280403230)
\uput[90](-0.195090322016128, 0.980785280403230){$\small 9$}
\pnode(-0.555570233019602, 0.831469612302545){10}
\psdot(-0.555570233019602, 0.831469612302545)
\uput[90](-0.555570233019602, 0.831469612302545){$\small 10$}
\pnode(-0.831469612302545, 0.555570233019602){11}
\psdot(-0.831469612302545, 0.555570233019602)
\uput[180](-0.831469612302545, 0.555570233019602){$\small 11$}
\pnode(-0.980785280403230, 0.195090322016128){12}
\psdot(-0.980785280403230, 0.195090322016128)
\uput[180](-0.980785280403230, 0.195090322016128){$\small 12$}
\pnode(-0.980785280403230, -0.195090322016128){13}
\psdot(-0.980785280403230, -0.195090322016128)
\uput[180](-0.980785280403230, -0.195090322016128){$\small 13$}
\pnode(-0.831469612302545, -0.555570233019602){14}
\psdot(-0.831469612302545, -0.555570233019602)
\uput[180](-0.831469612302545, -0.555570233019602){$\small 14$}
\pnode(-0.555570233019602, -0.831469612302545){15}
\psdot(-0.555570233019602, -0.831469612302545)
\uput[180](-0.555570233019602, -0.831469612302545){$\small 15$}
\pnode(-0.195090322016129, -0.980785280403230){16}
\psdot(-0.195090322016129, -0.980785280403230)
\uput[-90](-0.195090322016129, -0.980785280403230){$\small 16$}
\ncline[linecolor=blue,linewidth=.02]{1}{2}
\ncline[linecolor=blue,linewidth=.02]{2}{3}
\ncline[linecolor=blue,linewidth=.02]{3}{4}
\ncline[linecolor=blue,linewidth=.02]{4}{5}
\ncline[linecolor=blue,linewidth=.02]{5}{6}
\ncline[linecolor=blue,linewidth=.02]{6}{7}
\ncline[linecolor=blue,linewidth=.02]{7}{8}
\ncline[linecolor=blue,linewidth=.02]{8}{9}
\ncline[linecolor=blue,linewidth=.02]{9}{10}
\ncline[linecolor=blue,linewidth=.02]{10}{11}
\ncline[linecolor=blue,linewidth=.02]{11}{12}
\ncline[linecolor=blue,linewidth=.02]{12}{13}
\ncline[linecolor=blue,linewidth=.02]{13}{14}
\ncline[linecolor=blue,linewidth=.02]{14}{15}
\ncline[linecolor=blue,linewidth=.02]{15}{16}
\ncline[linecolor=blue,linewidth=.02]{16}{1}
%
\ncline[linecolor=green]{1}{9}
 \ncline[linecolor=red]{16}{8}
\ncline[linecolor=blue,linewidth=.02]{1}{4}
\ncline[linecolor=blue,linewidth=.02]{1}{8}
\ncline[linecolor=blue,linewidth=.02]{5}{8}
\ncline[linecolor=blue,linewidth=.02]{9}{16}
\ncline[linecolor=blue,linewidth=.02]{9}{12}
\ncline[linecolor=blue,linewidth=.02]{13}{16}
  \end{pspicture}}
\rput(2.1,0){%
  \begin{pspicture}(-1,-1)(1,1)
\pnode(0.195090322016128, -0.980785280403230){1}
\psdot(0.195090322016128, -0.980785280403230)
\uput[-90](0.195090322016128, -0.980785280403230){$1$}
\pnode(0.555570233019602, -0.831469612302545){2}
\psdot(0.555570233019602, -0.831469612302545)
\uput[0](0.555570233019602, -0.831469612302545){$2$}
\pnode(0.831469612302545, -0.555570233019602){3}
\psdot(0.831469612302545, -0.555570233019602)
\uput[0](0.831469612302545, -0.555570233019602){$3$}
\pnode(0.980785280403230, -0.195090322016129){4}
\psdot(0.980785280403230, -0.195090322016129)
\uput[0](0.980785280403230, -0.195090322016129){$4$}
\pnode(0.980785280403230, 0.195090322016128){5}
\psdot(0.980785280403230, 0.195090322016128)
\uput[0](0.980785280403230, 0.195090322016128){$5$}
\pnode(0.831469612302545, 0.555570233019602){6}
\psdot(0.831469612302545, 0.555570233019602)
\uput[0](0.831469612302545, 0.555570233019602){$6$}
\pnode(0.555570233019602, 0.831469612302545){7}
\psdot(0.555570233019602, 0.831469612302545)
\uput[90](0.555570233019602, 0.831469612302545){$7$}
\pnode(0.195090322016129, 0.980785280403230){8}
\psdot(0.195090322016129, 0.980785280403230)
\uput[90](0.195090322016129, 0.980785280403230){$8$}
\pnode(-0.195090322016128, 0.980785280403230){9}
\psdot(-0.195090322016128, 0.980785280403230)
\uput[90](-0.195090322016128, 0.980785280403230){$9$}
\pnode(-0.555570233019602, 0.831469612302545){10}
\psdot(-0.555570233019602, 0.831469612302545)
\uput[90](-0.555570233019602, 0.831469612302545){$10$}
\pnode(-0.831469612302545, 0.555570233019602){11}
\psdot(-0.831469612302545, 0.555570233019602)
\uput[180](-0.831469612302545, 0.555570233019602){$11$}
\pnode(-0.980785280403230, 0.195090322016128){12}
\psdot(-0.980785280403230, 0.195090322016128)
\uput[180](-0.980785280403230, 0.195090322016128){$12$}
\pnode(-0.980785280403230, -0.195090322016128){13}
\psdot(-0.980785280403230, -0.195090322016128)
\uput[180](-0.980785280403230, -0.195090322016128){$13$}
\pnode(-0.831469612302545, -0.555570233019602){14}
\psdot(-0.831469612302545, -0.555570233019602)
\uput[180](-0.831469612302545, -0.555570233019602){$14$}
\pnode(-0.555570233019602, -0.831469612302545){15}
\psdot(-0.555570233019602, -0.831469612302545)
\uput[180](-0.555570233019602, -0.831469612302545){$15$}
\pnode(-0.195090322016129, -0.980785280403230){16}
\psdot(-0.195090322016129, -0.980785280403230)
\uput[-90](-0.195090322016129, -0.980785280403230){$16$}
\ncline[linecolor=blue,linewidth=.02]{1}{2}
\ncline[linecolor=blue,linewidth=.02]{2}{3}
\ncline[linecolor=blue,linewidth=.02]{3}{4}
\ncline[linecolor=blue,linewidth=.02]{4}{5}
\ncline[linecolor=blue,linewidth=.02]{5}{6}
\ncline[linecolor=blue,linewidth=.02]{6}{7}
\ncline[linecolor=blue,linewidth=.02]{7}{8}
\ncline[linecolor=blue,linewidth=.02]{8}{9}
\ncline[linecolor=blue,linewidth=.02]{9}{10}
\ncline[linecolor=blue,linewidth=.02]{10}{11}
\ncline[linecolor=blue,linewidth=.02]{11}{12}
\ncline[linecolor=blue,linewidth=.02]{12}{13}
\ncline[linecolor=blue,linewidth=.02]{13}{14}
\ncline[linecolor=blue,linewidth=.02]{14}{15}
\ncline[linecolor=blue,linewidth=.02]{15}{16}
\ncline[linecolor=blue,linewidth=.02]{16}{1}
%
\ncline[linecolor=green]{1}{9}
 \ncline[linecolor=red]{12}{4}
\ncline[linecolor=blue,linewidth=.02]{1}{4}
\ncline[linecolor=blue,linewidth=.02]{4}{9}
\ncline[linecolor=blue,linewidth=.02]{5}{8}
\ncline[linecolor=blue,linewidth=.02]{9}{12}
\ncline[linecolor=blue,linewidth=.02]{13}{16}
\ncline[linecolor=blue,linewidth=.02]{1}{12}
  \end{pspicture}}
  \end{pspicture}
  \caption{Quadrangular dissections of the hexadecagon invariant under rotation.}
  \label{fig:decahexarotinv}
\end{figure}

Notice that if $m = 2k-1$ with $k$-odd, those pairs that consist of two
equal dissections, are also invariant under rotation by
$\pm \frac{\pi}{2}$, see for example Figure~\ref{fig:rotinvsuper20fr}
for a dissection of a dodecagon invariant under $\frac{\pi}{2}$
rotation.

\begin{figure}[ht]
  \centering
    \psset{unit=2.5}
  \begin{pspicture}(-1.3,-1.2)(1.2,1.2)
\pnode(-0.156434465040231, -0.987688340595138){1}
\psdot(-0.156434465040231, -0.987688340595138)
\uput[-90](-0.156434465040231, -0.987688340595138){$\small 1$}
\pnode(0.156434465040231, -0.987688340595138){2}
\psdot(0.156434465040231, -0.987688340595138)
\uput[-90](0.156434465040231, -0.987688340595138){$\small 2$}
\pnode(0.453990499739547, -0.891006524188368){3}
\psdot(0.453990499739547, -0.891006524188368)
\uput[-90](0.453990499739547, -0.891006524188368){$\small 3$}
\pnode(0.707106781186547, -0.707106781186548){4}
\psdot(0.707106781186547, -0.707106781186548)
\uput[-90](0.707106781186547, -0.707106781186548){$\small 4$}
\pnode(0.891006524188368, -0.453990499739547){5}
\psdot(0.891006524188368, -0.453990499739547)
\uput[0](0.891006524188368, -0.453990499739547){$\small 5$}
\pnode(0.987688340595138, -0.156434465040231){6}
\psdot(0.987688340595138, -0.156434465040231)
\uput[0](0.987688340595138, -0.156434465040231){$\small 6$}
\pnode(0.987688340595138, 0.156434465040231){7}
\psdot(0.987688340595138, 0.156434465040231)
\uput[0](0.987688340595138, 0.156434465040231){$\small 7$}
\pnode(0.891006524188368, 0.453990499739547){8}
\psdot(0.891006524188368, 0.453990499739547)
\uput[0](0.891006524188368, 0.453990499739547){$\small 8$}
\pnode(0.707106781186548, 0.707106781186547){9}
\psdot(0.707106781186548, 0.707106781186547)
\uput[90](0.707106781186548, 0.707106781186547){$\small 9$}
\pnode(0.453990499739547, 0.891006524188368){10}
\psdot(0.453990499739547, 0.891006524188368)
\uput[90](0.453990499739547, 0.891006524188368){$\small 10$}
\pnode(0.156434465040231, 0.987688340595138){11}
\psdot(0.156434465040231, 0.987688340595138)
\uput[90](0.156434465040231, 0.987688340595138){$\small 11$}
\pnode(-0.156434465040231, 0.987688340595138){12}
\psdot(-0.156434465040231, 0.987688340595138)
\uput[90](-0.156434465040231, 0.987688340595138){$\small 12$}
\pnode(-0.453990499739547, 0.891006524188368){13}
\psdot(-0.453990499739547, 0.891006524188368)
\uput[90](-0.453990499739547, 0.891006524188368){$\small 13$}
\pnode(-0.707106781186547, 0.707106781186548){14}
\psdot(-0.707106781186547, 0.707106781186548)
\uput[180](-0.707106781186547, 0.707106781186548){$\small 14$}
\pnode(-0.891006524188368, 0.453990499739547){15}
\psdot(-0.891006524188368, 0.453990499739547)
\uput[180](-0.891006524188368, 0.453990499739547){$\small 15$}
\pnode(-0.987688340595138, 0.156434465040231){16}
\psdot(-0.987688340595138, 0.156434465040231)
\uput[180](-0.987688340595138, 0.156434465040231){$\small 16$}
\pnode(-0.987688340595138, -0.156434465040231){17}
\psdot(-0.987688340595138, -0.156434465040231)
\uput[180](-0.987688340595138, -0.156434465040231){$\small 17$}
\pnode(-0.891006524188368, -0.453990499739547){18}
\psdot(-0.891006524188368, -0.453990499739547)
\uput[180](-0.891006524188368, -0.453990499739547){$\small 18$}
\pnode(-0.707106781186548, -0.707106781186547){19}
\psdot(-0.707106781186548, -0.707106781186547)
\uput[-90](-0.707106781186548, -0.707106781186547){$\small 19$}
\pnode(-0.453990499739547, -0.891006524188368){20}
\psdot(-0.453990499739547, -0.891006524188368)
\uput[-90](-0.453990499739547, -0.891006524188368){$\small 20$}
\ncline[linecolor=blue,linewidth=.02]{1}{2}
\ncline[linecolor=blue,linewidth=.02]{2}{3}
\ncline[linecolor=blue,linewidth=.02]{3}{4}
\ncline[linecolor=blue,linewidth=.02]{4}{5}
\ncline[linecolor=blue,linewidth=.02]{5}{6}
\ncline[linecolor=blue,linewidth=.02]{6}{7}
\ncline[linecolor=blue,linewidth=.02]{7}{8}
\ncline[linecolor=blue,linewidth=.02]{8}{9}
\ncline[linecolor=blue,linewidth=.02]{9}{10}
\ncline[linecolor=blue,linewidth=.02]{10}{11}
\ncline[linecolor=blue,linewidth=.02]{11}{12}
\ncline[linecolor=blue,linewidth=.02]{12}{13}
\ncline[linecolor=blue,linewidth=.02]{13}{14}
\ncline[linecolor=blue,linewidth=.02]{14}{15}
\ncline[linecolor=blue,linewidth=.02]{15}{16}
\ncline[linecolor=blue,linewidth=.02]{16}{17}
\ncline[linecolor=blue,linewidth=.02]{17}{18}
\ncline[linecolor=blue,linewidth=.02]{18}{19}
\ncline[linecolor=blue,linewidth=.02]{19}{20}
\ncline[linecolor=blue,linewidth=.02]{20}{1}
\ncline[linecolor=green,linewidth=.01]{1}{11}
 \ncline[linecolor=red,linewidth=.01]{6}{16}
\ncline[linecolor=blue,linewidth=.02]{1}{4}
\ncline[linecolor=blue,linewidth=.02]{1}{6}
\ncline[linecolor=blue,linewidth=.02]{6}{9}
\ncline[linecolor=blue,linewidth=.02]{6}{11}
\ncline[linecolor=blue,linewidth=.02]{11}{14}
\ncline[linecolor=blue,linewidth=.02]{11}{16}
\ncline[linecolor=blue,linewidth=.02]{16}{1}
\ncline[linecolor=blue,linewidth=.02]{16}{19}
  \end{pspicture}
  \caption{A quadrangular dissection of the icosagon invariant under
    rotation by $\frac{\pi}{2}$ radians.}
  \label{fig:rotinvsuper20fr}
\end{figure}
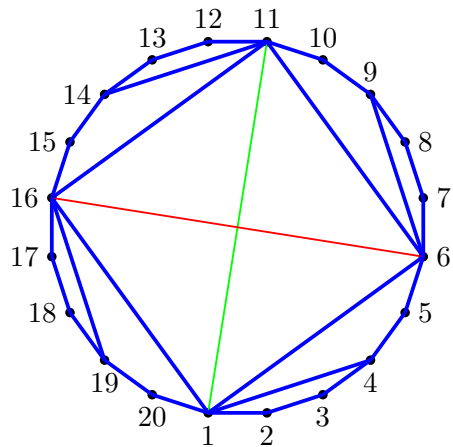

The analysis for reflection invariant dissections is analogous.  There
are two conjugacy classes of reflections in $\mathrm{D}_{2n}$: those whose axis
passes trough two diametrically opposite vertices, and those whose
axis passes through the midpoints of two diametrically opposite
edges. The second conjugacy class is represented by $r$ and
it has been dealt with in Theorem~\ref{thm:sdternformula}.

For a reflection whose axis passes through two diametrically opposite
vertices, we observe that if $m$ is even, there can be no invariant
cell (because there is an even number of them) and so the axis of
reflection is a dissecting diagonal.  The whole dissection then
consists of a dissection of a $(m+2)$-gon glued to its
reflection along an edge.  So there are $\nu_{m+2} = s_{m}$ of
invariant dissections, for each of the $m+1$ diameters.  For example,
Figure~\ref{fig:refl16inv} displays the three dissections of a decagon
that are invariant under reflection across the axis $1\,\,6$.

\begin{figure}[ht]
\psset{unit=1.6}
\begin{pspicture}(-4.3,-1.3)(4.3,1.3)
  \rput(-3,0){
  \begin{pspicture}(-1,-1)(1,1)
\pnode(0.0000000000, 1.0000000000){6}
\psdot(0.0000000000, 1.0000000000)
\uput[90](0.0000000000, 1.0000000000){$\small 6$}
\pnode(-0.5877852523, 0.8090169944){7}
\psdot(-0.5877852523, 0.8090169944)
\uput[90](-0.5877852523, 0.8090169944){$\small 7$}
\pnode(-0.9510565163, 0.3090169944){8}
\psdot(-0.9510565163, 0.3090169944)
\uput[180](-0.9510565163, 0.3090169944){$\small 8$}
\pnode(-0.9510565163, -0.3090169944){9}
\psdot(-0.9510565163, -0.3090169944)
\uput[180](-0.9510565163, -0.3090169944){$\small 9$}
\pnode(-0.5877852523, -0.8090169944){10}
\psdot(-0.5877852523, -0.8090169944)
\uput[-90](-0.5877852523, -0.8090169944){$\small 10$}
\pnode(-0.0000000000, -1.0000000000){1}
\psdot(-0.0000000000, -1.0000000000)
\uput[-90](-0.0000000000, -1.0000000000){$\small 1$}
\pnode(0.5877852523, -0.8090169944){2}
\psdot(0.5877852523, -0.8090169944)
\uput[-90](0.5877852523, -0.8090169944){$\small 2$}
\pnode(0.9510565163, -0.3090169944){3}
\psdot(0.9510565163, -0.3090169944)
\uput[0](0.9510565163, -0.3090169944){$\small 3$}
\pnode(0.9510565163, 0.3090169944){4}
\psdot(0.9510565163, 0.3090169944)
\uput[0](0.9510565163, 0.3090169944){$\small 4$}
\pnode(0.5877852523, 0.8090169944){5}
\psdot(0.5877852523, 0.8090169944)
\uput[90](0.5877852523, 0.8090169944){$\small 5$}
\ncline[linecolor=blue,linewidth=.02]{1}{2}
\ncline[linecolor=blue,linewidth=.02]{2}{3}
\ncline[linecolor=blue,linewidth=.02]{3}{4}
\ncline[linecolor=blue,linewidth=.02]{4}{5}
\ncline[linecolor=blue,linewidth=.02]{5}{6}
\ncline[linecolor=blue,linewidth=.02]{6}{7}
\ncline[linecolor=blue,linewidth=.02]{7}{8}
\ncline[linecolor=blue,linewidth=.02]{8}{9}
\ncline[linecolor=blue,linewidth=.02]{9}{10}
\ncline[linecolor=blue,linewidth=.02]{10}{1}
\ncline[linecolor=blue,linewidth=.02]{1}{6}
\ncline[linecolor=blue,linewidth=.02]{6}{9}
\ncline[linecolor=blue,linewidth=.02]{6}{3}
\end{pspicture}}
\rput(0,0){
\begin{pspicture}(-1,-1)(1,1)
\pnode(0.0000000000, 1.0000000000){6}
\psdot(0.0000000000, 1.0000000000)
\uput[90](0.0000000000, 1.0000000000){$\small 6$}
\pnode(-0.5877852523, 0.8090169944){7}
\psdot(-0.5877852523, 0.8090169944)
\uput[90](-0.5877852523, 0.8090169944){$\small 7$}
\pnode(-0.9510565163, 0.3090169944){8}
\psdot(-0.9510565163, 0.3090169944)
\uput[180](-0.9510565163, 0.3090169944){$\small 8$}
\pnode(-0.9510565163, -0.3090169944){9}
\psdot(-0.9510565163, -0.3090169944)
\uput[180](-0.9510565163, -0.3090169944){$\small 9$}
\pnode(-0.5877852523, -0.8090169944){10}
\psdot(-0.5877852523, -0.8090169944)
\uput[-90](-0.5877852523, -0.8090169944){$\small 10$}
\pnode(-0.0000000000, -1.0000000000){1}
\psdot(-0.0000000000, -1.0000000000)
\uput[-90](-0.0000000000, -1.0000000000){$\small 1$}
\pnode(0.5877852523, -0.8090169944){2}
\psdot(0.5877852523, -0.8090169944)
\uput[-90](0.5877852523, -0.8090169944){$\small 2$}
\pnode(0.9510565163, -0.3090169944){3}
\psdot(0.9510565163, -0.3090169944)
\uput[0](0.9510565163, -0.3090169944){$\small 3$}
\pnode(0.9510565163, 0.3090169944){4}
\psdot(0.9510565163, 0.3090169944)
\uput[0](0.9510565163, 0.3090169944){$\small 4$}
\pnode(0.5877852523, 0.8090169944){5}
\psdot(0.5877852523, 0.8090169944)
\uput[90](0.5877852523, 0.8090169944){$\small 5$}
\ncline[linecolor=blue,linewidth=.02]{1}{2}
\ncline[linecolor=blue,linewidth=.02]{2}{3}
\ncline[linecolor=blue,linewidth=.02]{3}{4}
\ncline[linecolor=blue,linewidth=.02]{4}{5}
\ncline[linecolor=blue,linewidth=.02]{5}{6}
\ncline[linecolor=blue,linewidth=.02]{6}{7}
\ncline[linecolor=blue,linewidth=.02]{7}{8}
\ncline[linecolor=blue,linewidth=.02]{8}{9}
\ncline[linecolor=blue,linewidth=.02]{9}{10}
\ncline[linecolor=blue,linewidth=.02]{10}{1}
\ncline[linecolor=blue,linewidth=.02]{1}{6}
\ncline[linecolor=blue,linewidth=.02]{7}{10}
\ncline[linecolor=blue,linewidth=.02]{2}{5}
\end{pspicture}}
\rput(3,0){
\begin{pspicture}(-1,-1)(1,1)
\pnode(0.0000000000, 1.0000000000){6}
\psdot(0.0000000000, 1.0000000000)
\uput[90](0.0000000000, 1.0000000000){$\small 6$}
\pnode(-0.5877852523, 0.8090169944){7}
\psdot(-0.5877852523, 0.8090169944)
\uput[90](-0.5877852523, 0.8090169944){$\small 7$}
\pnode(-0.9510565163, 0.3090169944){8}
\psdot(-0.9510565163, 0.3090169944)
\uput[180](-0.9510565163, 0.3090169944){$\small 8$}
\pnode(-0.9510565163, -0.3090169944){9}
\psdot(-0.9510565163, -0.3090169944)
\uput[180](-0.9510565163, -0.3090169944){$\small 9$}
\pnode(-0.5877852523, -0.8090169944){10}
\psdot(-0.5877852523, -0.8090169944)
\uput[-90](-0.5877852523, -0.8090169944){$\small 10$}
\pnode(-0.0000000000, -1.0000000000){1}
\psdot(-0.0000000000, -1.0000000000)
\uput[-90](-0.0000000000, -1.0000000000){$\small 1$}
\pnode(0.5877852523, -0.8090169944){2}
\psdot(0.5877852523, -0.8090169944)
\uput[-90](0.5877852523, -0.8090169944){$\small 2$}
\pnode(0.9510565163, -0.3090169944){3}
\psdot(0.9510565163, -0.3090169944)
\uput[0](0.9510565163, -0.3090169944){$\small 3$}
\pnode(0.9510565163, 0.3090169944){4}
\psdot(0.9510565163, 0.3090169944)
\uput[0](0.9510565163, 0.3090169944){$\small 4$}
\pnode(0.5877852523, 0.8090169944){5}
\psdot(0.5877852523, 0.8090169944)
\uput[90](0.5877852523, 0.8090169944){$\small 5$}
\ncline[linecolor=blue,linewidth=.02]{1}{2}
\ncline[linecolor=blue,linewidth=.02]{2}{3}
\ncline[linecolor=blue,linewidth=.02]{3}{4}
\ncline[linecolor=blue,linewidth=.02]{4}{5}
\ncline[linecolor=blue,linewidth=.02]{5}{6}
\ncline[linecolor=blue,linewidth=.02]{6}{7}
\ncline[linecolor=blue,linewidth=.02]{7}{8}
\ncline[linecolor=blue,linewidth=.02]{8}{9}
\ncline[linecolor=blue,linewidth=.02]{9}{10}
\ncline[linecolor=blue,linewidth=.02]{10}{1}
\ncline[linecolor=blue,linewidth=.02]{1}{6}
\ncline[linecolor=blue,linewidth=.02]{1}{8}
\ncline[linecolor=blue,linewidth=.02]{1}{4}
\end{pspicture}}
\end{pspicture}  
  \caption{The three dissections of the decagon invariant under under reflection across $1\,6$.}
  \label{fig:refl16inv}
\end{figure}
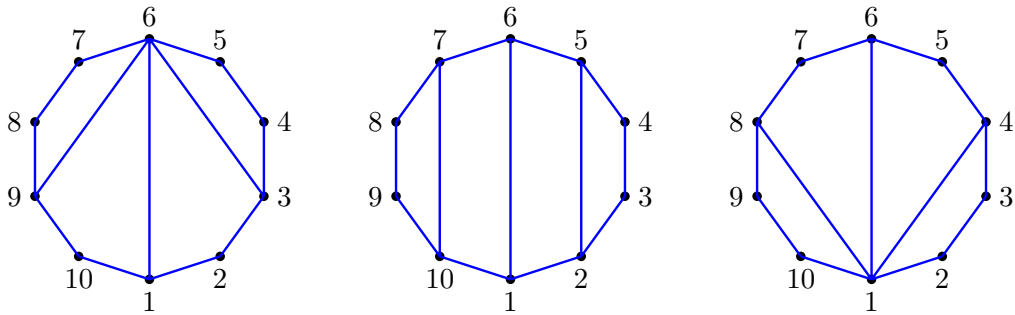

If $m$ is odd, because we have an odd number of cells, the axis of
symmetry cannot be one of the dissecting diagonals, and there has to
be an invariant cell, one of whose diagonals is the axis of symmetry.
That means that at in one of the fixed vertices (say $1$) we have two
(reflections of each other) dissecting chords, and the invariant cell
is completed by an other pair of reflected dissecting chords meeting
at the other vertex. An invariant dissection is then determined by an
ordered pair of dissections, one to the left of the chord $1\,j$ and
the other to the left of the chord $j\, n+1$.  So there are $s_{m}$
such invariant dissections of each of the $m$ axes that pass through
vertices. For example Figure~\ref{fig:refl17inv} shows all the
invariant quadrangular dissections of a dodecagon invariant under
reflection across the axis $1\,\,7$.

\begin{figure}[ht]
  \centering
\psset{unit=1.4}
  \begin{pspicture}(-3,-1.7)(3,9.5)
    \rput(-1.7,2.7){
\begin{pspicture}(-1,-1)(1,1)
\pnode(0.0000000000, 1.0000000000){7}
\psdot(0.0000000000, 1.0000000000)
\uput[90](0.0000000000, 1.0000000000){$\small 7$}
\pnode(-0.5000000000, 0.8660254038){8}
\psdot(-0.5000000000, 0.8660254038)
\uput[90](-0.5000000000, 0.8660254038){$\small 8$}
\pnode(-0.8660254038, 0.5000000000){9}
\psdot(-0.8660254038, 0.5000000000)
\uput[180](-0.8660254038, 0.5000000000){$\small 9$}
\pnode(-1.0000000000, 0.0000000000){10}
\psdot(-1.0000000000, 0.0000000000)
\uput[180](-1.0000000000, 0.0000000000){$\small 10$}
\pnode(-0.8660254038, -0.5000000000){11}
\psdot(-0.8660254038, -0.5000000000)
\uput[180](-0.8660254038, -0.5000000000){$\small 11$}
\pnode(-0.5000000000, -0.8660254038){12}
\psdot(-0.5000000000, -0.8660254038)
\uput[-90](-0.5000000000, -0.8660254038){$\small 12$}
\pnode(-0.0000000000, -1.0000000000){1}
\psdot(-0.0000000000, -1.0000000000)
\uput[-90](-0.0000000000, -1.0000000000){$\small 1$}
\pnode(0.5000000000, -0.8660254038){2}
\psdot(0.5000000000, -0.8660254038)
\uput[-90](0.5000000000, -0.8660254038){$\small 2$}
\pnode(0.8660254038, -0.5000000000){3}
\psdot(0.8660254038, -0.5000000000)
\uput[0](0.8660254038, -0.5000000000){$\small 3$}
\pnode(1.0000000000, -0.0000000000){4}
\psdot(1.0000000000, -0.0000000000)
\uput[0](1.0000000000, -0.0000000000){$\small 4$}
\pnode(0.8660254038, 0.5000000000){5}
\psdot(0.8660254038, 0.5000000000)
\uput[0](0.8660254038, 0.5000000000){$\small 5$}
\pnode(0.5000000000, 0.8660254038){6}
\psdot(0.5000000000, 0.8660254038)
\uput[90](0.5000000000, 0.8660254038){$\small 6$}
\ncline[linecolor=blue,linewidth=.02]{1}{2}
\ncline[linecolor=blue,linewidth=.02]{2}{3}
\ncline[linecolor=blue,linewidth=.02]{3}{4}
\ncline[linecolor=blue,linewidth=.02]{4}{5}
\ncline[linecolor=blue,linewidth=.02]{5}{6}
\ncline[linecolor=blue,linewidth=.02]{6}{7}
\ncline[linecolor=blue,linewidth=.02]{7}{8}
\ncline[linecolor=blue,linewidth=.02]{8}{9}
\ncline[linecolor=blue,linewidth=.02]{9}{10}
\ncline[linecolor=blue,linewidth=.02]{10}{11}
\ncline[linecolor=blue,linewidth=.02]{11}{12}
\ncline[linecolor=blue,linewidth=.02]{12}{1}
 \ncline[linecolor=green,linewidth=.01]{1}{7}
\ncline[linecolor=blue,linewidth=.02]{2}{7}
\ncline[linecolor=blue,linewidth=.02]{7}{12}
\ncline[linecolor=blue,linewidth=.02]{2}{5}
\ncline[linecolor=blue,linewidth=.02]{12}{9}
\end{pspicture}}
\rput(-1.7,5.4){
\begin{pspicture}(-1,-1)(1,1)
\pnode(0.0000000000, 1.0000000000){7}
\psdot(0.0000000000, 1.0000000000)
\uput[90](0.0000000000, 1.0000000000){$\small 7$}
\pnode(-0.5000000000, 0.8660254038){8}
\psdot(-0.5000000000, 0.8660254038)
\uput[90](-0.5000000000, 0.8660254038){$\small 8$}
\pnode(-0.8660254038, 0.5000000000){9}
\psdot(-0.8660254038, 0.5000000000)
\uput[180](-0.8660254038, 0.5000000000){$\small 9$}
\pnode(-1.0000000000, 0.0000000000){10}
\psdot(-1.0000000000, 0.0000000000)
\uput[180](-1.0000000000, 0.0000000000){$\small 10$}
\pnode(-0.8660254038, -0.5000000000){11}
\psdot(-0.8660254038, -0.5000000000)
\uput[180](-0.8660254038, -0.5000000000){$\small 11$}
\pnode(-0.5000000000, -0.8660254038){12}
\psdot(-0.5000000000, -0.8660254038)
\uput[-90](-0.5000000000, -0.8660254038){$\small 12$}
\pnode(-0.0000000000, -1.0000000000){1}
\psdot(-0.0000000000, -1.0000000000)
\uput[-90](-0.0000000000, -1.0000000000){$\small 1$}
\pnode(0.5000000000, -0.8660254038){2}
\psdot(0.5000000000, -0.8660254038)
\uput[-90](0.5000000000, -0.8660254038){$\small 2$}
\pnode(0.8660254038, -0.5000000000){3}
\psdot(0.8660254038, -0.5000000000)
\uput[0](0.8660254038, -0.5000000000){$\small 3$}
\pnode(1.0000000000, -0.0000000000){4}
\psdot(1.0000000000, -0.0000000000)
\uput[0](1.0000000000, -0.0000000000){$\small 4$}
\pnode(0.8660254038, 0.5000000000){5}
\psdot(0.8660254038, 0.5000000000)
\uput[0](0.8660254038, 0.5000000000){$\small 5$}
\pnode(0.5000000000, 0.8660254038){6}
\psdot(0.5000000000, 0.8660254038)
\uput[90](0.5000000000, 0.8660254038){$\small 6$}
\ncline[linecolor=blue,linewidth=.02]{1}{2}
\ncline[linecolor=blue,linewidth=.02]{2}{3}
\ncline[linecolor=blue,linewidth=.02]{3}{4}
\ncline[linecolor=blue,linewidth=.02]{4}{5}
\ncline[linecolor=blue,linewidth=.02]{5}{6}
\ncline[linecolor=blue,linewidth=.02]{6}{7}
\ncline[linecolor=blue,linewidth=.02]{7}{8}
\ncline[linecolor=blue,linewidth=.02]{8}{9}
\ncline[linecolor=blue,linewidth=.02]{9}{10}
\ncline[linecolor=blue,linewidth=.02]{10}{11}
\ncline[linecolor=blue,linewidth=.02]{11}{12}
\ncline[linecolor=blue,linewidth=.02]{12}{1}
 \ncline[linecolor=green,linewidth=.01]{1}{7}
\ncline[linecolor=blue,linewidth=.02]{2}{7}
\ncline[linecolor=blue,linewidth=.02]{7}{12}
\ncline[linecolor=blue,linewidth=.02]{3}{6}
\ncline[linecolor=blue,linewidth=.02]{11}{8}
\end{pspicture}}
\rput(-1.7,0){
\begin{pspicture}(-1,-1)(1,1)
\pnode(0.0000000000, 1.0000000000){7}
\psdot(0.0000000000, 1.0000000000)
\uput[90](0.0000000000, 1.0000000000){$\small 7$}
\pnode(-0.5000000000, 0.8660254038){8}
\psdot(-0.5000000000, 0.8660254038)
\uput[90](-0.5000000000, 0.8660254038){$\small 8$}
\pnode(-0.8660254038, 0.5000000000){9}
\psdot(-0.8660254038, 0.5000000000)
\uput[180](-0.8660254038, 0.5000000000){$\small 9$}
\pnode(-1.0000000000, 0.0000000000){10}
\psdot(-1.0000000000, 0.0000000000)
\uput[180](-1.0000000000, 0.0000000000){$\small 10$}
\pnode(-0.8660254038, -0.5000000000){11}
\psdot(-0.8660254038, -0.5000000000)
\uput[180](-0.8660254038, -0.5000000000){$\small 11$}
\pnode(-0.5000000000, -0.8660254038){12}
\psdot(-0.5000000000, -0.8660254038)
\uput[-90](-0.5000000000, -0.8660254038){$\small 12$}
\pnode(-0.0000000000, -1.0000000000){1}
\psdot(-0.0000000000, -1.0000000000)
\uput[-90](-0.0000000000, -1.0000000000){$\small 1$}
\pnode(0.5000000000, -0.8660254038){2}
\psdot(0.5000000000, -0.8660254038)
\uput[-90](0.5000000000, -0.8660254038){$\small 2$}
\pnode(0.8660254038, -0.5000000000){3}
\psdot(0.8660254038, -0.5000000000)
\uput[0](0.8660254038, -0.5000000000){$\small 3$}
\pnode(1.0000000000, -0.0000000000){4}
\psdot(1.0000000000, -0.0000000000)
\uput[0](1.0000000000, -0.0000000000){$\small 4$}
\pnode(0.8660254038, 0.5000000000){5}
\psdot(0.8660254038, 0.5000000000)
\uput[0](0.8660254038, 0.5000000000){$\small 5$}
\pnode(0.5000000000, 0.8660254038){6}
\psdot(0.5000000000, 0.8660254038)
\uput[90](0.5000000000, 0.8660254038){$\small 6$}
\ncline[linecolor=blue,linewidth=.02]{1}{2}
\ncline[linecolor=blue,linewidth=.02]{2}{3}
\ncline[linecolor=blue,linewidth=.02]{3}{4}
\ncline[linecolor=blue,linewidth=.02]{4}{5}
\ncline[linecolor=blue,linewidth=.02]{5}{6}
\ncline[linecolor=blue,linewidth=.02]{6}{7}
\ncline[linecolor=blue,linewidth=.02]{7}{8}
\ncline[linecolor=blue,linewidth=.02]{8}{9}
\ncline[linecolor=blue,linewidth=.02]{9}{10}
\ncline[linecolor=blue,linewidth=.02]{10}{11}
\ncline[linecolor=blue,linewidth=.02]{11}{12}
\ncline[linecolor=blue,linewidth=.02]{12}{1}
 \ncline[linecolor=green,linewidth=.01]{1}{7}
\ncline[linecolor=blue,linewidth=.02]{2}{7}
\ncline[linecolor=blue,linewidth=.02]{7}{12}
\ncline[linecolor=blue,linewidth=.02]{7}{4}
\ncline[linecolor=blue,linewidth=.02]{7}{10}
\end{pspicture}}
\rput(1.7,2.7){
\begin{pspicture}(-1,-1)(1,1)
\pnode(0.0000000000, 1.0000000000){7}
\psdot(0.0000000000, 1.0000000000)
\uput[90](0.0000000000, 1.0000000000){$\small 7$}
\pnode(-0.5000000000, 0.8660254038){8}
\psdot(-0.5000000000, 0.8660254038)
\uput[90](-0.5000000000, 0.8660254038){$\small 8$}
\pnode(-0.8660254038, 0.5000000000){9}
\psdot(-0.8660254038, 0.5000000000)
\uput[180](-0.8660254038, 0.5000000000){$\small 9$}
\pnode(-1.0000000000, 0.0000000000){10}
\psdot(-1.0000000000, 0.0000000000)
\uput[180](-1.0000000000, 0.0000000000){$\small 10$}
\pnode(-0.8660254038, -0.5000000000){11}
\psdot(-0.8660254038, -0.5000000000)
\uput[180](-0.8660254038, -0.5000000000){$\small 11$}
\pnode(-0.5000000000, -0.8660254038){12}
\psdot(-0.5000000000, -0.8660254038)
\uput[-90](-0.5000000000, -0.8660254038){$\small 12$}
\pnode(-0.0000000000, -1.0000000000){1}
\psdot(-0.0000000000, -1.0000000000)
\uput[-90](-0.0000000000, -1.0000000000){$\small 1$}
\pnode(0.5000000000, -0.8660254038){2}
\psdot(0.5000000000, -0.8660254038)
\uput[-90](0.5000000000, -0.8660254038){$\small 2$}
\pnode(0.8660254038, -0.5000000000){3}
\psdot(0.8660254038, -0.5000000000)
\uput[0](0.8660254038, -0.5000000000){$\small 3$}
\pnode(1.0000000000, -0.0000000000){4}
\psdot(1.0000000000, -0.0000000000)
\uput[0](1.0000000000, -0.0000000000){$\small 4$}
\pnode(0.8660254038, 0.5000000000){5}
\psdot(0.8660254038, 0.5000000000)
\uput[0](0.8660254038, 0.5000000000){$\small 5$}
\pnode(0.5000000000, 0.8660254038){6}
\psdot(0.5000000000, 0.8660254038)
\uput[90](0.5000000000, 0.8660254038){$\small 6$}
\ncline[linecolor=blue,linewidth=.02]{1}{2}
\ncline[linecolor=blue,linewidth=.02]{2}{3}
\ncline[linecolor=blue,linewidth=.02]{3}{4}
\ncline[linecolor=blue,linewidth=.02]{4}{5}
\ncline[linecolor=blue,linewidth=.02]{5}{6}
\ncline[linecolor=blue,linewidth=.02]{6}{7}
\ncline[linecolor=blue,linewidth=.02]{7}{8}
\ncline[linecolor=blue,linewidth=.02]{8}{9}
\ncline[linecolor=blue,linewidth=.02]{9}{10}
\ncline[linecolor=blue,linewidth=.02]{10}{11}
\ncline[linecolor=blue,linewidth=.02]{11}{12}
\ncline[linecolor=blue,linewidth=.02]{12}{1}
 \ncline[linecolor=green,linewidth=.01]{1}{7}
\ncline[linecolor=blue,linewidth=.02]{1}{6}
\ncline[linecolor=blue,linewidth=.02]{1}{8}
\ncline[linecolor=blue,linewidth=.02]{8}{11}
\ncline[linecolor=blue,linewidth=.02]{3}{6}
\end{pspicture}}
\rput(1.7,0){
\begin{pspicture}(-1,-1)(1,1)
\pnode(0.0000000000, 1.0000000000){7}
\psdot(0.0000000000, 1.0000000000)
\uput[90](0.0000000000, 1.0000000000){$\small 7$}
\pnode(-0.5000000000, 0.8660254038){8}
\psdot(-0.5000000000, 0.8660254038)
\uput[90](-0.5000000000, 0.8660254038){$\small 8$}
\pnode(-0.8660254038, 0.5000000000){9}
\psdot(-0.8660254038, 0.5000000000)
\uput[180](-0.8660254038, 0.5000000000){$\small 9$}
\pnode(-1.0000000000, 0.0000000000){10}
\psdot(-1.0000000000, 0.0000000000)
\uput[180](-1.0000000000, 0.0000000000){$\small 10$}
\pnode(-0.8660254038, -0.5000000000){11}
\psdot(-0.8660254038, -0.5000000000)
\uput[180](-0.8660254038, -0.5000000000){$\small 11$}
\pnode(-0.5000000000, -0.8660254038){12}
\psdot(-0.5000000000, -0.8660254038)
\uput[-90](-0.5000000000, -0.8660254038){$\small 12$}
\pnode(-0.0000000000, -1.0000000000){1}
\psdot(-0.0000000000, -1.0000000000)
\uput[-90](-0.0000000000, -1.0000000000){$\small 1$}
\pnode(0.5000000000, -0.8660254038){2}
\psdot(0.5000000000, -0.8660254038)
\uput[-90](0.5000000000, -0.8660254038){$\small 2$}
\pnode(0.8660254038, -0.5000000000){3}
\psdot(0.8660254038, -0.5000000000)
\uput[0](0.8660254038, -0.5000000000){$\small 3$}
\pnode(1.0000000000, -0.0000000000){4}
\psdot(1.0000000000, -0.0000000000)
\uput[0](1.0000000000, -0.0000000000){$\small 4$}
\pnode(0.8660254038, 0.5000000000){5}
\psdot(0.8660254038, 0.5000000000)
\uput[0](0.8660254038, 0.5000000000){$\small 5$}
\pnode(0.5000000000, 0.8660254038){6}
\psdot(0.5000000000, 0.8660254038)
\uput[90](0.5000000000, 0.8660254038){$\small 6$}
\ncline[linecolor=blue,linewidth=.02]{1}{2}
\ncline[linecolor=blue,linewidth=.02]{2}{3}
\ncline[linecolor=blue,linewidth=.02]{3}{4}
\ncline[linecolor=blue,linewidth=.02]{4}{5}
\ncline[linecolor=blue,linewidth=.02]{5}{6}
\ncline[linecolor=blue,linewidth=.02]{6}{7}
\ncline[linecolor=blue,linewidth=.02]{7}{8}
\ncline[linecolor=blue,linewidth=.02]{8}{9}
\ncline[linecolor=blue,linewidth=.02]{9}{10}
\ncline[linecolor=blue,linewidth=.02]{10}{11}
\ncline[linecolor=blue,linewidth=.02]{11}{12}
\ncline[linecolor=blue,linewidth=.02]{12}{1}
 \ncline[linecolor=green,linewidth=.01]{1}{7}
\ncline[linecolor=blue,linewidth=.02]{1}{6}
\ncline[linecolor=blue,linewidth=.02]{1}{8}
\ncline[linecolor=blue,linewidth=.02]{1}{4}
\ncline[linecolor=blue,linewidth=.02]{10}{1}
\end{pspicture}}
\rput(1.7,5.4){
\begin{pspicture}(-1,-1)(1,1)
\pnode(0.0000000000, 1.0000000000){7}
\psdot(0.0000000000, 1.0000000000)
\uput[90](0.0000000000, 1.0000000000){$\small 7$}
\pnode(-0.5000000000, 0.8660254038){8}
\psdot(-0.5000000000, 0.8660254038)
\uput[90](-0.5000000000, 0.8660254038){$\small 8$}
\pnode(-0.8660254038, 0.5000000000){9}
\psdot(-0.8660254038, 0.5000000000)
\uput[180](-0.8660254038, 0.5000000000){$\small 9$}
\pnode(-1.0000000000, 0.0000000000){10}
\psdot(-1.0000000000, 0.0000000000)
\uput[180](-1.0000000000, 0.0000000000){$\small 10$}
\pnode(-0.8660254038, -0.5000000000){11}
\psdot(-0.8660254038, -0.5000000000)
\uput[180](-0.8660254038, -0.5000000000){$\small 11$}
\pnode(-0.5000000000, -0.8660254038){12}
\psdot(-0.5000000000, -0.8660254038)
\uput[-90](-0.5000000000, -0.8660254038){$\small 12$}
\pnode(-0.0000000000, -1.0000000000){1}
\psdot(-0.0000000000, -1.0000000000)
\uput[-90](-0.0000000000, -1.0000000000){$\small 1$}
\pnode(0.5000000000, -0.8660254038){2}
\psdot(0.5000000000, -0.8660254038)
\uput[-90](0.5000000000, -0.8660254038){$\small 2$}
\pnode(0.8660254038, -0.5000000000){3}
\psdot(0.8660254038, -0.5000000000)
\uput[0](0.8660254038, -0.5000000000){$\small 3$}
\pnode(1.0000000000, -0.0000000000){4}
\psdot(1.0000000000, -0.0000000000)
\uput[0](1.0000000000, -0.0000000000){$\small 4$}
\pnode(0.8660254038, 0.5000000000){5}
\psdot(0.8660254038, 0.5000000000)
\uput[0](0.8660254038, 0.5000000000){$\small 5$}
\pnode(0.5000000000, 0.8660254038){6}
\psdot(0.5000000000, 0.8660254038)
\uput[90](0.5000000000, 0.8660254038){$\small 6$}
\ncline[linecolor=blue,linewidth=.02]{1}{2}
\ncline[linecolor=blue,linewidth=.02]{2}{3}
\ncline[linecolor=blue,linewidth=.02]{3}{4}
\ncline[linecolor=blue,linewidth=.02]{4}{5}
\ncline[linecolor=blue,linewidth=.02]{5}{6}
\ncline[linecolor=blue,linewidth=.02]{6}{7}
\ncline[linecolor=blue,linewidth=.02]{7}{8}
\ncline[linecolor=blue,linewidth=.02]{8}{9}
\ncline[linecolor=blue,linewidth=.02]{9}{10}
\ncline[linecolor=blue,linewidth=.02]{10}{11}
\ncline[linecolor=blue,linewidth=.02]{11}{12}
\ncline[linecolor=blue,linewidth=.02]{12}{1}
 \ncline[linecolor=green,linewidth=.01]{1}{7}
\ncline[linecolor=blue,linewidth=.02]{1}{6}
\ncline[linecolor=blue,linewidth=.02]{1}{8}
\ncline[linecolor=blue,linewidth=.02]{2}{5}
\ncline[linecolor=blue,linewidth=.02]{12}{9}
\end{pspicture}}
\rput(0,8.1){
\begin{pspicture}(-1,-1)(1,1)
\pnode(0.0000000000, 1.0000000000){7}
\psdot(0.0000000000, 1.0000000000)
\uput[90](0.0000000000, 1.0000000000){$\small 7$}
\pnode(-0.5000000000, 0.8660254038){8}
\psdot(-0.5000000000, 0.8660254038)
\uput[90](-0.5000000000, 0.8660254038){$\small 8$}
\pnode(-0.8660254038, 0.5000000000){9}
\psdot(-0.8660254038, 0.5000000000)
\uput[180](-0.8660254038, 0.5000000000){$\small 9$}
\pnode(-1.0000000000, 0.0000000000){10}
\psdot(-1.0000000000, 0.0000000000)
\uput[180](-1.0000000000, 0.0000000000){$\small 10$}
\pnode(-0.8660254038, -0.5000000000){11}
\psdot(-0.8660254038, -0.5000000000)
\uput[180](-0.8660254038, -0.5000000000){$\small 11$}
\pnode(-0.5000000000, -0.8660254038){12}
\psdot(-0.5000000000, -0.8660254038)
\uput[-90](-0.5000000000, -0.8660254038){$\small 12$}
\pnode(-0.0000000000, -1.0000000000){1}
\psdot(-0.0000000000, -1.0000000000)
\uput[-90](-0.0000000000, -1.0000000000){$\small 1$}
\pnode(0.5000000000, -0.8660254038){2}
\psdot(0.5000000000, -0.8660254038)
\uput[-90](0.5000000000, -0.8660254038){$\small 2$}
\pnode(0.8660254038, -0.5000000000){3}
\psdot(0.8660254038, -0.5000000000)
\uput[0](0.8660254038, -0.5000000000){$\small 3$}
\pnode(1.0000000000, -0.0000000000){4}
\psdot(1.0000000000, -0.0000000000)
\uput[0](1.0000000000, -0.0000000000){$\small 4$}
\pnode(0.8660254038, 0.5000000000){5}
\psdot(0.8660254038, 0.5000000000)
\uput[0](0.8660254038, 0.5000000000){$\small 5$}
\pnode(0.5000000000, 0.8660254038){6}
\psdot(0.5000000000, 0.8660254038)
\uput[90](0.5000000000, 0.8660254038){$\small 6$}
\ncline[linecolor=blue,linewidth=.02]{1}{2}
\ncline[linecolor=blue,linewidth=.02]{2}{3}
\ncline[linecolor=blue,linewidth=.02]{3}{4}
\ncline[linecolor=blue,linewidth=.02]{4}{5}
\ncline[linecolor=blue,linewidth=.02]{5}{6}
\ncline[linecolor=blue,linewidth=.02]{6}{7}
\ncline[linecolor=blue,linewidth=.02]{7}{8}
\ncline[linecolor=blue,linewidth=.02]{8}{9}
\ncline[linecolor=blue,linewidth=.02]{9}{10}
\ncline[linecolor=blue,linewidth=.02]{10}{11}
\ncline[linecolor=blue,linewidth=.02]{11}{12}
\ncline[linecolor=blue,linewidth=.02]{12}{1}
 \ncline[linecolor=green,linewidth=.01]{1}{7}
\ncline[linecolor=blue,linewidth=.02]{1}{4}
\ncline[linecolor=blue,linewidth=.02]{1}{10}
\ncline[linecolor=blue,linewidth=.02]{7}{4}
\ncline[linecolor=blue,linewidth=.02]{7}{10}
\end{pspicture}}
  \end{pspicture}
  \caption{Dissections of the dodecagon invariant under reflection across $1\,7$.}
  \label{fig:refl17inv}
\end{figure}
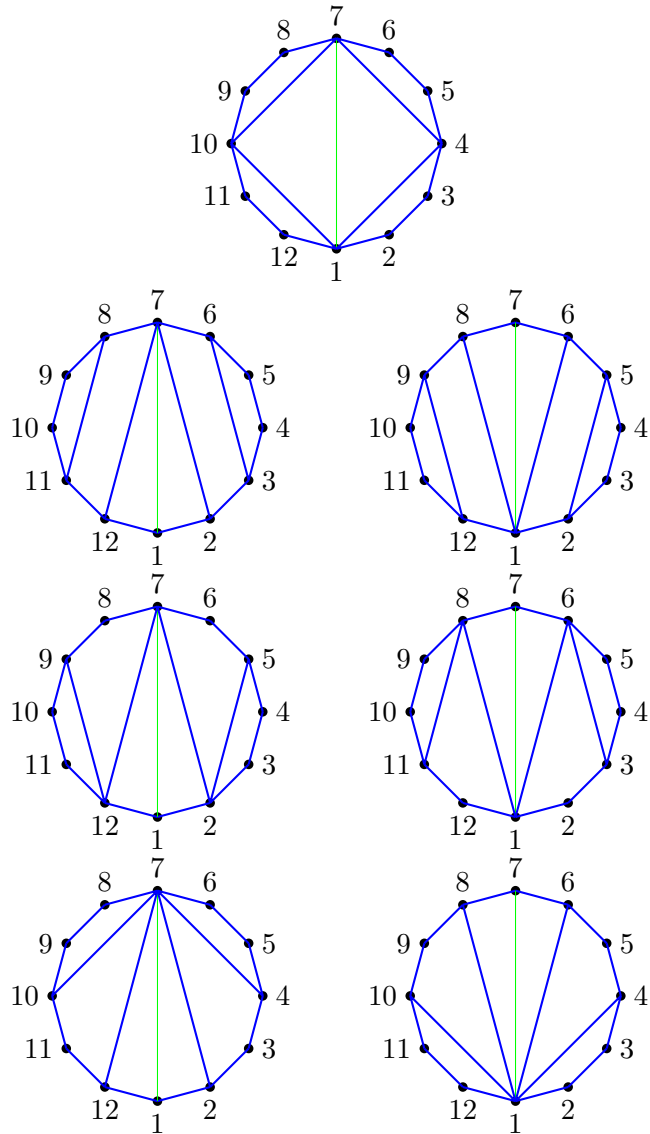

\end{proof}

Let $\mathcal{Q}'_{2n}$ be the set of \emph{unlabeled oriented}
quadrangular dissections of the $2n$-gon, that is quadrangular
dissections up to rotation, and $\widetilde{\mathcal{Q}}_{2n}$ the set
of \emph{unlabeled unoriented} quadrangular dissections that is,
quadrangular dissections up to rotations and reflections.  The number
of such dissections $q'_{2n}$ ($\tilde{q}_{2n}$ respectively), is 
sequence \href{https://oeis.org/A005034}{A005034}
(\href{https://oeis.org/A005036}{A005036} respectively) in the Online
Encyclopedia of Integer Sequences~\cite{oeis}.  Using
Theorem~\ref{thm:fp} and Burnside's lemma we obtain the following
explicit formulas:

\begin{thm}
  \label{thm:qdiss} 
  The number of quadrilateral dissections of a $2n$-gon up to rotations is
  $$ q_{2n}^{'} =
  \begin{cases}
    \dfrac{1}{2n}\left( \nu_{n-1} + n s_{n-1} \right) & \text{ if } n \equiv 1 \pmod{2}\\[10pt]
    \dfrac{1}{4n} \left(\nu_{n-1} + \dfrac{n}{2} s_{n-1}  \right) & \text{ if } n \equiv 0 \pmod{4}\\[10pt]
    \frac{1}{4n} \left( \nu_{n-1} + \dfrac{n}{2} s_{n-1} + n s_{\frac{n-2}{2}} \right) & \text{ if }  n \equiv 2 \pmod{4}.
  \end{cases}
$$
The number of quadrilateral dissections of a $2n$-gon up
  to rotations and reflections is
  $$ \tilde{q}_{2n} =
  \begin{cases}
    \dfrac{1}{4n}\left( \nu_{n-1} + 3n s_{n-1} \right) & \text{ if } n \equiv 1 \pmod{2}\\[10pt]
    \dfrac{1}{4n} \left(\nu_{n-1} + \dfrac{5n}{2} s_{n-1}  \right) & \text{ if } n \equiv 0 \pmod{4}\\[10pt]
    \dfrac{1}{4n} \left( \nu_{n-1} + \dfrac{5n}{2} s_{n-1} + n s_{\frac{n-2}{2}} \right) & \text{if }  n \equiv 2 \pmod{4}.
  \end{cases}
$$
\end{thm}

Recall that $\mathcal{N}_m^{'}$ stands for the set of oriented
unlabeled non-crossing trees with $m$ edges, in other words an element
of $\mathcal{N}_n^{'}$ is an orbit of the action of
$\left\langle \kappa^2 \right\rangle \cong \mathbb{Z}/n$, and
$\widetilde{\mathcal{N}}_n$ stands for the set of unoriented unlabeled
non-crossing trees, in other words an element of
$\widetilde{\mathcal{N}}_n$ is an orbit of the action of the dihedral
group $\mathrm{D}_n = \left\langle \kappa^2, s \right\rangle$.  So
Theorem~\ref{thm:fp} allow us to calculate the number of unlabeled
oriented and unoriented non-crossing trees as well.  Note that for
$m$ odd the central element of $D_{2n}$ (rotation by $\pi$) belongs
to $\mathrm{D}_n$ while for even $m$ it doesn't, so that for odd $m$ there are
no rotation invariant non-crossing trees.  So we have the following
theorem, proved in~\cite{Noy1998301}\footnote{The enumeration of
  oriented unlabeled trees is not explicitly stated there but a
  formula can be deduced from the calculations.}.

\begin{thm}[Noy]
  \label{thm:unc} 
The number of non-crossing trees with $n$ vertices up to rotations is
  $$ \nu_m^{'} =
  \begin{cases}
    \dfrac{\nu_m}{2(m+1)} & \text{ if $m$ is even} \\[10pt]
    \dfrac{1}{2(m+1)} \left(\nu_m + \dfrac{(m+1) s_m}{2} \right) & \text{ if $m$ is odd.}
  \end{cases}
$$
The number of unlabeled non-crossing trees with $n$ vertices is
  $$ \tilde{\nu}_m =
  \begin{cases}
    \dfrac{1}{2(m+1)} \left( \nu_m + (m+1) s_m  \right) & \text{ if $m$ is even} \\[10pt]
    \dfrac{1}{2(m+1)} \left( \nu_m + \dfrac{3(m+1)}{2} s_m \right) & \text{ if $m$ is odd.}
  \end{cases}
$$
\end{thm}

Finally we can use a generalization of Burnside's Lemma, the "Counting
Lemma" of~\cite{Robinson1981}, to count the number of self-dual
unlabeled oriented or unoriented trees.
\begin{lem}[Robinson's Counting Lemma]
\label{lem:robinson}
Let $G$ be a group acting on a set $X$ endowed with a permutation $r$
such that $rG = Gr$, so that $r$ is well defined in the orbits of
$G$. Then $N(G,r)$ the number of orbits fixed by $r$ is given by:
$$
N(G,r) = \frac{1}{\left| G \right|} \sum_{g\in G} \left| \left\{ x \in
    G : g r x = x \right\} \right|.
$$
\end{lem}

Applying this theorem in our case with
$G = \left\langle \kappa^2 \right\rangle\cong \mathbb{Z}/n$ and
$G = \left\langle \kappa^{2}, s \right\rangle \cong D_n$ we
have:

\begin{thm}
  The number of self-dual unlabeled oriented non-crossing trees with
  $m$ edges is
  $$ s^{'}_m = s_m.$$
  The number of self-dual unlabeled unoriented non-crossing trees 
  with $m$ edges is
  $$ \tilde{s}_m =
  \begin{cases}
    s_m & \text{ if } m \equiv 0 \pmod{2} \\[10pt]
    \dfrac{s_m + s_{\frac{m+1}{2}}}{2} & \text{ if } m \equiv 1 \pmod{4} \\[10pt]
    \dfrac{s_m}{2} & \text{ if } m \equiv 3 \pmod{4}.
  \end{cases}
$$
\end{thm}
\begin{proof}
  For the case of oriented non-crossing trees we need to look at the
  fixed points of $\kappa^{2i} r$ for $i=0,\ldots,n-1$.  Each of these
  elements is a reflection in $D_{2n}$ and so by Theorem~\ref{thm:fp}
  has $s_n$ fixed element.

  For the case of unoriented unlabeled non-crossing trees, we need in
  addition to take into account the fixed points of $\kappa^{2i} s\,r$
  for $i=0,\ldots,n-1$.  Since $s\,r = \kappa^{-1}$ this means that we
  have to take into account all the fixed points of odd powers of
  $\kappa$, so the result follows from Theorem~\ref{thm:fp}.
\end{proof}

It is also of interest to consider \emph{anti-self-dual} non-crossing
trees, that is non-crossing-trees $t\in \mathcal{N}_m'$ that satisfy
$t^{*} = \bar{t}$.  Since $rs = \kappa$ an anti-self-dual non-crossing tree is
fixed by $\kappa^i$ for some odd power $i$.  So we have:
\begin{thm}
  \label{thm:asdnc}
  The number of anti-self-dual non-crossing trees with $m$ vertices is
  $$
a_m =  \begin{cases}
    s_m & \text{ if } m \equiv 0 \pmod{2} \\[10pt]
    s_{\frac{m-3}{2}} & \text{ if } m \equiv 1 \pmod{4} \\[10pt]
    0 & \text{ if } m \equiv 3 \pmod{4}.
  \end{cases}
  $$
\end{thm}

\section{Future directions}
\label{sec:future}

A bijection $\phi$ can be defined more generally between appropriately
defined \emph{quadrangular dissections} of any surface with boundary and
graphs pegged in that surface.  Of particular interest is the case where
the surface is the annulus, in which case by a simple Euler characteristic
argument one sees that the graphs have to be unicyclic, and we plan to explore
that direction in a future project.

The $p=3$ Fuss-Catalan numbers appear in the theory of
non-crossing partitions not only as the number of maximal chains ``up
to commutation'' but also as the number of $2$-multichains, that is
they count the number of pairs $\pi_1 \le \pi_2$ of non-crossing
partitions.  More generally, $p$ Fuss-Catalan numbers count the number of
$p$-multichains, that is $p$-tuples
$\pi_1\le \pi_2 \le \cdots \le \pi_{p}$ (see~\cite{Armstrong2009}).
We plan to explore the relation between mind-body duality and the
Kreweras complement in that connection.

A further interesting line of future research in connection with
non-crossing partitions is to generalize the results of the current
work to all finite Coxeter groups.

Finally, using the connected sum of pegs defined in Section~4.2
of~\cite{Apos2018arXivApril}, one can ``glue'' a non-crossing tree and
its dual along their common boundary to obtain a self-dual map on the
sphere.  Not all self-dual maps are obtained with that construction
since the resulting graph will have no loops pendant edges.
Furthermore, since the first factors of a pair of dual factorizations
agree the resulting graph has at least on bigon.  One can rectify that
by contracting that bigon into an edge and deleting its dual degree
$2$ vertex to obtain a rooted self-dual map of the sphere that does
not necessarily contain bigons.  The following question then seems
interesting and open:

\begin{ques}
  Do all rooted self-dual maps of the sphere without loops arise from
  gluing together a pair of dual non-crossing trees?  If not
  characterize those that do.
\end{ques}

\bibliographystyle{plain}
\bibliography{brcov.bib}

\end{document}
